\documentclass[reqno,10pt]{amsart}

\usepackage{amsmath}
\usepackage{amsfonts,amsthm,amssymb}
\usepackage{indentfirst}
\usepackage{graphicx}
\usepackage{graphics}
\usepackage{pict2e}
\usepackage{epic}
\numberwithin{equation}{section}
\usepackage[margin=1.7cm]{geometry}
\usepackage{epstopdf} 
\usepackage[colorlinks,linkcolor=blue]{hyperref}
\usepackage[capitalise,noabbrev]{cleveref}
\usepackage{todonotes}

\makeatletter
\def\thanks#1{\protected@xdef\@thanks{\@thanks
		\protect\footnotetext{#1}}}
\makeatother

\crefformat{equation}{(#2#1#3)}
\crefmultiformat{equation}{(#2#1#3)}{ and~(#2#1#3)}{, (#2#1#3)}{ and~(#2#1#3)}
\crefrangeformat{equation}{(#3#1#4) to~(#5#2#6)}

\usepackage[normalem]{ulem}

\allowdisplaybreaks

\theoremstyle{plain}
\newtheorem{Thm}{Theorem}[section]
\newtheorem*{Thm*}{Theorem}
\newtheorem{Lem}[Thm]{Lemma}

\newtheorem{Prop}[Thm]{Proposition}

\theoremstyle{definition}

\newtheorem{Rem}[Thm]{Remark}
\newtheorem{?}[Thm]{Problem}

\newcommand{\p}{\partial}
\newcommand{\pt}{\partial_t}
\newcommand{\px}{\partial_1}
\newcommand{\R}{\mathbb{R}}
\newcommand{\T}{\theta}
\newcommand{\e}{\varepsilon}
\newcommand{\E}{\mathcal{E}}

\newcommand{\D}{\mathbf{D}}
\newcommand{\Do}{\mathbf{D}_0}
\newcommand{\Dn}{\mathbf{D}_{\neq}}
\newcommand{\Et}{\tilde{\mathcal{E}}}

\newcommand{\Eb}{\bar{\mathcal{E}}}
\newcommand{\Em}{\mathring{\mathcal{E}}}

\newcommand{\Tt}{\tilde{\theta}}
\newcommand{\Tb}{\bar{\theta}}
\newcommand{\etab}{\bar{\eta}}

\newcommand{\zetam}{\mathring{\zeta}}
\newcommand{\phim}{\mathring{\phi}}
\newcommand{\psim}{\mathring{\psi}}

\newcommand{\varphim}{\mathring{\varphi}}

\newcommand{\mr}{\mathring}

\newcommand{\uf}{\textbf{u}}

\newcommand{\Torus}{\mathbb{T}}

\newcommand{\dv}{\text{div}}
\newcommand{\lap}{\triangle}

\newcommand{\rhob}{\bar{\rho}}
\newcommand{\rhot}{\tilde{\rho}}

\newcommand{\rhom}{\mathring{\rho}}

\newcommand{\ub}{\bar{u}}

\newcommand{\ut}{\tilde{u}}

\newcommand{\um}{\mathring{u}}

\newcommand{\mb}{\bar{m}}
\newcommand{\mt}{\tilde{m}}

\newcommand{\mm}{\mathring{m}}

\newcommand{\ubb}{\mathbf{\bar{u}}}

\newcommand{\utb}{\mathbf{\tilde{u}}}

\newcommand{\mtb}{\mathbf{\tilde{m}}}

\newcommand{\thetat}{\tilde{\theta}}
\newcommand{\thetab}{\bar{\theta}}
\newcommand{\Thetab}{\bar{\Theta}}
\newcommand{\deltab}{\bar{\delta}}

\newcommand{\thetam}{\mathring{\theta}}

\newcommand{\m}{\textbf{m}}
\newcommand{\abs}[1]{\left\lvert#1\right\rvert}

\newcommand{\norm}[1]{\left\lVert#1\right\rVert}

\usepackage{color}

\usepackage{todonotes}

\begin{document}
	
	\title[Stability of planar entropy wave for 3-d Navier-Stokes in Eulerian coordinates]{Asymptotic stability of planar entropy wave for 3-d Navier-Stokes equations in Eulerian coordinates}

	\author[R.-J. Duan]{Renjun Duan}
	\address[R.-J. Duan]{Department of Mathematics, The Chinese University of Hong Kong, Shatin, Hong Kong, China.}
	\email{rjduan@math.cuhk.edu.hk}

	\author[F.~Huang]{Feimin Huang}
	\address[F.~Huang]{Academy of Mathematics and Systems Science, Chinese Academy of Sciences, Beijing 100190, China, and School of Mathematical Sciences, University of Chinese Academy of Sciences, Beijing 100049, China.}
	\email{fhuang@amt.ac.cn}

	\author[R.~Li]{Rui Li}
	\address[R.~Li]{Department of Mathematics, The Chinese University of Hong Kong, Shatin, Hong Kong, China.}
	\email{ruili001@cuhk.edu.hk}

	\author[L.~Xu]{Lingda Xu}
	\address[L.~Xu]{Department of Applied Mathematics, The Hong Kong Polytechnic University, Hong Kong, China.}
	\email{lingda.xu@polyu.edu.hk}

	\begin{abstract}
		We investigate the large-time asymptotic behavior toward the planar entropy wave for the three-dimensional Navier-Stokes equations in Eulerian coordinates, considering two types of initial perturbations---with and without the assumption that the integral of the initial perturbation is zero. Generic perturbations generate diffusion waves, and structural conditions fail for multi-dimensional Navier-Stokes equations in Eulerian coordinates. These two aspects have posed significant challenges and left the problem unresolved for years. On one hand, since \cite{LX}, the study of the entropy wave has been based on the left-right structural conditions. Without these structural conditions, the decay rates of lower-order terms become too slow to close the {\it a priori} assumption. On the other hand, the presence of diffusion waves yields problematic error terms in the perturbation system. In this work, we introduce a new transformation to ensure that both left-right structural conditions hold for the perturbation system. Additionally, using the fact that the derivative of the entropy wave maintains a fixed sign, we employ well-designed weighted energy estimates to control the slowly decaying terms. This enables us to establish asymptotic stability and derive the optimal decay rate. Furthermore, we address the case of initial perturbations with the zero mass condition and obtain the optimal decay rate by additionally developing a Poincar\'e type inequality and a key cancellation. 
	\end{abstract}

	

	
	\maketitle
	
	\setcounter{tocdepth}{2}
	\tableofcontents
	\thispagestyle{empty}

	\section{Introduction}
	\subsection{The models and problem}
	We study the three-dimensional (3-d) full Navier-Stokes (N-S) equations in Eulerian coordinates, which read as
	\begin{equation}\label{NS}
		\begin{cases}
			\rho_{t}+\dv (\rho \uf)=0, \\
			(\rho \uf)_{t}+\dv(\rho \uf \otimes \uf)+\nabla p=\dv\,\mathbb{S},\\
			(\rho \E)_{t}+\dv(\rho \E \uf+p \uf)=\kappa \Delta \theta+\dv(\uf \mathbb{S}),\\
			(\rho,\uf,\theta)(x,0)=(\rho_0,\uf_0,\theta_0),
		\end{cases}
		\qquad  (x,t)\in\R^3\times \R_{+}.
	\end{equation}
	Here, the unknown functions $\rho$, $\uf=\left(u_{1}, u_{2}, u_{3}\right)^{t}$ and $\theta$, depending on the spatial variable $x=\left(x_1,x_2 ,x_3\right) \in \mathbb{R}^{3}$ and the time variable $t \geq 0$, stand for the density, velocity, and absolute temperature of the fluid, respectively. Moreover, $p$ denotes the pressure, $\E:=e+\frac{1}{2}|\uf|^{2}$ denotes the specific total energy with $e$ being the internal energy, and $\kappa>0$ denotes the heat-conductivity coefficient for the fluid. The viscous stress tensor $\mathbb{S}$ is given as 
	\begin{align*}
		\mathbb{S}=2\mu\mathbb{D}(\uf)+\lambda\dv\, \uf\,\mathbb{I},
	\end{align*}
	where $\mathbb{D}(\uf):=\frac{\nabla \uf+(\nabla \uf)^{t}}{2}\in\R^{3\times3}$ represents the deformation tensor, and $\mathbb{I}$ is the identity matrix. The shear and bulk viscosity coefficients $\mu$ and $\lambda$ of the fluid satisfy the physical constraints
	$$
	\mu>0, \quad 2 \mu+3 \lambda \geq 0.
	$$
	Actually, we only need the weaker constraints $\mu>0$ and $\lambda+\mu\geq0$ throughout the paper.
	
	In the present work, we consider the ideal polytropic gas, that is, the pressure $p$ and the inner energy $e$ can be determined by the following state equations
	\begin{align}\notag
		p=R\rho\theta=A\rho^\gamma \exp(\frac{\gamma-1}{R}S),\qquad\quad e= \frac{R}{\gamma-1}\theta,
	\end{align}
	where $S$ denotes the entropy, $A>0$, $R>0$ are fluid constants, and $\gamma>1$ is the adiabatic exponent. Without loss of generality, we assume $R=\gamma-1$.
	
	When the viscous effect vanishes, Navier-Stokes equations in \cref{NS} are reduced to the Euler system
	\begin{align}\label{Eu}
		\begin{cases}
			\rho_{t}+\dv (\rho \uf)=0, \\
			(\rho \uf)_{t}+\dv(\rho \uf \otimes \uf)+\nabla p=0,\qquad (x,t)\in\R^3\times \R_{+}, \\
			(\rho \E)_{t}+\dv(\rho E \uf+p \uf)=0,
		\end{cases}
	\end{align}
	which is a classical system of hyperbolic conservation laws. One of the main features of the Euler system is the formulation of singularities. That is, no matter how smooth the initial data is, the discontinuity of the solution may form. To see the discontinuity structure of solutions, we study the Riemann problem for which the solution usually consists of three important basic waves: rarefaction wave, shock wave, and contact discontinuity. Depending on the type of discontinuity, contact discontinuities can be classified into entropy waves and vortex sheets. In entropy waves, the entropy is discontinuous, whereas in vortex sheets, the discontinuity lies in the shear velocity. It is known that shock waves and rarefaction waves are generated in the generally nonlinear characteristic fields and the entropy wave is generated in the linear degenerate field.

	In the present paper, we are interested in studying stability of planar wave patterns. For the purpose, the domain under consideration is the infinite long flat torus $\Omega:=\R\times\Torus^2$ in space of three dimensions, where $\Torus:=\R/\mathbb{Z}$. We would emphasize that even in such simple multi-dimensional setting, there are several essential difficulties compared to the 1-d case, and the study of the whole space $\R^3$ will be left for the future. Moreover, we focus on the planar entropy waves given as the weak solutions of the 1-d Euler equations as follows:
	\begin{align}\label{eq-rmns}
		&\left\{\begin{array}{l}
			\pt\rho+\px m_1=0,\\
			\pt{m_1}+\px\left(({\gamma-1}){\E}+\frac{3-\gamma}{2}\frac{{m}_{1}^2}{\rho}\right)=0, \\
			\pt{\E}+\px\left(\frac{\gamma{m}_1 {\E}}{\rho}-\frac{(\gamma-1){m}_1^{3}}{2 \rho^{2}}\right)=0,
		\end{array}\right.
	\end{align}
	in terms of the unknowns $(\rho,m_1=\rho u_1,\E)$. Studies for entropy waves are somehow more subtle than those for other wave patterns because of the linear degeneracy of entropy waves. Accordingly, there are some structural conditions that need to be considered when studying entropy waves. The Jacobi matrix of the flux for \cref{eq-rmns} is
	\begin{align*}
		A\left(\rho, m_1, \E\right)=\left(\begin{array}{ccc}
			0 & 1 & 0 \\
			-\frac{3-\gamma}{2}\frac{m_1^2}{\rho^2} & (3-\gamma)\frac{m_1}{ \rho} & \gamma-1 \\
			-\frac{\gamma m_1 \E}{ \rho^2}+\frac{(\gamma-1)m_1^3}{ \rho^3} & \frac{\gamma \E}{ \rho}-\frac{3(\gamma-1) m_1^2}{2 \rho^2} & \frac{\gamma m_1}{ \rho}
		\end{array}\right).
	\end{align*}
	Regarding the second eigenvalue $\lambda_2=\frac{m_1}{\rho}$ of $A\left(\rho, m_1, \E\right)$, the corresponding left and right eigenvectors are
	\begin{align}\label{NSlr}
		r_2=\left(1,\frac{m_1}{\rho},\frac{m^2_1}{2\rho^2}\right)^t,\qquad\qquad l_2=\left(\frac{\gamma\E}{\rho}-\frac{3(\gamma-1)m_1^2}{2\rho^2},(\gamma-1)\frac{m_1}{\rho},-1\right).
	\end{align}
	From direct computations, \cref{NSlr} further gives that
	\begin{align}\label{NSlr.veri}
		\nabla l_2\cdot r_2\neq0,\qquad\nabla r_2\cdot r_2=0.
	\end{align}
	However, the classical left and right structural conditions as in \cite{LX} are given as follows:
	\begin{align}\label{SC}
		\nabla l_2\cdot r_2=0, \quad \text{and} \quad \nabla r_2\cdot r_2=0,\quad\text{ respectively. } 
	\end{align}
	Therefore, from \cref{NSlr.veri} the left structural condition is violated for the Navier-Stokes equations \cref{eq-rmns} in Eulerian
	coordinates.

	The two structural conditions in \cref{SC} were originally proposed by Liu-Xin \cite{LX} for establishing the meta-stability of contact discontinuities for general  conservation laws with artificial viscosity via the Green's function method. Subsequently, Xin-Zeng \cite{XZ} demonstrated that the same result holds even without the left structural condition. A critical distinction, however, arises when considering conservation laws with physical viscosity instead of artificial viscosity. In the artificial viscosity case, the contact wave constitutes an exact solution, whereas for physical viscosity—as in the Navier-Stokes equations—it serves only as an approximate solution. This is intrinsic to the nature of contact waves as diffusion waves, which are generally not exact solutions of hyperbolic-parabolic coupled systems. For the Navier-Stokes equations in Lagrangian coordinates, this challenge was overcome by Huang et al.~\cite{HMX, HXY} through an ingenious construction of the contact wave and several key observations. Their approach, however, relies on the left structural condition and does not extend to the Eulerian coordinate framework. Consequently, the stability of contact discontinuities for the Navier-Stokes equations in Eulerian coordinates remains a major open problem, as previously noted for the one-dimensional case in \cite{XZ}; the multi-dimensional case presents even greater difficulties. It is also noteworthy that Huang-Matsumura-Xin \cite{HMX} conjectured the optimal decay rate for the contact wave. The question recently has been resolved by Liu-Wang-Xu \cite{LWX} under both structural conditions for the 1-d Navier-Stokes equations in Lagrangian coordinates. This success underscores the significant and ongoing challenge of deriving the optimal decay rate of solutions to \cref{eq-rmns} in the absence of the left structural condition as in \cref{NSlr.veri}.
	
	\subsection{Literature review}
	
	For the 1-d case, there have been numerous results on the stability of waves. In 1960, Ole\v{\i}nik et al.~\cite{O} studied the stability of shock and rarefaction waves for a single conservation law. In 1985, Matsumura-Nishihara~\cite{MN} and Goodman~\cite{G} independently used the anti-derivative method to prove the stability of shock waves for systems, under the zero-mass condition. In 1986, Matsumura-Nishihara~\cite{MN2} proved the stability of rarefaction waves for the Navier-Stokes equations. In 2005, Huang-Matsumura-Xin~\cite{HMX} established the stability of entropy waves for the Navier-Stokes equations. This result was further extended to the Boltzmann equation without requiring the zero-mass condition by Huang-Xin-Yang~\cite{HXY} in 2008. 
	
	For solutions involving multiple wave patterns, Huang-Matsumura~\cite{HM} in 2009 constructed diffusion waves and proved the stability of composite waves containing two shock waves for the non-isentropic Navier-Stokes equations without imposing the zero-mass condition. In 2010, Huang-Li-Matsumura~\cite{HLM} proved the stability of composite waves consisting of two rarefaction waves and entropy waves via a new inequality. In 2021, Kang-Vasseur-Wang~\cite{KVW} studied the large-time behavior of composite waves consisting of a shock and a rarefaction wave in the isentropic Navier-Stokes equations by introducing a time-dependent shift.
	
	The above results are obtained via the energy method. On the other hand, the Green's function method, initiated in~\cite{Liu1,Liu3}, has proven powerful in studying the pointwise behavior of waves. Using a novel construction of approximate Green's functions, Liu-Zeng~\cite{LZ3} obtained pointwise estimates of shock waves in 2009. Their method applies to physical viscosity cases, such as the Navier-Stokes and MHD equations~\cite{LZ5}. In 2017, Liu-Wang~\cite{LW} obtained pointwise estimates of rarefaction waves for scalar conservation laws. For the derivation of nonlinear stability from linear and spectral stability, we refer to~\cite{Z}. Since this paper focuses on entropy waves, relevant results will be introduced in more detail later. We also note that many interesting results are not included here; see~\cite{SX}.
	
	Regarding entropy waves, for the Navier-Stokes equations in Lagrangian coordinates, both structural conditions hold. However, viscous terms introduce errors, which complicates the analysis. These errors cause the energy of anti-derivatives to grow in time. Huang et al.~\cite{HMX,HXY} found that this growth can be compensated by the decay of the energy of derivatives. They obtained a decay rate of $(1+t)^{-1/4}$ and conjectured an optimal rate of $(1+t)^{-1/2}$, based on the results of Liu-Xin and Xin-Zeng~\cite{LX,XZ}. This conjecture has been studied by several groups, including Huang-Wang-Wang~\cite{HWW} and Yang~\cite{Yang}, where the best known decay rate is $(1+t)^{-1/3}$. More recently,~\cite{LWX} revisited this conjecture for the Navier-Stokes equations in Lagrangian coordinates. The studies of entropy waves in the setting that both structural conditions hold are now quite complete.
	
	However, in the Eulerian coordinates for Navier-Stokes equations, the left structural condition fails, and consequently, very few results are available. In fact, to the best of our knowledge, there are no known results, especially without the zero-mass condition, proving the stability of entropy waves for the Navier-Stokes equations in Eulerian coordinates, even in the 1-d case. In this paper, we consider the multi-dimensional case and obtain the optimal decay rate $(1+t)^{-1/2}$ for the Eulerian coordinates. The presence of the diffusion wave plays a crucial role in achieving this optimal rate.
	
	For multi-dimensional planar waves, in 1990, Xin~\cite{Xin1990} studied planar rarefaction waves for a single conservation law. In 2018, Li-Wang-Wang~\cite{LW-1,LWW} proved the stability of planar rarefaction waves with physical viscosity in $\R\times \Torus$ and $\R \times \mathbb{T}^2$. Due to the fact that approximate rarefaction waves are not exact solutions for viscous conservation laws, the multi-dimensional perturbations cannot be localized. Instead, Huang-Xu-Yuan~\cite{HXY2022}  proved stability under periodic (non-local) perturbations. Recently, Huang-Xin-Xu-Yuan~\cite{HXXY} and Yuan-Zhao~\cite{YZ2025} proved the stability of the viscous vortex sheet in the 3-D isentropic Navier-Stokes equation under periodic (non-local) perturbations. 
	
	There are also many notable results on planar shock waves. Notably, Humpherys-Lyng-Zumbrun~\cite{HLZ2017} provided numerically well-conditioned and analytically justified computations verifying the nonlinear stability of planar viscous shocks for the Navier-Stokes equations. Freist\"uhler-Szmolyan~\cite{FS1,FS2} proved spectral stability for small-amplitude viscous shocks in systems with artificial viscosity. More recently, Yuan~\cite{Yuan} and Wang-Wang~\cite{WW2022} made significant progress on the rigorous proof of nonlinear stability for Navier-Stokes shock waves:~\cite{WW2022} applied the $a$-contraction method with a time-dependent shift, while~\cite{Yuan} introduced a framework involving multi-dimensional anti-derivatives. We also refer to Liu-Yu~\cite{LY}, where the propagation of planar shock waves is studied in a quantitative, pointwise manner, and more complex wave patterns, including Rayleigh-type waves, are observed.

	\subsection{Difficulties and strategy}
	
	To the best of our knowledge, there are no such stability results concerning the m-d viscous entropy wave. It should be pointed out that in the present understanding, the structural conditions are crucial for viscous entropy waves by the energy method and Green's function approach. Thus, the stability of the viscous entropy wave is interesting and completely different from the other two wave patterns.

	The stability of planar contact waves in multi-dimensional compressible Navier-Stokes equations involves several essential differences from the previous results as far as we know. In what follows, we introduce the difficulties and novelties with more details.
	\begin{flushleft}
		\textbf{Difficulties:}
	\end{flushleft}
	\begin{itemize}
		\item The limitation of structural conditions \cref{SC}.
		
		Although the derivative of the contact discontinuity maintains a fixed sign, $\px\rhob$ and $\px\Tb$ cannot share the same sign; see \cref{eq-diffusionwave,cd-eu,cd}. That is one of the main differences between entropy wave and other basic wave patterns such as rarefaction wave and shock wave. Fortunately, previous research works \cite{LX,HMX,HXY} showed, with the help of two sides of structural conditions, that the lower order terms of perturbation $(\phi,\varphi,\zeta):=(\rho-\rhob,u-\ub,\theta-\Tb)$ are 
		\begin{align*}
			\int_{\Omega}(1+t)^{-1}e^{-\frac{cx^2}{1+t}}\Big(\phi^2+\varphi^2+\zeta^2\Big)dx.
		\end{align*}
		In the 1-d case, \cite{HXY} found that although the energy of anti-derivatives increases with the rate $(1+t)^{\frac{1}{2}}$, the $L^{\infty}$-norm can be bounded since the energy of perturbation decays at the rate of $(1+t)^{-\frac{1}{2}}$. Thus, both of the two structural conditions are crucial. If the left structural condition fails, the lower order terms become
		\begin{align*}
			\int_{\Omega}(1+t)^{-\frac{1}{2}}e^{-\frac{cx^2}{1+t}}\Big(\phi^2+\varphi^2+\zeta^2\Big)dx.
		\end{align*}
		Then, arguments similar to those in \cite{HMX,HXY} can only yield uniform bounds for the first-order energy of anti-derivatives within the anti-derivative framework, which is not enough to close the a {\it priori} assumptions. Without the approaches of anti-derivatives, it is also hard to control such terms. Note that one can construct the entropy wave as an exact solution for conservation laws with artificial viscosity; thus there will not be growth of energy for anti-derivatives and this problem is somewhat simpler in this case.
		
		In sum, the stability of entropy wave for systems with physical viscosity and without left structural condition is very difficult and essentially different from other cases introduced previously.

		\item The optimal decay rate.
		
		Note that the structural conditions in \cref{SC} are at the level of the first order of anti-derivatives, which helps one to obtain a good estimate for zero-order energy. In the higher-order estimates, it fails and terms with low decay rates arise. In \cite{HMX}, the authors conjectured that the optimal rate should be $(1+t)^{-\frac{1}{2}}$ due to the presence of diffusion wave. Recently, \cite{LWX} solved this problem in Lagrange coordinates, but not for Eulerian coordinates. As explained previously, the structural conditions are crucial in computing the decay rates. Thus, our cases are more subtle.

	\end{itemize}
	Based on these difficulties, we solve the problems by the following strategies.
	
	\begin{flushleft}
		\textbf{Strategies:}
	\end{flushleft}
	\begin{itemize}
		\item The decomposition of zero and non-zero mode in Fourier space. 
		
		The perturbations can be decomposed into
		\begin{align*}
			(\phim,\psim,\mr{w})=\int_{\Torus^2}(\rho-\rhot,\m-\mtb,\E-\Et)(x,t)dx_2dx_3,\qquad (\phi_{\neq},\psi_{\neq},w_{\neq})=(\phi,\psi,w)-(\phim,\psim,\mr{w}),
		\end{align*}
		where $(\rhot,\mt,\Et)$ is  defined as the coupling of entropy waves and diffusion waves in \cref{tilde}; more details can be found in \cref{section-2-t}. By this spirit, we can decompose the perturbation system into two coupled parts. The key observation is that the non-zero mode of perturbations decays exponentially.

		\item The integrated system and a new transformation.
		
		We study the perturbation in an integrated system with
		\begin{align*}
			(\Phi,\Psi,W)=\int_{-\infty}^{x_1}\int_{\Torus^2}(\rho-\rhot,\m-\mtb,\E-\Et)(y,t)dy'dy_1.
		\end{align*}
		Next, we will apply the following transformation
		\begin{align*}
			&\tilde{\Phi}=\Tt\Phi,\qquad\tilde{\Psi}={\Psi}- \mathbf{\ut} \Phi,\\
			&\tilde{W}={W}-\mathbf{{\utb}} \cdot \tilde{\Psi}-\left({\thetat}+\frac{| \mathbf{\ut}|^{2}}{2}\right) \Phi,
		\end{align*}
		where $\tilde{\uf}$ and $\thetat$ are the non-conserved quantities of velocity and temperature under the new ansatz \cref{non-conserved quantities-111}, respectively. 
		Through this transformation, the new system is simplified a lot and satisfies both the two structural conditions, see \cref{eq-ever}.
		
		\item Estimates on the non-zero mode.

		We observe that the Poincar\'{e} inequality holds for non-zero modes. Thus the exponential decay rate for them is expected.

		\item The optimal decay rate.
		
		Note that although the constructed entropy wave \cref{eq-diffusionwave} is monotonic, the sign of derivatives of $\Tb$ and $\rhob$ are opposite, see \cref{cd}. Thus, one cannot use arguments similar to those for shock waves or rarefaction waves that essentially rely on the sign of derivatives of profiles. In the current case, we perform energy estimates using the diagonalized system to fully capture the intrinsic feature and determine the dissipation structure as specified in \cref{sec-n-3-tt,ds}. In particular, we obtain the dissipation term
		$$
		G_k:=C\delta^{-\frac12}\int_{\R} \px\rhob\abs{\px^kb_1}^2+ \px\rhob\abs{\px^kb_3}^2dx_1,
		$$
		see  \cref{sec-n-3-tt}.
		It is also worth noting that, since $\p^{k+1} \rhob$ cannot be directly controlled by $(1+t)^{-\frac{k}{2}}\px \rhob$ in terms of \cref{2025-10-6}, we further develop weighted estimates of the form $\int_{\R}\tilde{\omega}_{-\frac12}|f|^2dx_1$ as in \cref{2025-10-4,2025-10-5} based on the obtained dissipation term $G_k$, where $\tilde{\omega}_{-\frac12}$ is given in \cref{errors-1} with $\tilde{c}_0 <c_0$ and $c_0$ given in \cref{2025-10-6}.
		Then the terms with lower decay rate can be controlled. 
		
		\item On the case of zero initial mass.

		There are several key cancellations we need to achieve in this case. For instance, \cref{key} shows that the slow decay terms will not appear in the 2-characteristic field. And we need a Poincar\'e type inequality \cref{lem5} to deal with the critical terms.
		
	\end{itemize}

	We believe that the transformation and the weighted energy method proposed here, rooted in the structural properties of entropy waves, have the potential of applications to a variety of important systems, including the Navier-Stokes-Poisson and magnetohydrodynamics equations.

	The rest of this paper is organized as follows. In Section 2, we will construct the ansatz and then introduce the main results. Then we will introduce the perturbation system, the key transformations,  energy estimate for zero mode and   energy estimate for non-zero mode for the case without initial zero mass in Section 3. Finally, in Section 4, we obtain the optimal decay rate for the case with initial zero mass.
	
	\section{Construction of ansatz and main theorems}\label{section-2-t}
	\subsection{Planar viscous entropy wave for 3-d Navier-Stokes equations}

	In this subsection, we give the explicit expression of planar entropy wave to the Navier-Stokes equations \cref{NS}. We start from the 1-d Euler system \cref{NS} along the $x_1$-direction: 
	\begin{align*}
		\begin{cases}
			\p_{t}\rho+\px\left(\rho u_{1}\right)=0, &  \\
			\p_{t}\left(\rho u_{1}\right)+\px\left(\rho u_{1}^{2}+p\right)=0, & x_1  \in \mathbb{R}, t>0,\\
			\p_{t}(\rho E)+\px\left(\rho E u_{1}+p u_{1}\right)=0, &
		\end{cases}
	\end{align*}
	with the initial data:
	\begin{align*}
		\left(\rho, u_{1}, u_2,u_3,\T\right)\left(x_1,0 \right)
		=\left\{\begin{aligned}
			\left(\rho_{-}, 0, 0, 0, \theta_{-}\right), &\quad x_1 <0, \\
			\left(\rho_{+}, 0, 0, 0, \theta_{+}\right), &\quad x_1 >0.
		\end{aligned}\right.
	\end{align*}
	In this paper, we study the weak entropy wave, that is, the strength of wave ${\delta}:=\abs{\rho_+-\rho_-}+\abs{\T_+-\T_-}$ is small.
	It is known that entropy wave is generated in the linear degenerate characteristic field
	\begin{align}\label{cd-eu}
		R{\rho_+}\T_+=p_+=p_-= R{\rho_{-}}{\T_{-}}.
	\end{align} 
	In the case of Navier-Stokes equations, we will modify the entropy wave as a smooth profile, called the viscous entropy wave.
	
	We start from the construction of diffusion waves:
	\begin{align}\label{eq-diffusionwave}
		\begin{cases}
			\p_t\bar{\rho}=\px\left(a\frac{ \px\bar{\rho}}{\bar{\rho}}\right),\quad\quad\quad a:=\frac{\kappa  }{\gamma },\\[3mm]
			\bar{\rho}(-\infty,t)=\rho_-,\quad \bar{\rho}(+\infty,t)=\rho_+.
		\end{cases}
	\end{align}
	From  \cite{HL,LX}, we know that \cref{eq-diffusionwave} admits a unique self-similar solution $\bar{\rho}(\frac{x_1}{\sqrt{t+1}})$, satisfying
	\begin{align}\label{2025-10-6}
		(1+t)^{\frac{1}{2}}\p_{1}\bar{\rho}=C\delta e^{-\frac{c_0x_1^2}{1+t}}.
	\end{align}
	
	In terms of $\bar{\rho}$, we can construct an approximation solution for 3-d Navier-Stokes equations. Indeed, setting
	\begin{align}\label{cd}
		&\bar{\rho}=\bar{\rho}, \quad
		\bar{\uf}=(-a\frac{\px\bar{\rho}}{|\bar{\rho}|^{2}},0,0),\nonumber\\
		&\Tb=\frac{p_+}{ R\bar{\rho}}-\frac{\bar{u}_{1}^{2}}{2} , \quad
		R \bar{\rho} \Tb=\bar{p}=p_{+}-\frac{(\gamma-1)}{2} \bar{\rho} |\bar{u}_{1}|^{2},
	\end{align}
	it holds that
	\begin{align}\label{eq-ecd}
		\left\{\begin{array}{l}
			\pt\bar{\rho}+\px(\bar{\rho} \bar{u}_1)=0, \\[2mm]
			\pt(\bar{\rho} \bar{u}_1)+\px\left(\bar{\rho}  |\bar{u}_1|^{2}+\bar{p}\right)=(2\mu+\lambda)\px^2 \bar{u}_{1}+\px{Q}_{1} ,\\[2mm]
			\pt\left(\bar{\rho}\left( \Tb +\frac{|\bar{\uf}|^{2}}{2}\right)\right)+\px\left(\bar{\rho} \bar{u}_1\left( \Tb+\frac{ |\bar{\uf}|^{2}}{2}\right)+\bar{p} \bar{u}_1\right)\\[3mm]
			\qquad\qquad\quad\qquad\qquad\quad=\kappa \px^2\Tb+(2\mu+\lambda)\px \left(\bar{u}_1 \px\bar{u}_{1}\right)+\px{Q}_{2},
		\end{array}\right.
	\end{align}
	where 
	\begin{equation}\notag
		\left\{
		\begin{aligned}
			&Q_{1}=-a\frac{\pt\bar{\rho}}{\bar{\rho}}+\frac{(3-\gamma)}{2}\bar{\rho}(\bar{u}_{1})^2-(2\mu+\lambda)\p_1\bar{u}_1
			\le C\delta (1+t)^{-1}e^{-\frac{cx_1^2}{1+t}},\\
			&Q_2=-\frac{(\gamma-1)}{2}\bar{\rho}(\bar{u}_{1})^3+\frac{\p_1(|\bar{u}_1|^2)}{2}-(\lambda+2\mu)(\bar{u}_{1}\px \bar{u}_{1}) \le C\delta (1+t)^{-\frac{3}{2}}e^{-\frac{cx_1^2}{1+t}},
		\end{aligned}
		\right.
	\end{equation}
	with some constant $c>0$. For convenience, we denote the conserved quantities by
	\begin{align*}
		U=(\rho,m_1,m_2,m_3,\E)^{t},\quad\;\bar{U}=(\rhob,\mb_1, \mb_2,\mb_3,\bar{\E})^{t},\quad\; U^{\#}=(\rho,m_1,\E)^{t}, \quad\;\bar{U}^{\#}=(\rhob,\mb_1, \bar{\E})^{t}.  
	\end{align*}

	\subsection{The extra initial mass and the diffusion waves}
	
	We will construct diffusion waves to carry the extra initial mass, which is inspired by \cite{Liu1,SX,HXY}, to solve this problem. Note that the extra initial mass is distributed along the $x_1$-direction. The Jacobi matrices for the flux of 1-d Navier-Stokes equations in Eulerian coordinates \cref{eq-rmns} at $(\rho_+,m_{1+},\E_+)$ and $(\rho_-,m_{1-},\E_-)$ are given, respectively, as 
	\begin{align*}
		A_\pm=\left(\begin{array}{ccc}
			0&1&0\\
			0&0&\gamma-1\\
			0&\gamma\frac{\E_\pm}{\rho_\pm}&0
		\end{array}\right).
	\end{align*}
	Then it is direct to compute that $\lambda_{1}^-=-\sqrt{\gamma (\gamma-1)\frac{\E_-}{\rho_-}}$ is the first eigenvalue of $A_-$ corresponding with $r_1^{-}=(1,\lambda_1^-,\frac{(\lambda_1^{-})^2}{\gamma-1})^t$ and $\lambda_{3}^+=\sqrt{\gamma (\gamma-1)\frac{\E_+}{\rho_+}}$ is the third eigenvalue of $A_+$ corresponding with $r_3^{+}=(1,\lambda_3^+,\frac{(\lambda_3^{+})^2}{\gamma-1})^t$. Since the following three vectors, $r_1^-$, $(\rho_--\rho_+,m_{1-}-m_{1+},\E_--\E_+)$ and $r_3^+$ are linearly independent, we have
	\begin{align}\label{eq-initial mass}
		\int_{\Omega}U^{\#}-\bar{U}^{\#}dx=\bar{\Theta}_1r_1^-+\bar{\Theta}_2(\rho_--\rho_+,m_{1-}-m_{1+},\E_--\E_+)^t+\bar{\Theta}_3r_3^+.
	\end{align}
	We will couple two diffusion waves with $(\rhob,\ub,\thetab)$, which are defined by
	\begin{align*}
		\Theta_1(x_1,t)=\frac{1}{\sqrt{4\pi(1+t)}}e^{-\frac{(x_1-\lambda_1^-(1+t))^2}{4(1+t)}},\qquad\Theta_3(x_1,t)=\frac{1}{\sqrt{4\pi(1+t)}}e^{-\frac{(x_1-\lambda_3^+(1+t))^2}{4(1+t)}}.
	\end{align*}
	It is direct to verity that they satisfy
	\begin{align*}
		\left\{\begin{aligned}
			&\pt\Theta_{1}+\lambda_1^{-}\px\Theta_{1}=\px^2\Theta_{1},\\
			& \int_{-\infty}^{+\infty}\Theta_1(x_1,t)dx=1,
		\end{aligned}\right.\qquad\left\{\begin{aligned}
			&\pt\Theta_{3}+\lambda_3^{+}\px\Theta_{3}=\px^2\Theta_{3},\\
			&\int_{-\infty}^{+\infty}\Theta_3(x_1,t)dx=1.
		\end{aligned}\right.
	\end{align*}
	Let
	\begin{align*}
		\bar{\Theta}_{i+2}:=\int_{\Omega}m_i(x,0)dx,\qquad i=2,3.
	\end{align*}
	Then we define the new ansatz $(\tilde{\rho},\tilde{m},\tilde{\E})$ by
	\begin{align}\label{tilde}
		\begin{aligned}
			&\tilde{\rho}(x_1, t)={\bar{\rho}}\left(x_1+\bar{\Theta}_2, t\right)+\bar{\Theta}_{1} \Theta_{1}+\bar{\Theta}_{3} \Theta_{3}, \\
			&\tilde{m}_{1}(x_1, t)={\bar{m}}_{1}\left(x_1+\bar{\Theta}_2, t\right)+\lambda_{1}^{-} \bar{\Theta}_{1} \Theta_{1}+\lambda_{3}^{+} \bar{\Theta}_{3} \Theta_{3}, \\
			&\tilde{m}_{i}=\frac{\bar{\Theta}_{i+2}}{\sqrt{4 \pi(1+t)}} e^{-\frac{x_1^{2}}{4(1+t)}}, \quad i=2,3, \\
			&\tilde{\E}(x_1, t)=\bar{\E}\left(x_1+\bar{\Theta}_2, t\right)+\left(\frac{(\lambda_1^-)^2}{\gamma-1}\bar{\Theta}_{1} \Theta_{1}+\frac{(\lambda_3^+)^2}{\gamma-1}\bar{\Theta}_{3} \Theta_{3}\right).
		\end{aligned}
	\end{align}
	Without loss of generality, we assume $\bar{\Theta}_2=0$. Then the initial extra mass for new ansatz \cref{tilde} is
	\begin{align*}
		\begin{aligned}
			&\int_{\Omega}\left[U(x,0)-\tilde{U}(x_1,0)\right]dx=\int_{\Omega}\left[U^{\#}(x,0)-\bar{U}^{\#}(x_1,0)\right]dx+\sum_{i=2}^3\bar{\Theta}_{i+2}+\int_{\R}\left(\bar{U}-\tilde{U}\right)(x_1,0)dx_1=0,
		\end{aligned}
	\end{align*}
	where $\tilde{U}=(\rhot,\mt_1,\mt_2,\mt_3,\Et)$. 
	For $i=1,2,3$, we also define non-conserved quantities: velocity  $\ut_i$, temperature $\thetat$ and pressure $\tilde{p}$ as follows
	\begin{align}\label{non-conserved quantities-111}
		\ut_i:=\frac{\mt_i}{\rhot},\qquad \qquad\thetat:=\frac{\Et}{\rhot}-\frac{\abs{\mt}^2}{2\rhot^2},\qquad\qquad \tilde{p}:=(\gamma-1)\Et-\frac{(\gamma-1)\abs{\mt}^2}{2\rhot}.
	\end{align}
	By a direct calculation, we find that the new ansatz \cref{tilde} satisfies the following approximate Navier-Stokes equations
	\begin{equation}\label{NS-ansatz}
		\begin{cases}
			\p_t\tilde{\rho}+\dv \, \mt={\px\tilde{H}_1},\\[2mm]
			\p_t\tilde{m}_{1}+\dv ( \frac{\tilde{m}_1 \tilde{\mathbf{m}}}{\tilde{\rho}})+ \p_{1}\tilde{p}=\mu\triangle\tilde{u}_1+(\mu+\lambda)\px(\dv \tilde{\mathbf{u}})+\dv\tilde{\mathbf{H}}_{21},\\[2mm]
			\p_t\tilde{m}_{2}+\dv ( \frac{\tilde{m}_2 \tilde{\mathbf{m}}}{\tilde{\rho}})+ \p_2\tilde{p}=\mu\triangle\tilde{u}_2+(\mu+\lambda)\p_2(\dv \tilde{\mathbf{u}})+\dv\tilde{\mathbf{H}}_{22},\\[2mm]
			\p_t\tilde{m}_{3}+\dv ( \frac{\tilde{m}_3 \tilde{\mathbf{m}}}{\tilde{\rho}})+ \p_3\tilde{p}=\mu\triangle\tilde{u}_3+(\mu+\lambda)\p_3(\dv \tilde{\mathbf{u}})+\dv\tilde{\mathbf{H}}_{23},\\[2mm]
			\p_t\tilde{\mathcal{E}}+\dv (\tilde{\mathbf{u}}(\tilde{\mathcal{E}}+\tilde{p}))=\dv (\tilde{\mathbf{u}}\tilde{S}+\kappa\nabla\tilde{\theta})+\dv\tilde{\mathbf{H}}_{3},\\[2mm]
		\end{cases}
	\end{equation}
	where
	\begin{align*}
		&\tilde{H}_1=\Thetab_1\px\Theta_{1}+\Thetab_3\px\Theta_{3},\\
		&{\tilde{\mathbf H}}_{21}:=\lambda_1^-\Thetab_1\px\Theta_{1}\mathbb{I}_1+\lambda_3^+\Thetab_3\px\Theta_{3}\mathbb{I}_1-(\lambda+\mu)\dv\left(\frac{\mtb}{\rhot}-\frac{\bar{\mathbf{m}}}{\bar{\rho}}\right)\mathbb{I}_1+Q_1\mathbb{I}_1\\
		&\qquad+ \dv\left(\frac{\mt_1\mtb}{\rhot}-\frac{(\gamma-1)|\mtb|^2}{2\rhot}-\frac{\bar{m}_1\bar{\mathbf{m}}}{\bar{\rho}}+\frac{(\gamma-1)|\bar{\mathbf{m}}|^2}{2\bar{\rho}}\right)\mathbb{I}_1-\mu\px\left(\frac{\mt_1}{\rhot}-\frac{\bar{m}_1}{\bar{\rho}}\right)\mathbb{I}_1,\nonumber\\
		&{\tilde{\mathbf H}}_{2i}:=\Thetab_{i+2}\px\Theta_{i+2}\mathbb{I}_1-(\lambda+\mu)\dv\left(\frac{\mtb}{\rhot}-\frac{\bar{\mathbf{m}}}{\bar{\rho}}\right)\mathbb{I}_1\\
		&\qquad+ \px\left(\frac{\mt_i\mt_1}{\rhot}-\frac{(\gamma-1)|\mtb|^2}{2\rhot}-\frac{\bar{m}_i\bar{m}_1}{\bar{\rho}}+\frac{(\gamma-1)|\bar{\mathbf{m}}|^2}{2\bar{\rho}}\right)\mathbb{I}_1-\mu\px\left(\frac{\mt_i}{\rhot}-\frac{\bar{m}_i}{\bar{\rho}}\right)\mathbb{I}_1,\quad i=2,3,\nonumber\\
		&\tilde{\mathbf{H}}_3:=\frac{(\lambda_1^-)^2}{\gamma-1}\Thetab_1\px\Theta_{1}\mathbb{I}_1+\frac{(\lambda_3^+)^2}{\gamma-1}\Thetab_3\px\Theta_{3}\mathbb{I}_1-(\tilde{\mathbf{u}}\tilde{S}-\bar{\mathbf{u}}\bar{S})+Q_2\mathbb{I}_1\\
		&\qquad+(\gamma-1)\left(\frac{|\mtb|^{2}\mtb}{2 \rhot^{2}}-\frac{ |\bar{\mathbf{m}}|^{2}\bar{\mathbf{m}}}{2 \bar{\rho}^{2}}\right)-\kappa\nabla \left(\frac{\Et}{\rhot}-\frac{1}{2}\left|\frac{\mtb}{\rhot}\right|^2-\frac{\bar{\E}}{\bar{\rho}}+\frac{1}{2}\left|\frac{\bar{\mathbf{m}}}{\bar{\rho}}\right|^2\right)\notag\\
		&\qquad +\gamma \left( \frac{\mt_1\tilde{\E}}{\rhot}-\frac{\bar{m}_1 \bar{\E}}{\rhob} -\lambda_1^{-} \frac{\E_{-}}{\rho_{-}}\bar{\Theta}_3\Theta_1 -\lambda_3^{+}\frac{\E_{+}}{\rho_{+}}\bar{\Theta}_3\Theta_3    \right)\mathbb{I}_1\nonumber,
	\end{align*}
	and $\mathbb{I}_1:=(1,0,0)^t$.
	For $i=1,2,3$, simple calculations yield
	\begin{align}\label{tildeE}
		\big|\tilde{H}_{i}\big|+\abs{\tilde{\mathbf{H}}_{2i}}+\abs{\tilde{\mathbf{H}}_3}\leq& C\left(\delta+\sum_{j=1}^{5}\left|\bar{\Theta}_{j}\right|\right) \frac{1}{1+t}\left(e^{-\frac{c x_1^{2}}{1+t}}+e^{-\frac{c\left(x_1-\lambda_{1}^{-}(1+t)\right)^{2}}{1+t}}+e^{-\frac{c\left(x_1-\lambda_{3}^{+}(1+t)\right)^{2}}{1+t}}\right),
	\end{align}
	where $c>0$ is a constant independent of any small parameters in this paper. 
	
	Recalling the definition of $c_0$ in \cref{2025-10-6}, we use the following notation for convenience
	\begin{align}
		&D_{-\alpha}(x_1,t):=\frac{1}{(1+t)^{\alpha}}\left(e^{-\frac{c x_1^{2}}{1+t}}+e^{-\frac{c\left(x_1-\lambda_{1}^{-}(1+t)\right)^{2}}{1+t}}+e^{-\frac{c\left(x_1-\lambda_{3}^{+}(1+t)\right)^{2}}{1+t}}\right),\qquad\quad
		\omega_{-\alpha}(x_1,t):={(1+t)^{-\alpha}}e^{-\frac{c x_1^{2}}{1+t}},\label{errors}\\
		&\tilde{\omega}_{-\frac12}(x_1,t):=(1+t)^{-\frac12}e^{-\frac{\tilde{c}_0x_1^2}{1+t}}, \quad \text{for} \quad \tilde{c}_0 <c_0.\label{errors-1}
	\end{align}

	\subsection{Main results}
	
	We define the perturbation as
	\begin{align*}
		(\phi,\varphi,\zeta,\psi,w)(x,t):=(\rho-\rhot,\uf-\tilde{\uf},\theta-\Tt,\m-{\mtb},\mathcal{E}-\Et),\qquad (\phi_0,\varphi_0,\zeta_0,\psi_0,w_0):=(\phi,\varphi,\zeta,\psi,w)(x,0),
	\end{align*}
	and further the anti-derivatives for conserved quantities as
	\begin{align*}
		(\Phi,\Psi,W):=\int_{-\infty}^{x_1}\int_{\Torus^2}(\rho-\rhot,\m-\mt,\E-\Et)(y,t)dy'dy_1,\qquad(\Phi_0,\Psi_0,W_0):=(\Phi,\Psi,W)(x_1,0).
	\end{align*}
	Then we state the main results of this paper as follows.
	
	\begin{Thm}[Non-zero initial mass]\label{mt}
		Let $\Omega=\R\times\Torus^2$. Assume that $\left(\rho_{+}, u_{+}, \theta_{+}\right)$ and $\left(\rho_{-}, u_{-}, \theta_{-}\right)$ are the constant states of entropy wave satisfying \cref{cd-eu}, and $(\bar{\rho}, \bar{u}, \bar{\theta})\left( x_{1},t\right)$ is the planar viscous entropy wave connecting these two states as defined in \cref{eq-diffusionwave,cd}. Then there exist positive small constants $\varepsilon_{0}$, $\delta$ such that if
		$$
		{\delta}:=\abs{\rho_+-\rho_-}+\abs{\T_+-\T_-},\qquad\quad{{\|({\Phi}_0,{\Psi}_0,{W}_0)\|_{ H^3(\R)}^2+\norm{(\phi_0,\varphi_0,\zeta_0)}_{H^3(\Omega)}^2+\norm{(\phi_0,\psi_0,w_0)}_{L^1(\Omega)}^2:=\varepsilon^2 \leq \varepsilon^2_{0}}},
		$$
		then the Cauchy problem \cref{NS} admits a unique global smooth solution $(\rho, u, \theta)$ satisfying
		$$
		(\rho-\rhob, \uf-\bar{\uf}, \theta-{\thetab}) \in C\left(0,+\infty ; H^{3}(\Omega)\right), \nabla(\uf-\bar{\uf}, \theta-\Tb) \in L^{2}\left(0,+\infty ; H^3(\Omega)\right), \nabla(\rho-\rhob)\in L^{2}\left(0,+\infty ; H^2(\Omega)\right),
		$$
		and further
		\begin{align}\label{mt-decay}
			\sup _{x \in \Omega
			}\left|(\rho, \uf, \theta)(x, t)-(\bar{\rho}, \bar{\uf}, \bar{\theta})\left(x_{1},t\right)\right|\le C(\varepsilon_0^{\frac12}+\delta^{\frac12})(1+t)^{-\frac{1}{2}}.
		\end{align}
		Moreover, the non-zero modes of the unique solution have an exponential decay rate
		\begin{align*}
			\sup _{x \in \Omega
			}\left|\D_{\neq}\left[(\rho, \uf, \theta)(x,t)-(\bar{\rho}, \bar{\uf}, \bar{\theta})(x_1,t)\right]\right|\le C\varepsilon_0e^{-c t},
		\end{align*}
		where $C,c$ are  positive constants independent of $\varepsilon_0$ and $\delta$, and $\D_{\neq}$ is defined in \eqref{def-decom}.
	\end{Thm}
	
	\begin{Rem}
		We emphasize that according to the results of \cite{LX,XZ} where they used the method of approximate Green's functions to obtain pointwise estimates, the decay rate obtained in \cref{mt-decay} is optimal. Moreover, we also obtain the exponential decay rate for the non-zero mode.
		
	\end{Rem}
	\begin{Rem}
		To the best of our knowledge,  \cref{mt} can be regarded as the first result for the nonlinear asymptotic stability of the multi-dimensional entropy wave and also for systems with physical viscosity but without the left structural condition as in \eqref{NSlr.veri}.
	\end{Rem}
	
	\begin{Thm}[Zero initial mass]\label{mt0}
		Under the same assumptions as \cref{mt} and further assuming that the initial mass is zero, that is 
		\begin{align*}
			\int_{\Omega}{[(\rho_0,\m_0,\E_0)(x)-(\rhob,\mb,\Eb)(x_1,0)]}dx=0,
		\end{align*}
		the Cauchy problem \cref{NS} admits a unique global smooth solution $(\rho, u, \theta)$ satisfying
		$$
		(\rho-\rhob, \uf-\bar{\uf}, \theta-{\thetab}) \in C\left(0,+\infty ; H^{3}(\Omega)\right), \nabla(\uf-\bar{\uf}, \theta-\Tb) \in L^{2}\left(0,+\infty ; H^3(\Omega)\right), \nabla(\rho-\rhob)\in L^{2}\left(0,+\infty ; H^2(\Omega)\right),
		$$
		and further
		\begin{align}\label{mt-decay-1}
			\sup _{x \in \Omega}\left|(\rho, \uf, \theta)(x, t)-(\bar{\rho}, \bar{\uf}, \bar{\theta})\left(x_{1},t\right)\right|\le C(\varepsilon_0+\delta^{\frac12})(1+t)^{-\frac{3}{4}}\ln^{\frac{1}{2}} (2+t).
		\end{align}
		Moreover, the non-zero modes of the unique solution have an exponential decay rate
		\begin{align}\notag
			\sup _{x \in \Omega 
			}\left|\D_{\neq}\left[(\rho, \uf, \theta)(x, t)-(\bar{\rho}, \bar{\uf}, \bar{\theta})(x_1,t)\right]\right|\le C\varepsilon_0e^{-c t},
		\end{align}
		where $C,c$ are  positive constants independent of $\varepsilon_0$ and $\delta$, and $\D_{\neq}$ is defined in \eqref{def-decom}.
	\end{Thm}
	
	\begin{Rem}
		We point out that the time decay rate in \cref{mt-decay-1} is optimal in terms of the result obtained by using pointwise estimate for conservation laws with artificial viscosity in \cite{XZ}.
	\end{Rem}
	
	\begin{Rem}
		It should be noted that the optimal decay rate \cref{mt-decay-1} cannot be achieved simply by applying the conclusions and methods of deriving \cref{mt-decay} in case of non-zero initial mass. For zero initial mass, there are several key cancellations we need to make use of. For instance, \cref{key} shows that the slow decay terms do not appear in the 2-characteristic field. And we need a Poincar\'e type inequality as in \cref{lem5} to deal with the critical terms.
	\end{Rem}

	\subsection{The decomposition for zero and non-zero modes}
	
	We decompose the solution into the principal and transversal parts corresponding to the zero and non-zero modes in Fourier space, respectively. For the purpose, we define $\mathbf{D}_0$ and $\mathbf{D}_{\neq}$ by
	\begin{align}\label{def-decom}
		\D_{0} f:=\mathring{f}:=\int_{\Torus^2}f dx_2dx_3,\qquad \D_{\neq} f:=f_{\neq}:=f-\mathring{f},
	\end{align} 
	for any function $f$ integrable on $\Torus^2$. We also introduce the following two notations for the sake of simplicity:
	\begin{align}\notag
		\mathcal{N}_{2}(U):=\frac{\m\otimes \m}{\rho}-(\gamma-1)\frac{|\m|^2}{2\rho}\mathbb{I},\qquad\qquad \mathcal{N}_{3}(U):=\frac{\m\E}{\rho}+p\frac{\m}{\rho}-\frac{\m}{\rho}S,
	\end{align}
	where $\mathbb{I}$ denotes the identity matrix. With these, the Navier-Stokes equations can be rewritten as
	\begin{align}\label{NS-ln}
		\begin{cases}
			&\p_t\rho+\dv\ \m=0,\\
			&\pt \m+(\gamma-1)\nabla\E+\dv\ \mathcal{N}_2(U)=\dv S,\\
			&\pt \E+\dv\ \mathcal{N}_3(U)=\frac{\kappa(\gamma-1)}{R}\Delta(\frac{\E}{\rho}-\frac{|\uf|^2}{2}).
		\end{cases}
	\end{align}
	
	\section{Proof of \cref{mt}: the non-zero initial mass}
	In this section, we are devoted to proving \cref{mt} in case of non-zero initial mass. We first derive the perturbation equations, and then introduce a decomposition of zero and non-zero modes for the conserved quantities. The key part is to obtain the {\it a priori} estimates under suitable {\it a priori} assumptions for the zero and non-zero modes, respectively. Note that the zero and non-zero modes are not completely decoupled. 
	
	\subsection{Systems for perturbations}
	We start to derive the systems of perturbations. Recall
	\begin{align}\label{2025.6.02-1}
		(\phi,\mathbf{\psi},w,\mathbf{\varphi},\zeta):=(\rho-\tilde{\rho},\m-\mtb,\E-\Et,\uf-\utb,\T-\Tt).
	\end{align}
	By \cref{NS,NS-ansatz}, the system of perturbations can be written as 
	\begin{equation}\label{sys-pertur}
		\left\{
		\begin{aligned}
			&\pt\phi+ \uf \cdot\nabla\phi+\rho \dv\varphi=\hat{R}_1,\\
			&\rho\pt\varphi+\rho \uf \cdot\nabla\varphi+(\gamma-1)(\T\nabla\phi+\rho\nabla\zeta)=\mu \triangle\varphi+(\mu+\lambda)  \nabla \dv\varphi+\hat{\mathbf{R}}_2,\\
			&\rho\pt\zeta+\rho \uf \cdot\nabla\zeta+ (\gamma-1)\rho\T \dv\varphi=\kappa \triangle\zeta+\hat{R}_3,
		\end{aligned}
		\right.
	\end{equation}
	where 
	\begin{equation}\label{2025.6.02-8}
		\begin{aligned}		\hat{R}_{1}:=&-\px\tilde{H}_1-\varphi\cdot\nabla\tilde{\rho}-\dv\utb\phi,\\
			\hat{\mathbf{R}}_{2}:=&-\rho\mathbf{\varphi}\cdot\nabla\mathbf{\tilde{u}}+ R\nabla\tilde{\rho}(\frac{\tilde{\T}}{\tilde{\rho}}\phi-\zeta)-\frac{\phi}{\tilde{\rho}}(\mu\lap\mathbf{\tilde{u}}+(\mu+\lambda)\nabla\dv\mathbf{\tilde{u}})-\frac{\rho}{\tilde{\rho}}\big(\px\tilde{\mathbf{H}}_2-\utb(\px \tilde{H}_1)\big),\\
			\hat{R}_3:=&\frac{\mu}{2} \left|\nabla\uf+(\nabla\uf)^{t}\right|^2+\lambda(\dv\uf)^2-\frac{\mu}{2} \left|\nabla\utb+(\nabla\utb)^{t}\right|^2-\lambda(\dv\utb)^2\\
			&-\rho\varphi\cdot\nabla\Tt- (\gamma-1)\rho\dv\utb\zeta-\frac{\phi}{\tilde{\rho}}\left[\kappa \Delta\tilde{\T}+\mu\frac{\abs{\nabla\utb+(\nabla\utb)^t}^2}{2}+\lambda\abs{\dv\utb}^2\right]\\
			&-\frac{\rho}{\tilde{\rho}}\bigg(\dv\tilde{\mathbf{H}}_3-\utb\cdot\dv{\tilde{\mathbf{H}}}_2-\big(\frac{\abs{\utb}^2}{2}+\Tt\big)\px\tilde{H}_1\bigg).
		\end{aligned}
	\end{equation}
	
	We further study the perturbations for the zero mode of $(\rho,m,\E)$. The anti-derivatives for the zero mode is crucial in our analysis, and thus we denote
	\begin{align}\label{pert-anti}
		\Phi:=\int_{-\infty}^{x_1}\mr{\phi}dy_1,\quad\Psi:=\int_{-\infty}^{x_1}\mr{\psi}{dy_1},\quad W:=\int_{-\infty}^{x_1}\mr{w}{dy_1}.\quad
	\end{align}
	Then, combining $\Do$\cref{NS-ln} and \cref{NS-ansatz}, we have
	\begin{align}\label{eq-pert1}
		\left\{\begin{array}{l}
			\pt\Phi+\px\Psi_{1}=\mathcal{Q}_1, \\ 
			\pt\Psi_{1}+(\gamma-1)\px W=\left(2\mu+\lambda\right)\px \left(\mathring{u}_{1}- \ut_{1}\right)+\mathcal{Q}_{21}, \\ 
			\pt\Psi_{i}=\mu \px\left(\mathring{u}_{i}- \ut_{i}\right)+\mathcal{Q}_{2i},\quad i=2,3, \\ 
			\pt W+\gamma\tilde{\theta}\px\Psi_1= \kappa \px\zetam + \mathcal{Q}_{3},\end{array}\right.
	\end{align}
	where
	\begin{align}\label{2025.6.02-3}
		\begin{aligned}
			\mathcal{P}(U):=&\frac{m_1^2}{\rho}-(\gamma-1)\frac{\abs{\mathbf{m}}^2}{2\rho}\qquad\quad\mathcal{Q}_1:=-{\tilde{H}}_1,\\
			\mathcal{Q}_{21}:=&\bigg(-\mathring{\mathcal{P}}(U)+\mr{\mathcal{P}}(\tilde{U})\bigg)-\tilde{\mathbf H}_{21}:=\mathcal{Q}_{21}^{(1)}+\mathcal{Q}_{21}^{(2)},\\
			\mathcal{Q}_{2i}:=&\bigg(-\D_0(\frac{{m}_{1} {m}_{i}}{{\rho}})+\D_0(\frac{{\mt}_{1} {\mt}_{i}}{{\rhot}})\bigg)-{\tilde{\mathbf H}}_{2i}:=\mathcal{Q}_{2i}^{(1)}+\mathcal{Q}_{2i}^{(2)},\quad i=2,3,\\
			\mathcal{Q}_{3}:=&\bigg({\gamma}\tilde{\theta}\px\Psi_1-\mr{\mathcal{N}}_{31}(U)+\mr{\mathcal{N}}_{31}(\tilde{U})\bigg)-{\tilde{\mathbf H}}_3:=\mathcal{Q}_{3}^{(1)}+\mathcal{Q}_{3}^{(2)}.
		\end{aligned}
	\end{align}
	
	\subsection{Proof of \cref{mt}}
	\begin{flushleft}
		\textbf{(i) Local existence theorem.}
	\end{flushleft}
	Let us start from studying the local existence. We first define the solution space as
	\begin{align*}
		&X_{m,M}(T):=\sup_{0\leq t\leq T}\bigg\{({\phi},{\varphi},{\zeta},\Phi,\Psi,W)\bigg|	({\phi},{\varphi},{\zeta}) \in C\left(0,+\infty ; H^{3}(\Omega)\right),
		(\phi,\nabla{\varphi},\nabla{\zeta}) \in L^{2}\left(0,+\infty ; H^3(\Omega)\right),\nonumber\\
		&\quad\quad\quad\quad\quad\quad\quad\quad\quad\quad\quad\quad\quad (\Phi,\Psi,W)\in C\left(0,+\infty ;H^2(\R)\right),\ \ (\px\Phi,\px\Psi,\px W)\in L^2\left(0,+\infty ;H^2(\R)\right),\nonumber\\
		&\quad\quad\quad\quad\quad\quad\quad\quad\quad\quad\quad\quad\quad \norm{(\Phi,\Psi,W)}_{H^2(\R)}+\left\|({\phi},{\varphi},{\zeta})\right\|_{H^3(\Omega)}\leq M,\ \inf_{x,t} ({\phi}+\bar{\rho})\geq m>0\bigg\}.
	\end{align*}
	
	\begin{Prop}[Local existence]\label{Thm-local}
		There exist constants $b$, $M$, $m>0$ such that if $$\|({\Phi}_0,{\Psi}_0,{W}_0)\|_{H^3}+\left\|({\phi}_0,{\varphi}_0,{\zeta}_0)\right\|_{H^3}+\norm{(\phi_0,\psi_0,w_0)}_{L^1}\leq M\ \text{and}\ \inf_{x} ({\phi}_0+\bar{\rho}(x_1,0))\geq m>0,$$
		then there exists a positive time $T_0=T_0(m,M)>0$ such that \cref{NS} admits a unique solution $(\rho,u,\theta)$ satisfying
		\begin{align}\notag
			({\phi},{\varphi},{\zeta},{\Phi},{\Psi},{W})\in X_{\frac{1}{2}m,bM}(T_0),
		\end{align}
		as well as
		\begin{align}\notag
			\|({\Phi},{\Psi},{W})\|_{L^\infty}\leq bM.
		\end{align}
	\end{Prop}
	
	\begin{flushleft}
		\textbf{(ii) {The {\it a priori} assumptions.}}
	\end{flushleft}

	Recall that we have decomposed the solution $(\rho,u,\theta)$ into the principal and transversal parts. In what follows we impose the {\it a priori} assumptions on the principal and transversal parts separately. That is, for a positive $c>0$,
	\begin{align}\label{apa}
		\left\{\begin{aligned}
			&\sup_{0\leq t\leq T}\Big\{\left\|(\Phi,\Psi, W)\right\|^2_{L^\infty}+(1+t)^{\frac{1}{2}}\norm{(\phi,\varphi,\zeta)}_{L^2}^2+(1+t)^{\frac{3}{2}}\norm{(\nabla\phi,\nabla\varphi,\nabla\zeta)}_{L^2}^2+\left\|(\phi,\varphi,\zeta)\right\|_{H^3}^2\Big\}\leq \chi^2,\\
			&\sup_{0\leq t\leq T}e^{ct}\norm{\phi_{\neq},\varphi_{\neq},\zeta_{\neq}}^2_{H^1}\leq \varepsilon^2+\delta.
		\end{aligned}\right.
	\end{align} 
	Note that we need to close the above {\it a priori} assumptions. Indeed, we are able to prove the following {\it a priori} estimates.
	
	\begin{Prop}[{\it a priori} estimates]\label{Thm-ape}
		Assume that $(\phi,\varphi,\zeta)$ is the unique solution given in \cref{Thm-local} and satisfies the {\it a priori} assumptions \cref{apa}. Then the following estimates hold
		\begin{align*}
			&\left\|({\Phi},{\Psi},{W})\right\|_{L^\infty}^2+(1+t)^{\frac{1}{2}}\left\|({\phi},{\varphi},{\zeta})\right\|_{L^2}^2+(1+t)^{\frac{3}{2}}\left\|\px({\phi},{\varphi},{\zeta})\right\|_{L^2}^2+\sum_{i=2}^{3}\left\|\px^{i}({\phi},{\varphi},{\zeta})\right\|_{L^2}^2\leq C(\delta+\varepsilon),\\
			&\qquad\left\|(\phi_{\neq},\varphi_{\neq},\zeta_{\neq})\right\|_{H^1}^2\leq \varepsilon^2e^{-c(1+t)},
		\end{align*}
		where $C,c>0$ are the universal constants independent of any small parameters throughout the paper.
	\end{Prop}
	By \cref{Thm-ape} and \cref{Thm-local}, using the standard continuity argument, we obtain \cref{mt}. For simplicity, we omit the proof of \cref{Thm-local}. In what follows, we focus only on the proof of \cref{Thm-ape}. 
	
	\medskip
	\begin{flushleft}
		\textbf{(iii) The decay rate.}
	\end{flushleft}
	By the definition of zero mode and non-zero mode decomposition, one has
	\begin{align}\notag
		\|(\phi,\varphi,\zeta)\|_{L^{\infty}}\leq C\|(\phi_{\neq},\varphi_{\neq},\zeta_{\neq})\|_{L^\infty}+C\|(\phim,\varphim,\zetam)\|_{L^\infty}.
	\end{align}
	Furthermore, we have
	\begin{align}\notag
		\|(\phim,\varphim,\zetam)\|_{L^\infty}\leq C\|(\phim,\varphim,\zetam)\|^{\frac{1}{2}}_{L^2}\|\px(\phim,\varphim,\zetam)\|^{\frac{1}{2}}_{L^2}\leq C(\delta^{\frac{1}{2}}+\varepsilon_0^{\frac{1}{2}})(1+t)^{-\frac{1}{2}},
	\end{align}
	and
	\begin{align*}
		\|(\phi_{\neq},\varphi_{\neq},\zeta_{\neq})\|_{L^{\infty}}\leq& C\|(\phi_{\neq},\varphi_{\neq},\zeta_{\neq})\|^{\frac{1}{4}}_{L^2}\|\nabla^2(\phi_{\neq},\varphi_{\neq},\zeta_{\neq})\|^{\frac{3}{4}}_{L^2}
		\leq C\varepsilon_0 e^{-ct},
	\end{align*}
	where $C$ and $c$ are positive constants. Thus
	\begin{align}\notag
		\|(\phi,\varphi,\zeta)\|_{L^{\infty}(\Omega)}\leq C(\delta^{\frac{1}{2}}+\varepsilon_0^{\frac{1}{2}})(1+t)^{-\frac{1}{2}}.
	\end{align}
	We still need to study the differences between the ansatz $(\rhot,\ut,\thetat)$ and the entropy wave $(\rhob,\ub, \thetab)$. By \cref{tilde},
	we have
	\begin{align}\notag
		\|(\tilde{\m}-\bar{\m},\tilde{\E}-\bar{\E})\|_{L^{\infty}(\Omega)}
		\leq C(\delta+\varepsilon_0) (1+t)^{-\frac{1}{2}}.
	\end{align}
	Direct calculation shows
	\begin{align}\notag
		\big| (\tilde{\T}-\bar{\T})\big|=&\bigg|\frac{\Et}{\rhot}-\frac{|\tilde{\m}|^2}{|\rhot|^2}-\frac{\bar{\E}}{\bar{\rho}}+\frac{|\bar{\m}|^2}{|\bar{\rho}|^2}\bigg|\leq C(\delta+\varepsilon_0)D_{-\frac{1}{2}}.
	\end{align}
	Similarly
	\begin{align}\notag
		|\utb-\bar{\uf}|+|\rhot-\bar{\rho}|\leq C(\delta+\varepsilon_0)D_{-\frac{1}{2}}.
	\end{align}
	Then \cref{mt-decay} has been proved.
	
	In the next section, we are going to give the proof of \cref{Thm-ape}.
	
	\subsection{Estimates on zero modes}\label{EZM}
	
	One of the key points in our proof is to introduce the following transformation
	\begin{align}\label{transformation}
		\begin{aligned}
			&\tilde{\Phi}=\Tt\Phi,\qquad\tilde{\Psi}={\Psi}- \mathbf{\ut} \Phi,\\
			&\tilde{W}={W}-\mathbf{{\utb}} \cdot \tilde{\Psi}-\left({\thetat}+\frac{| \mathbf{\ut}|^{2}}{2}\right) \Phi.
		\end{aligned}
	\end{align}
	
	In what follows, we present the relationship between different variables more clearly. For the first-order derivatives, we have
	\begin{align}\label{2025.6.02-2}
		\begin{aligned}
			&\px\Psi_{i}=\px\tilde{\Psi}_{i}+ \ut_{i} \px\Phi+ \px\ut_{i} \Phi:=\px\tilde{\Psi}_{i}+\mathcal{B}_{ir},\qquad i=1,2,3,\\
			&\px W=\px\tilde{W}+\px\left({\thetat}\Phi\right) +\px\left(\mathbf{{\utb}} \cdot {\tilde\Psi}+\frac{|\mathbf{{\utb}}|^{2}}{2} \Phi\right):=\px\tilde{W}+\px\tilde{\Phi}+\mathcal{B}_{wr}.
		\end{aligned}
	\end{align}
	
	By direct calculations, one has
	\begin{align}\label{id-wd}
		\begin{aligned}
			\pt\Psi_{i}=&\pt\tilde{\Psi}_{i}- \ut_{i}\px\Psi_{1}+\tilde u_i\mathcal Q_1 + \pt\ut_{i} \Phi:=\pt\tilde{\Psi}_i+\mathcal{B}_{id},\\
			\pt W=& \pt\tilde{W}-\left(\Tt+\frac{| \mathbf{\ut}|^{2}}{2}\right)\bigg(\px\Psi_{1}-\mathcal{Q}_1\bigg)+ \tilde{\uf} \cdot \pt\tilde{\Psi}+\pt\tilde{\uf}\cdot\tilde{\Psi}+\pt\left( {\thetat}+\frac{| \mathbf{\ut}|^{2}}{2}\right) \Phi\\
			:=&\pt\tilde{W}-\Tt\px\tilde{\Psi}_{1}+\mathcal{B}_{wd}.
		\end{aligned}
	\end{align}

	Then, by \cref{eq-pert1} and \cref{transformation}-\cref{id-wd}, we have the following system for $(\tilde{\Phi}, \tilde{\Psi}, \tilde{W})$:
	\begin{align}\label{sys-Eu0-1}
		\left\{\begin{array}{l}
			\pt\tilde{\Phi}+\Tt \px\tilde{\Psi}_{1}=J_1+\Tt\mathcal{Q}_{1}, \\
			\pt\tilde{\Psi}_{1}+ (\gamma-1)\px\left(\tilde{W}+\tilde{\Phi}\right)
			=\frac{(2\mu+\lambda)}{\rhot}  \px^2\tilde{\Psi}_{1}+\mathcal{Q}_{21}+J_{21}, \\
			\pt\tilde{\Psi}_{i}=\frac{\mu}{\rhot}  \px^2\tilde{\Psi}_{i} 
			+J_{2i}+\mathcal{Q}_{2i}, i=2,3, \\
			\pt\tilde{W}+ (\gamma-1)  {\thetat} \px\tilde{\Psi}_{1}=\frac{\kappa}{\rhot}  \px^2\tilde{W} 
			+J_{3}+\mathcal{Q}_{3},
		\end{array}\right.
	\end{align}
	where
	\begin{align*}
		J_1&:=\pt{\Tt}\Phi-\Tt\mathcal{B}_{1r},\qquad J_{2i}:=-\mathcal{B}_{id}+\left[\mu\px(\um_i-\tilde{u}_i)-\frac{\mu}{\rhot}\px^2\tilde{\Psi}_i\right]:=J_{2i}^{(1)}+J_{2i}^{(2)},\qquad i=2,3,\\
		J_{21}&:=\left[-\mathcal{B}_{1d}-(\gamma-1)\mathcal{B}_{wr}\right]+\left[(2\mu+\lambda)\px(\mr{u}_1-\tilde{u}_1)-\frac{(2\mu+\lambda)}{\rhot}\px^2\tilde{\Psi}_1\right]:=J_{21}^{(1)}+J_{21}^{(2)},\\
		J_{3}&:=\left[-\mathcal{B}_{wd}-\gamma\Tt\mathcal{B}_{1r}\right]+\left[\kappa \px\zetam-\frac{\kappa}{\tilde{\rho}}\px^2\tilde{W}\right]:=J_{3}^{(1)}+J_{3}^{(2)}.
	\end{align*}
	Before deducing the {\it a priori} estimates, we should study the source terms in \cref{sys-Eu0-1}. We use the following notations for the sake of convenience:
	\begin{align}\label{useful-notation}
		&{V}:=(\tilde{\Phi},\tilde{\Psi},\tilde{W})^{t},\qquad\qquad v:=(\phi,\varphi,\zeta)^{t},\qquad\qquad \tilde{V}:=(\tilde{\Phi},\tilde{\Psi}_1,\tilde{W})^{t},\notag\\
		&\mathcal{D}^{(k)}:=(\delta+\varepsilon) \sum_{j=1}^{k+2} D_{-\frac{j}{2}} \abs{\px^{k-j+2} V},\qquad 
		\mathcal{T}^{(0)}:=\abs{\px V}^2+ (\delta+\varepsilon)  D_{-\frac{1}{2}}\abs{\px V},\notag\\
		&\mathcal{T}^{(1)}:=\abs{\px^2 V}\abs{\px V}+(\delta+\varepsilon) \left(D_{-\frac{1}{2}}\abs{\px V}^2+D_{-1}\abs{\px V}+D_{-\frac{1}{2}}\abs{\px^2 V}\right),\notag\\
		&\mathcal{T}^{(2)}:=\abs{\px^3V}\abs{\px V}+\abs{\px^2V}^2+ (\delta+\varepsilon) D_{-\frac{3}{2}}\abs{\px V}\\
		&\qquad\qquad +(\delta+\varepsilon) D_{-\frac{1}{2}}\big(\abs{\px V}\abs{\px^2 V}+\abs{\px^3V}\big)+ (\delta+\varepsilon) D_{-1}\big(\abs{\px^2V}+\abs{\px V}^2\big),\notag\\
		&\mathcal{Z}^{(0)}:=\abs{v_{\neq}}^2,\qquad\qquad\mathcal{Z}^{(1)}:=\abs{v_{\neq}}\abs{\nabla v_{\neq}},\notag\\
		&\mathcal{Z}^{(2)}:=\abs{\nabla v_{\neq}}^2+\abs{\nabla^2 v_{\neq}}\abs{v_{\neq}},\qquad\mathcal{Z}^{(3)}:=\abs{\nabla v_{\neq}}\abs{\nabla^2 v_{\neq}}+\abs{\nabla^3 v_{\neq}}\abs{v_{\neq}}.\notag
	\end{align}
	Then we have the following result.
	\begin{Lem}\label{mathcal-Z}
		Under the same assumptions as \cref{Thm-ape}, one has
		\begin{align*}
			\norm{\mathcal{Z}^{(i)}}_{L^2}\leq C (\varepsilon^2+\delta) e^{-\frac{c}{4}t},\quad \text{for} \quad i=1,2,3.
		\end{align*}
	\end{Lem}
	\begin{proof}
		By the {\it a priori} assumption \cref{apa} and {Gagliardo-Nirenberg} inequality, one has
		\begin{align*}
			\norm{\mathcal{Z}^{(3)}}_{L^2}\lesssim\left( \norm{\nabla v_{\neq}}_{L^\infty}+\norm{v_{\neq}}_{L^\infty}\right)\norm{\nabla^3 v_{\neq}}_{L^2}\lesssim\norm{\nabla v_{\neq}}_{L^2}^{\frac14}\norm{\nabla^3 v_{\neq}}_{L^2}^{\frac74}\leq C(\varepsilon^2+\delta) e^{-\frac{c}{4}t}.
		\end{align*}
		Moreover, $\norm{\mathcal{Z}^{(1)}}_{L^2}$ and $\norm{\mathcal{Z}^{(2)}}_{L^2}$ can be treated in the same way. Then we have completed the proof of \cref{mathcal-Z}.
	\end{proof}
	
	\begin{Lem}\label{mathcal-Q}
		Under the same assumptions as \cref{Thm-ape}, one has
		\begin{align*}
			\begin{aligned}
				&\abs{\px^k\mathcal{Q}_{2i}^{(1)}}\leq C\left(\mathcal{T}^{(k)}+\mathcal{Z}^{(k)}\right),\quad k=0,1,2,\ \ i=1,2,3,\\
				&\abs{\px^k\mathcal{Q}_3^{(1)}}\leq C\big(\mathcal{T}^{(k)}+\mathcal{T}^{(k+1)}+\mathcal{Z}^{(k)}+\mathcal{Z}^{(k+1)}\big),\quad k=0,1,\\
				&\abs{\px^k\mathcal{Q}_{2i}^{(2)}}+\abs{\px^k\mathcal{Q}_3^{(2)}}\le C(\delta+\varepsilon) D_{-\frac{2+k}{2}}\quad k=0,1,2,\ \ i=1,2,3,
			\end{aligned}
		\end{align*} 
		where $\mathcal{Q}_{2,3}$ is defined in \cref{2025.6.02-3}.
	\end{Lem}
	\begin{proof}
		We first study $\mathcal{Q}_{21}^{(1)}$. Note
		\begin{align*}
			\begin{aligned}
				-\mr{\mathcal{P}}(U)+\mr{\mathcal{P}}(\tilde{U})=-\bigg({\mathcal{P}}(\mr{U})-\mathcal{P}(\tilde{U})\bigg)+\bigg({\mathcal{P}}(\mr{U})-\mr{\mathcal{P}}(U)\bigg).
			\end{aligned}
		\end{align*}
		Direct calculations yield
		\begin{align}\notag
			{\mathcal{P}}(\mr{U})-\mathcal{P}(\tilde{U})-\nabla\mathcal{P}(\tilde{U})\left(\mr{U}-\tilde{U}\right)\leq C\left[\abs{\px\Psi}^2+\abs{\px\Phi}^2\right],
		\end{align}
		where
		\begin{align*}
			\nabla\mathcal{P}(\tilde{U})\left(\mr{U}-\tilde{U}\right)=(3-\gamma)\left[\frac{\tilde{m}_1}{\tilde{\rho}}\px\Psi_1-\frac{\tilde{m}_1^2}{2\tilde{\rho}^2}\px\Phi\right]-(\gamma-1)\sum_{i=2,3}\left[\frac{\tilde{m}_i}{\rhot}\px\Psi_i-\frac{\tilde{m}_i^2}{2\tilde{\rho}^2}\px\Phi\right].
		\end{align*}
		Next
		\begin{align}\notag
			\mathcal{P}(\mathring{U})-\mathring{ \mathcal{P}}(U)=\Do\left[\frac{\mm_1^2}{\rhom}-(\gamma-1)\frac{\abs{\mathbf{\mm}}^2}{2\rhom}-\frac{m_1^2}{\rho}+(\gamma-1)\frac{\abs{\mathbf{m}}^2}{2\rho}\right]\leq C \abs{v_{\neq}}^2.
		\end{align}
		Thus
		\begin{align}\notag
			\abs{{\mr{\mathcal{P}}}(U)-\mathcal{P}(\tilde{U})}\leq C\abs{\px V}^2+C(\delta+\varepsilon) D_{-\frac{1}{2}}\abs{\px V}+C\abs{v_{\neq}}^2.
		\end{align}
		Similarly, for $\mathcal{Q}_{22}^{(1)}$, $\mathcal{Q}_{23}^{(1)}$ and $\mathcal{Q}_{3}^{(1)}$, one has
		\begin{align*}
			&\abs{-\D_0(\frac{{m}_{1} {m}_{i}}{{\rho}})+\D_0(\frac{{\mt}_{1} {\mt}_{i}}{{\rhot}})}\leq C\abs{\px V}^2+C(\delta+\varepsilon) D_{-\frac{1}{2}}\abs{\px V}+C\abs{v_{\neq}}^2,\\
			&\abs{\gamma\tilde{\theta}\px\Psi_1-\mr{\mathcal{N}}_{31}(U)+\mr{\mathcal{N}}_{31}(\tilde{U})}\leq C(\abs{\px V}^2+\abs{\px^2 V}\abs{\px V})+C(\delta+\varepsilon) D_{-\frac{1}{2}}\abs{\px V}\\
			&\qquad\qquad\qquad+C(\delta+\varepsilon) \bigg(D_{-\frac{1}{2}}\abs{\px V}^2+D_{-1}\abs{\px V}+D_{-\frac{1}{2}}\abs{\px^2 V}\bigg)+C\abs{v_{\neq}}\abs{\nabla v_{\neq}},
		\end{align*}
		where we have used
		\begin{align*}
			&\abs{\D_{0}\big[\frac{m_i}{\rho}\px(\frac{m_i}{\rho})\big]-	\frac{\tilde{m}_i}{\rho}\px(\frac{\tilde{m}_i}{\rho})}\\
			\leq& C\abs{\px^2 V}\abs{\px V}+C(\delta+\varepsilon) \bigg(D_{-\frac{1}{2}}\abs{\px V}^2+D_{-1}\abs{\px V}+D_{-\frac{1}{2}}\abs{\px^2 V}\bigg)+C\abs{v_{\neq}}\abs{\nabla v_{\neq}}.
		\end{align*}
		Furthermore, it is direct to have
		\begin{align*}
			&\abs{\px\bigg[{\mathcal{P}}(\mr{U})-\mathcal{P}(\tilde{U})\bigg]}+\abs{\px\bigg[-\D_0(\frac{{m}_{1} {m}_{i}}{{\rho}})+\D_0(\frac{{\mt}_{1} {\mt}_{i}}{{\rhot}})\bigg]}\\
			\leq& C\abs{\px^2 V}\abs{\px V}+C(\delta+\varepsilon) \bigg(D_{-\frac{1}{2}}\abs{\px V}^2+D_{-1}\abs{\px V}+D_{-\frac{1}{2}}\abs{\px^2 V}\bigg)+C\abs{\nabla v_{\neq}}\abs{v_{\neq}},\\
			&\abs{\px^2\bigg[{\mathcal{P}}(\mr{U})-\mathcal{P}(\tilde{U})\bigg]}+\abs{\px^2\bigg[-\D_0(\frac{{m}_{1} {m}_{i}}{{\rho}})+\D_0(\frac{{\mt}_{1} {\mt}_{i}}{{\rhot}})\bigg]}\\
			\leq& C\left(\abs{\px^3 V}\abs{\px V}+\abs{\px^2 V}^2\right)+ C(\delta+\varepsilon) D_{-\frac{3}{2}}\abs{\px V}+ C(\delta+\varepsilon) D_{-1}\big(\abs{\px^2V}+\abs{\px V}^2\big)\\
			& +C(\delta+\varepsilon) D_{-\frac{1}{2}}\big(\abs{\px V}\abs{\px^2 V}+\abs{\px^3V}\big)+C\left(\abs{\nabla v_{\neq}}^2+\abs{\nabla^2 v_{\neq}}\abs{v_{\neq}} \right),\\
			&\abs{\px\bigg[\gamma\tilde{\theta}\px\Psi_1-\mr{\mathcal{N}}_{31}(U)+\mr{\mathcal{N}}_{31}(\tilde{U})\bigg]}\\
			\leq&C\left(\abs{\px^2V}\abs{\px V}+\abs{\px^3V}\abs{\px V}+\abs{\px^2V}^2\right)+C(\delta+\varepsilon) D_{-\frac{1}{2}}\big(\abs{\px V}\abs{\px^2 V}+\abs{\px^2V}+\abs{\px^3V}\big)\\
			&+C(\delta+\varepsilon) D_{-1}\big(\abs{\px V}+\abs{\px^2V}+\abs{\px V}^2\big)+C(\delta+\varepsilon) D_{-\frac{3}{2}}\abs{\px V}+C\left(\abs{\nabla v_{\neq}}^2+\abs{\nabla^2 v_{\neq}}\abs{v_{\neq}} \right).
		\end{align*}
		And combining \cref{tildeE},  we then have completed the proof of \cref{mathcal-Q}.
	\end{proof}
	
	\begin{Lem}\label{J}
		Under the same assumptions as \cref{Thm-ape}, one has
		\begin{align*}
			\begin{aligned}
				&\abs{\px^{k}J_1,\px^kJ^{(1)}_{2i}}\leq C\mathcal{D}^{(k)},\ \ k=0,1,2,\ \ i=1,2,3,\\
				&\abs{\px^{k}J_{2i}^{(2)},\px^{k}J_{3}^{(2)}}	\leq C\big(\mathcal{D}^{(k+1)}+\mathcal{T}^{(k+1)}+\mathcal{Z}^{(k+1)}\big),\ \ k=0,1,\ \ i=1,2,3,\\
				&\abs{\px^{k}J_{3}^{(1)}}	\leq C(\delta+\varepsilon) D_{-\frac{k+2}{2}}+C\big(\mathcal{D}^{(k)}+\mathcal{D}^{(k+1)}+\mathcal{T}^{(k+1)}+\mathcal{Z}^{(k+1)}\big),\ \ k=0,1.
			\end{aligned}
		\end{align*} 
	\end{Lem}
	
	\begin{proof}
		We take $J_3$ as an example to give the detailed computations; $J_1$ and $J_2$ can be treated in the same way. We first calculate $J_{3}^{(2)}$ through which one can see that the transformation \cref{transformation} plays an important role in the viscosity part. Indeed, we have
		\begin{align}\label{lem-2-1-3}
			&\kappa\left(\px\zetam-\frac{\px^2 \tilde{W}}{\rhot} \right)=\kappa \px\Do\left[ \frac{\E}{\rho}-\frac{\Em}{\rhom}+\frac{\abs{\mm}^2}{2\rhom^2}-\frac{| \mathbf{m}|^{2}}{2\rho^2} \right]-\kappa\px\left[\frac{\abs{\mm}^2}{2\rhom^2}-\frac{| \mathbf{\mt}|^{2}}{2\rhot^2}+\frac{| \mathbf{\ut}|^{2}\phim}{2\rhom} \right]\notag
			\\
			&\qquad\qquad\qquad\qquad+\kappa\left[\px\left( \frac{\mathring{w}}{\rhom}  \right) - \frac{\px^2 W}{\rhot}-\px\left( \frac{\phim \thetat}{\rhom} \right)+\frac{\px^2(\thetat \Phi)}{\rhot}+\frac{\px \mathcal{B}_{wr}}{\rhot}\right]=\mathcal{L}_1+\mathcal{L}_2+\mathcal{L}_3,
		\end{align}
		where we have used
		\begin{align*}
			&\zetam = \Do\left[\frac{\E}{\rho}-\frac{\Et}{\rhot}-\left( \frac{| \mathbf{m}|^{2}}{2\rho^2} - \frac{| \mathbf{\mt}|^{2}}{2\rhot^2}  \right)\right],\qquad \frac{\Em}{\rhom}-\frac{\Et}{\rhot}=\frac{\mathring{w}}{\rhom}-\frac{\phim\thetat}{\rhom}-\frac{| \mathbf{\ut}|^{2}\phim}{2\rhom},\\
			&\frac{\mm_i^2}{\rhom^2}-\frac{\mt_i^2}{\rhot^2}=\frac{\rhot\mm_i+\rhom\mt_i}{\rhom^2\rhot} \psim_i-  \frac{\mt_i(\rhot\mm_i+\rhom\mt_i) }{\rhom^2 \rhot^2}\phim.
		\end{align*}
		It is direct to have
		\begin{align*}
			&\abs{\mathcal{L}_1}\le C \abs{v_{\neq}}\abs{\nabla_x v_{\neq}},\qquad\abs{\mathcal{L}_2}+\abs{\mathcal{L}_3}\leq C (\delta+\varepsilon) D_{-1}\abs{V}+ C (\delta+\varepsilon) D_{-\frac{1}{2}}\left(\abs{\px V}+ \abs{\px^2 V}\right)+C\abs{\px^2 V}\abs{\px V}.
		\end{align*}
		Then we have $\abs{J_{3}^{(2)}}\leq C\big(\mathcal{D}^{(1)}+\mathcal{T}^{(1)}+\mathcal{Z}^{(1)}\big)$.
		
		For  $\abs{J_{3}^{(1)}}$ , one has the following estimates by  \cref{id-wd} and \cref{mathcal-Q}:
		\begin{align*}
			\mathcal{B}_{wd}=&\thetat\mathcal{Q}_{1}-\frac{| \mathbf{\ut}|^{2}}{2}\bigg(\px\Psi_{1}-\mathcal{Q}_1\bigg)+ \tilde{\uf} \cdot \pt\tilde{\Psi}+\pt\tilde{\uf}\cdot\tilde{\Psi}+\pt\left( {\thetat}+\frac{| \mathbf{\ut}|^{2}}{2}\right) \Phi\\
			\leq &C\mathcal{Q}_1+C\mathcal{D}^{(0)}+CD_{-\frac{1}{2}}\left(\px\tilde{W}+\px\tilde{\Phi}+  \px^2\tilde{\Psi}_{1}+\mathcal{Q}_{21}+J_{21}\right)\\
			\leq&C\left(\mathcal{D}^{(0)}+\mathcal{D}^{(1)}+(\delta+\varepsilon) D_{-1}+\mathcal{T}^{(1)}+\mathcal{Z}^{(1)}\right).
		\end{align*}
		The remaining terms can be estimated in a similar way. Then we have proved \cref{J}.
	\end{proof}

	Recall the definition $\tilde{V}=(\tilde{\Phi},\tilde{\Psi}_1,\tilde{W})^{t}$ in \cref{useful-notation}. By the equation \cref{sys-Eu0-1},  $\tilde{V}$ can be expressed as
	\begin{align}\label{2025-11-8-1}
		\pt \tilde{V}+A_1\px \tilde{V}=A_2\px^2\tilde{V}+A_3,
	\end{align}
	where 
	\begin{align*}
		A_1:=\left(\begin{array}{ccc}
			0 & \Tt & 0 \\
			\gamma-1& 0 & \gamma-1 \\
			0 & (\gamma-1)\Tt & 0
		\end{array}\right),
		\qquad
		A_2:= \left(\begin{array}{ccc}
			0 & 0 & 0 \\
			0& \frac{\lambda+2\mu}{\bar{\rho}} & 0 \\
			0& 0 & \frac{\kappa}{\tilde{\rho}}
		\end{array}\right).
	\end{align*}
	and
	\begin{align}
		A_3:=&(A_{31},A_{32},A_{33})^{t}
		=(J_1+\tilde{\theta}\mathcal{Q}_1,J_{21}+\mathcal{Q}_{21},J_{3}+\mathcal{Q}_3)^t\nonumber.
	\end{align}

	Direct computation yields that the eigenvalues of $A_1$ are given by
	\begin{align*}
		\tilde{\lambda}_1=-\sqrt{\gamma(\gamma-1)\Tt},\quad \lambda_2=0,\quad\tilde{\lambda}_3=\sqrt{\gamma(\gamma-1)\Tt},\qquad \Lambda:=\text{diag}\{\tilde{\lambda}_1,0,\tilde{\lambda}_3\},
	\end{align*}
	and the corresponding eigenvectors are
	\begin{align}
		&\tilde{L}:=\left(\begin{array}{ccc}
			\frac{1}{\sqrt{2\gamma}} & \frac{1}{\sqrt{2\gamma}}\frac{\tilde{\lambda}_1}{\gamma-1} &\frac{1}{\sqrt{2\gamma}} \\
			\sqrt{\frac{\gamma-1}{\gamma}} & 0 & -\frac{1}{\sqrt{\gamma(\gamma-1)}}\\
			\frac{1}{\sqrt{2\gamma}} & -\frac{1}{\sqrt{2\gamma}}\frac{\tilde{\lambda}_1}{\gamma-1} &\frac{1}{\sqrt{2\gamma}}
		\end{array}\right),\quad
		&\tilde{R}:=\left(\begin{array}{ccc}
			\frac{1}{\sqrt{2\gamma}} & \sqrt{\frac{\gamma-1}{\gamma}} &\frac{1}{\sqrt{2\gamma}} \\
			\frac{1}{\sqrt{2\gamma}}\frac{\tilde{\lambda}_1}{\tilde{\T}} & 0 & -\frac{1}{\sqrt{2\gamma}}\frac{\tilde{\lambda}_1}{\tilde{\T}}\\
			\frac{\gamma-1}{\sqrt{2\gamma}} & -\sqrt{\frac{\gamma-1}{\gamma}} &\frac{\gamma-1}{\sqrt{2\gamma}}
		\end{array}\right).\label{eq-ever}
	\end{align}
	
	\begin{Rem}
		Direct calculations imply that the two structural conditions in \cref{SC} hold.
	\end{Rem}
	
	Set $B=\tilde{L}\tilde{V}=(b_1,b_2,b_3)^t$, then $\tilde{V}=\tilde{R}B$. By \cref{2025-11-8-1}, we have the diagonalized equations:
	\begin{align}\label{eq-diaB}
		\pt B+\Lambda \px B=\tilde{L}A_2\tilde{R}\px ^2B+2\tilde{L}A_2\px \tilde{R}\px B+\left[\left(\pt\tilde{L}+\Lambda \px \tilde{L}\right)\tilde{R}+\tilde{L}A_2 \px ^2\tilde{R}\right]B+\tilde{L} A_3.
	\end{align}
	Denote
	\begin{align}
		K_i:=&\int_{\R}\abs{\px^{i+1}\tilde{\Phi}}^2+\abs{\px^{i+1}\tilde{\Psi}}^2+\abs{\px^{i+1}\tilde{W}}^2dx_1,\label{sec-n-4-t}\\
		\quad{E}_i:=&\tilde{E}_i+\int_{\R}\abs{\px^{i}\tilde{\Psi}_2}^2+\abs{\px^{i}\tilde{\Psi}_3}^2dx_1+\tilde{c}\int_{\R}\left(\px^{i+1}\tilde{\Phi}\px^i \tilde{\Psi}_{1}+ \frac{\lambda+2\mu}{2 \rhot\Tt} \abs{\px^{i+1}\tilde{\Phi}}^{2}-(\lambda+2\mu)\frac{\px\Phi\abs{\px^{i+1}\tilde{\Phi}}^2}{2\rhom\rhot}\right)dx_1,\label{sec-n-4}
	\end{align} 
	where $\tilde{E}_i$ is defined in \cref{sec-n-3} and $\tilde{c}$ is a sufficiently small constant used to ensure that $E_i$ is positive. Before proving \cref{Thm-ape}, we should go back to see the {\it a priori} assumptions. To be convenient, we use the following notations. For any functions $a$ and $b$, we write $a\approx b$ when
	\begin{align}\notag
		\|a-b\|_{L^2(\Omega)}\leq C\bar{\delta}e^{-ct},\qquad\qquad \bar{\delta}:=\varepsilon+\delta,
	\end{align}
	where $c$ is a positive constant independent of any small constants throughout the paper including $\delta$, $\varepsilon$ and so on. By \cref{apa,2025.6.02-1,2025.6.02-2}, one has 
	\begin{align}\label{123456}
		\begin{aligned}
			\phim=&\px\Phi=\frac{\thetat\px\tilde{\Phi}-\px\Tt\Phi}{\tilde{\theta}^2}\Rightarrow \norm{\px\tilde{\Phi}}_{L^2}^2\leq C\chi^2(1+t)^{-\frac{1}{2}},\\
			\varphim_i\approx&\frac{1}{\rhom}\px\tilde{\Psi}_i+\frac{1}{\rhom}\mathcal{B}_{ir}-\frac{\tilde{m}_i}{\tilde{\rho}\rhom}\px\Phi\Rightarrow\norm{\px\tilde{\Psi}_i}_{L^2}^2\leq C\chi^2(1+t)^{-\frac{1}{2}},\\
			\zetam\approx&\frac{1}{\rhom}\px \tilde{W}+\frac{1}{\rhom}\px \tilde{\Phi}+\frac{1}{\rhom} \mathcal{B}_{wr}-\frac{\Et}{\rhom\rhot}\px\Phi-\left( \frac{\abs{\mm}^2}{\rhom^2}-\frac{\abs{\mt}^2}{\rhot^2}\right)\Rightarrow\norm{\px\tilde{W}}\leq C\chi^2(1+t)^{-\frac{1}{2}}.
		\end{aligned}
	\end{align}
	Similarly, recalling the definition $V=(\tilde{\Phi},\tilde{\Psi},\tilde{W})$ in \cref{useful-notation}, one has
	\begin{align}\notag
		\norm{\px(\phim,\varphim_i,\zetam) }_{L^2}^2\leq C\chi^2(1+t)^{-\frac{3}{2}}\Rightarrow \norm{\px^2 V}_{L^2}^2\leq C\chi^2(1+t)^{-\frac{3}{2}}.
	\end{align}

	Then, the following result is obtained.
	
	\begin{Lem}\label{nzm}
		Under the same assumptions as \cref{Thm-ape}, one has
		\begin{align}
			&\frac{d}{dt}\big(\sum_{i=0}^2E_i\big)+\sum_{i=0}^2(K_i+G_{i})\leq C\etab(1+t)^{-1}\big(\sum_{i=0}^2{E}_i\big)+C{\bar{\delta}}(1+t)^{-\frac{1}{2}},\label{11-1}\\
			&\frac{d}{dt}\big(\sum_{i=1}^2{E}_i\big)+\sum_{i=1}^2(K_i+G_{i})\leq C\etab(1+t)^{-1}\big(\sum_{i=1}^2{E}_i+G_0\big)+C\etab(1+t)^{-2}{E}_0+C\bar{\delta}(1+t)^{-\frac{3}{2}},\label{22-1}\\
			&\frac{d}{dt}{E}_2+{K}_2+G_{2}\leq C\etab\sum_{i=0}^2(1+t)^{-(3-i)}{E}_{i}+C\etab\big[(1+t)^{-2}G_0+(1+t)^{-1}G_1\big]+C\bar{\delta}(1+t)^{-\frac{5}{2}},\label{33-1}
		\end{align}
		where $G_i$ ($i=1,2,3$) are defined in \cref{sec-n-3-tt} and $\etab:=\chi+\deltab^{\frac{1}{2}}$.
	\end{Lem}
	
	\begin{proof}
		Applying $\px^k$ ($k=0,1,2$) to \cref{eq-diaB}, one has
		\begin{align}\label{equ-Bk}
			\begin{aligned}
				\px^kB_t+\Lambda \px^{k+1}B=\tilde{L} A_2 \tilde{R} \px^{k+2}B+\mathcal{M}_k,
			\end{aligned}
		\end{align}
		where
		\begin{align*}
			\mathcal{M}_k:=&\sum_{j=1}^{k}\bigg[-\px^j\Lambda\px^{k-j+1}B+\px^j(\tilde{L}A_2\tilde{R})\px^{k-j+2}B\bigg]+2\sum_{i=0}^{k}\px^i\big(\tilde{L} A_2 \px\tilde{R}\big) \px^{k-i+1}B\\
			&+\sum_{i=0}^{k}\bigg\{\px^i\big[\left(\tilde{L}_t+\Lambda \px\tilde{L}\right) \tilde{R}+\tilde{L} A_2 \tilde{R}_{x_1 x_1}\big] \px^{k-i}B \bigg\}+\px^{k}\left(\tilde{L} A_3\right)\notag\\
			\leq&C\sum_{j=1}^{k+1}\left(\abs{\px^j \rhob} \abs{\px^{k-j+1}b_1},0,\abs{\px^j\rhob} \abs{\px^{k-j+1}b_3}\right)^{t}
			+C\bar{\delta} \sum_{j=1}^{k+2}\abs{D_{-\frac{j}{2}}\px^{k-j+2}B}+\px^{k}\left(\tilde{L} A_3\right):=\sum_{i=1}^{3}\mathcal{M}_{k}^{(i)}.\notag
		\end{align*}
		\begin{flushleft}
			\textbf{Step 1. The diagonalized system.}
		\end{flushleft}
		
		We shall use the weighted energy method to derive the intrinsic dissipation. Without loss of generality, we assume that $\px{\rhob}>0 $ since the case that $\px{\rhob}<0$ is similar.
		
		Setting $v_{1}=\frac{\rhob}{\rho_{+}}$, one has $\left|v_{1}-1\right| \leq C \delta $. Applying $\px^k$ to \cref{eq-diaB} and then multiplying the resulted equations by $\bar{B}^{(k)}=\left(v_{1}^{n} \px^kb_{1}, \px^kb_{2}, v_{1}^{-n} \px^kb_{3}\right)$, where $n=4[\delta^{-\frac{1}{2}}]+1$ is a large positive integer, one has
		\begin{align}\label{sec-3-1}
			\begin{aligned}
				&\int_\R\left(\frac{v_1^n}{2} \abs{\px^kb_1}^2+\frac{1}{2} \abs{\px^kb_2}^2+\frac{v_1^{-n}}{2} \abs{\px^kb_3}^2\right)_t+\bar{B}^{(k)}_{x_1} A_4 \px^{k+1}Bdx_1
				+\int_{\R} a_1\abs{\px^kb_1}^2+ a_2\abs{\px^kb_3}^2dx_1\\
				=&\int_{\R}\bar{B}^{(k)}\px A_{4}\px^{k}\px B+\bigg[\left(\frac{v_1^n}{2}\right)_t \abs{\px^kb_1}^2+\left(\frac{v_1^{-n}}{2}\right)_t \abs{\px^kb_3}^2\bigg]+\bar{B}^{(k)}\mathcal{M}_kdx_1:=I_1+I_2+I_3,
			\end{aligned}
		\end{align}
		where $A_{4}=\tilde{L} A_{2} \tilde{R}$ and
		\begin{align*}
			&a_1:=-\frac{v_1^{n-1}}{2}\left(n \tilde\lambda_1 \px v_{1}+v_1\px  \tilde\lambda_{1 }\right)\ge C\delta ^{-\frac{1}{2}}\px \rhob-C\bar{\delta}D_{-1},\\
			&a_2:=\frac{v_1^{-n-1}}{2}\left(n \tilde\lambda_3 \px  v_{1 }-v_1 \px \tilde\lambda_{3 }\right)\ge C\delta ^{-\frac{1}{2}}\px \rhob-C\bar{\delta}D_{-1}.
		\end{align*} 
		For convenience, we denote
		\begin{align}
			&\tilde{E}_k:=\int_{\R}\frac{v_1^n}{2} \abs{\px^kb_1}^2+\frac{1}{2} \abs{\px^kb_2}^2+\frac{v_1^{-n}}{2} \abs{\px^kb_3}^2dx_1,\label{sec-n-3}\\
			&\tilde{K}_k:=\int_{\R}\px^{k+1}B^tA_4\px^{k+1}Bdx_1,\quad G_k:=C\delta^{-\frac12}\int_{\R} \px\rhob\abs{\px^kb_1}^2+ \px\rhob\abs{\px^kb_3}^2dx_1.\label{sec-n-3-tt}
		\end{align}
		For $k=0$, $\tilde{E}_0=C\norm{\tilde{V}}_{L^2}^2$, and for $k\ge 1$
		\begin{align}\label{sec-n-5}
			C\left(\norm{\px^{k}\tilde{V}}_{L^2}^2-\deltab\sum_{j=0}^{k-1}(1+t)^{-(k-j)}\norm{\px^{j}\tilde{V}}_{L^2}^2\right)\leq \tilde{E}_{k}\le  C\left(\norm{\px^{k}\tilde{V}}_{L^2}^2+\deltab\sum_{j=0}^{k-1}(1+t)^{-(k-j)}\norm{\px^{j}\tilde{V}}_{L^2}^2\right).
		\end{align}
		We first verify the dissipation term. A direct computation shows that the matrix $A_{4}$ is nonnegative definite. In fact, one has
		\begin{align*}
			A_4=\left(\begin{array}{ccc}
				\frac{\tilde{\mu}}{2}+\frac{\gamma-1}{2\gamma}\tilde{\kappa} & a_3 & -\frac{\tilde{\mu}}{2}+\frac{\gamma-1}{2\gamma}\tilde{\kappa}\\
				a_3& \frac{\tilde{\kappa}}{2\gamma} & a_3\\
				-\frac{\tilde{\mu}}{2}+\frac{\gamma-1}{2\gamma}\tilde{\kappa}& a_3 &\frac{\tilde{\mu}}{2}+\frac{\gamma-1}{2\gamma}\tilde{\kappa}
			\end{array}\right),
			\text{\quad where\quad}
			\tilde{\mu}=\frac{\lambda+2\mu}{\tilde{\rho}} ,\quad\tilde{\kappa}=\frac{\kappa}{\tilde{\rho}} ,\quad a_3=-\sqrt{\frac{\gamma-1}{2}}\frac{\tilde{\kappa}}{\gamma}.
		\end{align*}
		Furthermore,
		\begin{align}
			\begin{aligned}\label{ds}
				&\px^{k+1}B^tA_4\px^{k+1}B\\
				=& \tilde{\kappa}\left[\sqrt{\frac{\gamma-1}{2\gamma}}(\px^{k+1} b_{1}+\px^{k+1} b_{3})-\frac{1}{\sqrt{\gamma}}\px^{k+1} b_{2}\right]^2+\tilde{\mu}\left[\frac{1}{\sqrt{2}}(\px^{k+1} b_{3}-\px^{k+1} b_{1})\right]^2\\
				\ge&C\big(|\px^{k+1}\tilde{\Psi}_1|^2+|\px^{k+1}\tilde{W}|^2\big)-C\bar{\delta}\sum_{j=1}^{k+1}D_{-j}\abs{\px^{k+1-j}\tilde{\Psi}_1}^2.
			\end{aligned}
		\end{align}
		And we also have
		\begin{align}\notag
			\int_{\R} \p_{x_1} \left(\bar{B}^{(k)}-\px^{k}B^t\right)A_4\px^{k+1}Bdx_1 \leq C(\chi+\bar{\delta}^{\frac12})(1+t)^{-1}E_k+C(\chi+\bar{\delta}^{\frac12})K_k.
		\end{align}
		By \cref{sec-n-4} and \cref{sec-n-5} and {\it a priori} assumptions \cref{apa}, one has
		\begin{align}\label{I12}
			I_1+I_2\leq C(\chi+\bar{\delta}^{\frac12})(1+t)^{-1}E_k+C(\chi+\bar{\delta}^{\frac12})K_k.
		\end{align}
		The estimate of $I_3$ is relatively subtle. We first estimate $\int_{\R}\bar{B}^{(k)}\mathcal{M}_{k}^{(2)}dx_1$. It holds that 
		\begin{align}
			\int_{\R}\bar{B}^{(k)}\mathcal{M}_{k}^{(2)}dx_1\leq C\deltab\sum_{i=0}^{k+1}(1+t)^{-i}\norm{\px^{k-i+1}B}_{L^2}^2, \quad \text{for}\quad k\ge 0.\label{I14}
		\end{align}
		Next, we consider the estimate of $\int_{\R}\bar{B}^{(k)}\mathcal{M}_{k}^{(1)}dx_1$. For the case of $k=0$, it is easy to see
		\begin{align}\label{IL12}	
			\int_{\R}\bar{B}^{(0)}\mathcal{M}_{0}^{(1)}dx_1\le C\int_{\R} \px \rhob \abs{b_1}^2+\px \rhob \abs{b_3}^2 dx_1.
		\end{align}
		For the case of $1\le k \le 2$, noticing
		$$
		\abs{\px^2 \rhob}\lesssim \frac{\abs{x_1}}{1+t}\px\rhob,\quad
		\abs{\px^3 \rhob}\lesssim \frac{\abs{x_1}^2}{(1+t)^2}\px\rhob+(1+t)^{-1}\px\rhob,
		$$ 
		we have
		\begin{align}
			&\int_{\R}\bar{B}^{(1)}\mathcal{M}_{1}^{(1)}dx_1\lesssim \sum_{i=1,3}\int_{\R} \frac{\abs{x_1}}{1+t} \px \rhob \abs{b_i}\abs{\px b_i} dx_1+\sum_{i=1,3}\int_{\R}  \px \rhob \abs{\px b_i}^2 dx_1 \notag\\
			&\qquad\qquad\qquad\quad\; \lesssim \deltab\sum_{i=1,3}\int_{\R} \tilde{\omega}_{-\frac12} \abs{\px b_i}^2 dx_1+\sum_{j=0}^1\sum_{i=1,3}(1+t)^{-j} \int_{\R} \px \rhob \abs{\px^{1-j}b_i}^2 dx_1,\label{2025-10-30-1}\\
			&\int_{\R}\bar{B}^{(2)}\mathcal{M}_{2}^{(1)} dx_1\lesssim \sum_{i=1,3}\int_{\R} \frac{x_1^2}{(1+t)^2} \px \rhob \abs{b_i}\abs{\px^2 b_i} dx_1+\sum_{i=1,3}\int_{\R} \frac{\abs{x_1}}{1+t} \px \rhob \abs{\px b_i}\abs{\px^2 b_i} dx_1\notag\\
			&\qquad\qquad\qquad\qquad\;\;+(1+t)^{-1}\sum_{i=1,3}\int_{\R}  \px \rhob \abs{b_i}\abs{\px^2 b_i} dx_1+\sum_{i=1,3}\int_{\R}  \px \rhob \abs{\px^2 b_i}^2 dx_1 \notag\\
			&\qquad\qquad\qquad\quad\; \lesssim\deltab \sum_{i=1,3}\int_{\R} \tilde{\omega}_{-\frac12} \abs{\px^2 b_i}^2 dx_1+\sum_{j=0}^2\sum_{i=1,3}(1+t)^{-j} \int_{\R} \px \rhob \abs{\px^{2-j}b_i}^2 dx_1,\label{2025-10-30-2}
		\end{align}
		where by \cref{errors-1},  we have used
		\begin{align*}
			&\int_{\R} \frac{\abs{x_1}}{1+t} \px \rhob \abs{b_i}\abs{\px b_i} dx_1 \lesssim \int_{\R} \frac{x_1^2}{1+t}\px\rhob\abs{\px b_i}^2dx_1+(1+t)^{-1}\int_{\R}\px\rhob\abs{ b_i}^2dx_1\\
			&\qquad\qquad\qquad\qquad\qquad\quad\lesssim\deltab \int_{\R} \tilde{\omega}_{-\frac12}\abs{\px b_i}^2dx_1+(1+t)^{-1}\int_{\R}\px\rhob\abs{ b_i}^2dx_1,\\
			&\int_{\R} \frac{x_1^2}{(1+t)^2} \px \rhob \abs{b_i}\abs{\px^2 b_i} dx_1 \lesssim\int_{\R} \frac{x_1^4}{(1+t)^2}\px\rhob\abs{\px^2 b_i}^2dx_1+(1+t)^{-2}\int_{\R}\px\rhob\abs{ b_i}^2dx_1\\
			&\qquad\qquad\qquad\qquad\qquad\qquad\;\;\lesssim\deltab \int_{\R} \tilde{\omega}_{-\frac12}\abs{\px^2 b_i}^2dx_1+(1+t)^{-2}\int_{\R}\px\rhob\abs{ b_i}^2dx_1.
		\end{align*}
		Next, we estimate terms involving $\mathcal{M}_k^{(3)}$. For $k=0$, it holds that
		\begin{align}\notag
			\begin{aligned}
				\norm{\bar{B}^{(k)}\mathcal{M}_{k}^{(3)}}_{L^1}\leq&C\int_{\R} \sum_{l=0}^1 (\deltab D_{-1}+\mathcal{D}^{(l)}+\mathcal{T}^{(l)}+\mathcal{Z}^{(l)})\abs{\bar{B}^{(0)}}dx_1:=\sum_{j=1}^7 I_{4}^{(j)}.
			\end{aligned}
		\end{align} 
		Then we estimate $I_4$ term by term as follows:
		\begin{align}
			I_{4}^{(1)}\leq& C\deltab \int_{\R}D_{-1}|\bar{B}| dx_1\leq C\bar{\delta}  (1+t)^{-\frac{1}{2}}+C\bar{\delta}(1+t)^{-1}\norm{B}^2_{L^2},\label{I61}\\
			I_{4}^{(2)}\leq&C\deltab \int_{\R}\big(D_{-\frac{1}{2}}\abs{\px V}+D_{-1}\abs{V}\big)\abs{\bar{B}}dx_1 \leq C\deltab (1+t)^{-1}\norm{V}^2_{L^2}+C\deltab \norm{\px V}^2_{L^2},\label{I62}\\
			I_{4}^{(3)}\leq&C\deltab \int_{\R}\big(D_{-1}\abs{\px V}+D_{-\frac{3}{2}}\abs{V}+D_{-\frac{1}{2}}\abs{\px^2 V}\big)\abs{\bar{B}}dx_1 \leq C\deltab (1+t)^{-1}\norm{V}^2_{L^2}+C\deltab \norm{\px V}^2_{H^1},\label{I63}\\
			I_{4}^{(4)}\leq&\int_{\R}\bigg(C\bar{\delta} \big(D_{-1}\abs{ V}^2+D_{-1}\abs{ V}+D_{-\frac{1}{2}}\abs{\px V}\big)+\abs{\px V}^2\bigg)\abs{\bar{B}}dx_1\label{I64}\\
			\leq& C\deltab (1+t)^{-1}\norm{V}^2_{L^2}+C(\deltab+\chi)\norm{\px V}^2_{L^2},\nonumber\\
			I_{4}^{(5)}\leq&C\int_{\R}\abs{\px^2 V}\abs{\px V}\abs{\bar{B}}+\deltab\bigg(D_{-\frac{1}{2}}\abs{\px V}^2+D_{-1}\abs{\px V}+D_{-\frac{1}{2}}\abs{\px^2 V}\bigg)\abs{\bar{B}}dx_1\label{I65}\\
			\leq& C(\deltab+\chi)(1+t)^{-1}\norm{V}^2_{L^2}+C(\deltab+\chi)\norm{\px V}^2_{H^1},\nonumber\\
			I_4^{(6)}+&I_4^{(7)}\le  C(\varepsilon^2+\delta)(1+t)^{-1}\norm{V}^2_{L^2}+C(\varepsilon^2+\delta)(1+t)e^{-ct},\label{I66}
		\end{align}
		where we have used the {\it a priori} assumptions \cref{apa}. For $k=1,2$, one should take care of order of derivatives so that
		\begin{align}
			\begin{aligned}
				&\int_{\R}\bar{B}^{(1)}\mathcal{M}_{1}^{(3)}dx_1\leq C\int_{\R}\sum_{l=0}^1 (\deltab D_{-1}+\mathcal{D}^{(l)}+\mathcal{T}^{(l)}+\mathcal{Z}^{(l)})\left(\abs{\px^2{B}}+\bar{\delta}^{\frac{1}{2}}D_{-\frac{1}{2}}|\px B|\right)dx_1,
			\end{aligned}\label{2025.6.02-6}\\
			\begin{aligned}
				\int_{\R}\bar{B}^{(2)}\mathcal{M}_{2}^{(3)}dx_1\leq&\int_{\R}\left(\abs{\px^3B}+C\deltab^{\frac{1}{2}}D_{-\frac{1}{2}}\abs{\px^2B}\right)\left( L \px A_3+\px L A_3\right)dx_1\\
				\leq&C\int_{\R}\sum_{l=0}^1 (\deltab D_{-1}+\mathcal{D}^{(l)}+\mathcal{T}^{(l)}+\mathcal{Z}^{(l)})\big(\deltab D_{-\frac{1}{2}}\abs{\px^3{B}}+\deltab^{\frac{3}{2}}D_{-1}\abs{\px^2B}\big)dx_1\\
				&+C\int_{\R}\sum_{l=1}^2 (\deltab D_{-\frac32}+\mathcal{D}^{(l)}+\mathcal{T}^{(l)}+\mathcal{Z}^{(l)})\big(\abs{\px^3{B}}+\deltab^{\frac{1}{2}}D_{-\frac{1}{2}}\abs{\px^2B}\big)dx_1.
			\end{aligned}\label{2025.6.02-7}
		\end{align} 
		We focus on \cref{2025.6.02-6}. For the error terms in \cref{2025.6.02-6}, we have
		\begin{align}\notag
			\int_{\R}&\deltab D_{-1}\abs{\px^{2}B}+\deltab^{\frac{3}{2}}D_{-\frac{3}{2}}\abs{\px B}dx_1\leq C\deltab (1+t)^{-\frac{3}{2}}+C\deltab \norm{\px^2B}_{L^2}^2+C\deltab (1+t)^{-1}\norm{\px B}_{L^2}^2.
		\end{align}
		The calculation of the remaining terms in \cref{2025.6.02-6} is similar to  \cref{I61,I62,I63,I64,I65}. Then we have
		\begin{align}\notag
			\int_{\R}\bar{B}^{(1)}\mathcal{M}_{1}^{(3)}dx_1\leq&C(\deltab+\chi)\sum_{i=0}^2(1+t)^{-i}\norm{\px^{2-i}B}_{L^2}^2+{C\bar\delta(1+t)^{-\frac{3}{2}}}.
		\end{align}
		Next, we  calculate \cref{2025.6.02-7}. For the error terms in \cref{2025.6.02-7}, one has
		\begin{align}\notag
			C\int_{\R}\deltab D_{-\frac{3}{2}}\abs{\px^3{B}}+\deltab^{\frac32}D_{-2}\abs{\px^2B}dx_1\leq C\deltab (1+t)^{-\frac{5}{2}}+C\deltab \norm{\px^3B}_{L^2}^2+ C\deltab (1+t)^{-1}\norm{\px^2 B}_{L^2}^2.
		\end{align}
		At the moment, by the {\it a priori} assumptions \cref{apa}, we focus on the terms involving third-order derivatives:
		\begin{align*}
			\int_{\R}\mathcal{D}^{(1)}&\abs{\px^3B}dx_1\leq \int_{\R}C\deltab \sum_{j=1}^{3} D_{-\frac{j}{2}} \abs{\px^{3-j} V}\abs{\px^3B}dx_1\leq C\deltab \sum_{j=0}^{3}(1+t)^{-j}\norm{\px^{3-j}V}_{L^2}^2,\\
			\int_{\R}\mathcal{D}^{(2)}&\abs{\px^3B}dx_1\leq \int_{\R}C\deltab \sum_{j=1}^{4} D_{-\frac{j}{2}} \abs{\px^{4-j} V}\abs{\px^3B}dx_1\leq C\deltab \sum_{j=0}^{3}(1+t)^{-j}\norm{\px^{3-j}V}_{L^2}^2,\\
			\int_{\R}\mathcal{T}^{(1)}\abs{\px^3B}dx_1\leq& \int_{\R} \left[\abs{\px^2 V}\abs{\px V}+C\deltab \left(D_{-\frac{1}{2}}\abs{\px V}^2+D_{-1}\abs{\px V}+D_{-\frac{1}{2}}\abs{\px^2 V}\right) \right]\abs{\px^3B}dx_1\\
			\leq &C(\deltab+\chi)\sum_{j=0}^3(1+t)^{-j}\norm{\px^{3-j}V}_{L^2}^2\nonumber,\\
			\int_{\R}\mathcal{T}^{(2)}\abs{\px^3B}dx_1\leq& \int_{\R}\left(\abs{\px^3V}\abs{\px V}+\abs{\px^2V}^2\right)\abs{\px^3B}+C\deltab  D_{-\frac{1}{2}}\big(\abs{\px V}\abs{\px^2 V}+\abs{\px^3V}\big)\abs{\px^3B}dx_1\nonumber\\
			&+\int_{\R}\deltab D_{-1}\big(\abs{\px^2V}+\abs{\px V}^2\big)\abs{\px^3B}+\deltab D_{-\frac{3}{2}}\abs{\px V}\abs{\px^3B}dx_1\\
			\leq &C(\deltab+\chi)\sum_{j=0}^3(1+t)^{-j}\norm{\px^{3-j}V}_{L^2}^2+C\norm{\px^2 V}_{L^4}^4\nonumber.
		\end{align*}
		By Gagliardo-Nirenberg inequality and the {\it a priori} assumption \cref{apa}, one has
		\begin{align}\notag
			\norm{\px^2 V}_{L^4}^4\leq C\norm{\px^3 V}^{\frac{5}{2}}_{L^2}\norm{\px V}^{\frac{3}{2}}_{L^2}\leq C\chi^2\norm{\px^3 V}^2_{L^2}.
		\end{align}
		It is also the same to treat the terms containing $\mathcal{Z}^{l}$ as for treating \cref{I66}. In sum, one has
		\begin{align}\label{Mk}
			\int_{\R}\bar{B}^{(k)}\mathcal{M}_{k}^{(3)}dx_1\leq C\deltab (1+t)^{-\frac{1}{2}-k}+C(\deltab+\chi)\sum_{j=0}^{k+1}(1+t)^{-j}\norm{\px^{k+1-j}V}_{L^2}^2.
		\end{align}

		\begin{flushleft}
			\textbf{Step 2. Estimates on $\int_{\R} \tilde{\omega}_{-\frac12} \abs{\px^k b_i}^2 dx_1$.}
		\end{flushleft}
		
		Notice that we still need to control $\int_{\R} \tilde{\omega}_{-\frac12} \abs{\px^k b_i}^2 dx_1$ in \cref{2025-10-30-1,2025-10-30-2}. From \cref{errors-1} and \cref{2025-10-6}, we have 
		$$
		\deltab\tilde{\omega}_{-1/2}\approx\deltab (1+t)^{-\frac12}e^{-\frac{\tilde{c}_0 x_1^2}{1+t}},  
		\quad \px \rhob \approx\deltab (1+t)^{-\frac12}e^{-\frac{{c}_0 x_1^2}{1+t}},
		$$ 
		with $\tilde{c}_0<c_0$. Therefore, $\int_{\R} \tilde{\omega}_{-\frac12} \abs{\px^k b_i}^2 dx_1$ cannot be directly controlled by $G_k$ in \cref{sec-n-3-tt}. We need to develop further weighted heat kernel inequalities \cref{2025-10-4,2025-10-5} to obtain the control of $\int_{\R} \tilde{\omega}_{-\frac12} \abs{\px^k b_i}^2 dx_1$ for $1\le k \le 2$. To derive these estimates, let 
		$$
		\tilde{h}(x_1,t)=\int_{-\infty}^{x_1} \tilde{\omega}_{-\frac12}(y,t)dy.
		$$
		It follows that $\norm{\tilde{h}}_{L^\infty}\leq C$ and $4\tilde{c}_0\p_t \tilde{h}= \px\tilde{\omega}_{-\frac12}$. Multiplying \cref{equ-Bk}$_1$ by $\tilde{h}\px^k{b}_1$, we have
		\begin{align}\label{2025-10-30-3}
			& \frac{1}{2} \frac{d}{dt}\left( \tilde{h} \px^kb_1 \right)^2-\frac{1}{2}\p_t \tilde{h}\left(\px^k h \right)^2 + \frac{1}{2}\px \left[ \tilde{h} \tilde{\lambda}_1 \left(\px^k b_1 \right)^2 \right] - \frac{1}{2} \px \left( \tilde{h} \tilde{\lambda}_1\right)\left( \px^k b_1\right)^2\notag\\
			=&\left(\mathcal{M}_k^{(1)} \right)_1\tilde{h}\px^k b_1+ \left[ \left(\tilde{L} A_2 \tilde{R} \px^{k+2} B \right)_1 + \sum_{i=2}^3  \left(\mathcal{M}_k^{(i)} \right)_1 \right] \tilde{h} \px^k b_1.
		\end{align}
		Using the same method as for deriving \cref{2025-10-30-1,2025-10-30-2}, one has
		\begin{align*}
			\int_{\R}\left(\mathcal{M}_k^{(1)} \right)_1\tilde{h}\px^k b_1 dx_1 \lesssim \deltab \int_{\R} \tilde{\omega}_{-\frac12} \abs{\px^k b_1}^2 dx_1 + \sum_{i=0}^k(1+t)^{-i} \int_{\R} \px \rhob \abs{\px^{k-i}b_1}^2 dx_1. 
		\end{align*}
		Using an argument similar to \cref{2025.6.02-6}-\cref{Mk}, one has
		\begin{align*}
			\int_{\R}\Big[ \left(\tilde{L} A_2 \tilde{R} \px^{k+2} B \right)_1 + \sum_{i=2}^3  \left(\mathcal{M}_k^{(i)} \right)_1 \Big] \tilde{h} \px^k b_1 dx_1 \lesssim \sum_{j=0}^{k+1}(1+t)^{j-k-1}\norm{\px^j V}_{L^2}^2+(1+t)^{-\frac12-k}.
		\end{align*}
		Thus, integrating \cref{2025-10-30-3} with respect to $x_1$, we obtain
		\begin{align}\label{2025-10-4}
			\int_{\R} \tilde{\omega}_{-\frac12} \abs{\px^k b_1}^2 dx_1 +\p_t \int_{\R} \tilde{h} \left(\px^k b_1 \right)^2 dx_1 \lesssim &  \sum_{j=0}^{k+1} (1+t)^{j-k-1}\norm{\px^j V}_{L^2}^2 + (1+t)^{-\frac12-k} \notag \\
			&+ \sum_{j=0}^k (1+t)^{-j} \int_{\R} \px \rhob \abs{\px^{k-j}b_1}^2 dx_1 .
		\end{align}
		Using the same argument as \cref{2025-10-4}, one has
		\begin{align}\label{2025-10-5}
			\int_{\R} \tilde{\omega}_{-\frac12} \abs{\px^k b_2}^2 dx_1- \p_t \int_{\R} \tilde{h} \left(\px^k b_2 \right)^2 dx_1 \lesssim &\sum_{j=0}^{k+1} (1+t)^{j-k-1}\norm{\px^j V}_{L^2}^2 + (1+t)^{-\frac12-k} \notag \\
			&+ \sum_{j=0}^k (1+t)^{-j} \int_{\R} \px \rhob \abs{\px^{k-j}b_2}^2 dx_1 .
		\end{align}
		Finally, by collecting \cref{sec-3-1}-\cref{Mk}, \cref{2025-10-4,2025-10-5}, we arrive at
		\begin{align}
			&\frac{d}{dt}\big(\sum_{i=0}^2\tilde{E}_i\big)+\sum_{i=0}^2(\tilde{K}_i+G_{i})\leq C\bar{\eta}\left[(1+t)^{-1}\big(\sum_{i=0}^2 E_i\big)+\left(\norm{\px\tilde{\Phi}}_{H^2}^2+\norm{\px\tilde{\Psi}_2}_{H^2}^2+\norm{\px\tilde{\Psi}_3}_{H^2}^2\right)\right]+C\bar{\delta}(1+t)^{-\frac12},\label{1}\\
			&\frac{d}{dt}\big(\sum_{i=1}^2\tilde{E}_i\big)+\sum_{i=1}^2(\tilde{K}_i+G_{i})\leq C\etab(1+t)^{-1}\big(\sum_{i=1}^2E_i+G_0\big)+C\deltab(1+t)^{-\frac{3}{2}}\nonumber\\
			&\qquad\qquad\qquad\qquad+C\etab(1+t)^{-2}E_0+C\etab\left(\norm{\px^2\tilde{\Phi}}_{H^1}^2+\norm{\px^2\tilde{\Psi}_2}_{H^1}^2+\norm{\px^2\tilde{\Psi}_3}_{H^1}^2\right),\label{2}\\
			&\frac{d}{dt}\tilde{E}_2+\tilde{K}_2+G_{2}\leq C\etab\sum_{i=0}^2(1+t)^{-(3-i)}E_{i}+C\deltab(1+t)^{-\frac{5}{2}}+C\etab\left(\norm{\px^3\tilde{\Phi}}_{L^2}^2+\norm{\px^3\tilde{\Psi}_2}_{L^2}^2+\norm{\px^3\tilde{\Psi}_3}_{L^2}^2\right)\nonumber\\
			&\qquad\qquad\qquad\qquad\quad+C\etab\big[(1+t)^{-2}G_0+(1+t)^{-1}G_1\big].\label{3}
		\end{align}
		
		\begin{flushleft}
			\textbf{Step 3. Estimates on $\norm{\px^{k+1}\tilde{\Phi}}_{L^2}$.}
		\end{flushleft}
		
		We still need to estimate $\norm{\px^{k+1}\tilde{\Phi}}_{L^2}$. Applying $\px^k$ ($k=0,1,2$) to \cref{sys-Eu0-1}$_2$ and then multiplying the resulting equation by $\px^{k+1}\tilde{\Phi}$, one has
		\begin{align}\label{sec-4-2}
			\p_t \int_{\R}\px^{k+1}\tilde{\Phi}\px^k \tilde{\Psi}_{1}dx_1+\left(\gamma-1\right)\norm{\px^{k+1}\tilde{\Phi}}^{2}_{L^2}=\int_{\R}\frac{\lambda+2\mu}{\rhot}\px^{k+2}\tilde{\Psi}_1\px^{k+1}\tilde{\Phi}dx_1+I_5+I_6,
		\end{align}
		where 
		\begin{align*}
			&  I_5= \int_{\R} \px^{k+1} \tilde{\Phi} \px^{k} J_{21}^{(2)}dx_1, \\   
			&I_6=\int_{\R} \px^{k+1}\tilde{\Phi} \left[ \px^k\mathcal{Q}_{21} -\left(\gamma-1\right)\px^{k+1}\tilde{W} -\px^{k}\left(J_1+\Tt\mathcal{Q}_1-J_{21}^{(1)}-\Tt\px\tilde{\Psi}_1\right) +\sum_{i=0}^{k-1}\px^{k-i}\left( \frac{\lambda+2\mu}{\rhot} \right)\px^{i+2}\tilde{\Psi}_1 \right]dx_1.
		\end{align*}
		Here we have used
		\begin{align*}
			\int_{\R} \p_t \px^{k}\tilde{\Psi}_1\px^{k+1}\tilde{\Phi}dx_1=\p_t \int_{\R} \px^{k}\tilde{\Psi}_1\px^{k+1}\tilde{\Phi}dx_1-\int_{\R} \px^{k+1}\tilde{\Psi}_1\px^{k}\left( \thetat\px\tilde{\Psi}_1-J_1-\thetat\mathcal{Q}_1 \right) dx_1.
		\end{align*}
		Next, multiplying $ \frac{1}{\thetat\rhot}\px^{k+1}$\cref{sys-Eu0-1}$_1$ by $\px^{k+1}\tilde{\Phi}$, one has
		\begin{align}\label{sec-4-5}
			\p_t \int_{\R} \frac{\lambda+2\mu}{2 \rhot\Tt} \abs{\px^{k+1}\tilde{\Phi}}^{2} dx_1=-\int_{\R}\frac{\lambda+2\mu}{\rhot}\px^{k+2}\tilde{\Psi}_1\px^{k+1}\tilde{\Phi}dx_1+I_7+I_8,
		\end{align}
		where
		\begin{align*}
			I_7&=\int_{\R} \frac{\px^{k+1} J_{1}}{\thetat\rhot} \px^{k+1}\tilde{\Phi}dx_1,
			\\
			I_8&=\int_{\R} \p_t\left(\frac{\lambda+2\mu}{2 \rhot\Tt} \right)\abs{\px^{k+1}\tilde{\Phi}}^{2}+\px^{k+1}\tilde{\Phi}\left[ \frac{\px^{k+1}(\thetat\mathcal{Q}_1)}{\thetat\rhot} - \sum_{i=0}^{k+1} \frac{\px^{k+1-i}\thetat \px^i\tilde{\Psi}_1}{\rhot\thetat}  \right] dx_1.
		\end{align*}
		By \cref{mathcal-Z}, \cref{mathcal-Q} and \cref{J}  and {\it a priori} assumptions \cref{apa}, one has
		\begin{align}\label{sec-4-8}
			\abs{I_6}+\abs{I_8} \le &\frac{\gamma-1}{160}\norm{\px^{k+1}\tilde{\Phi}}_{L^2}^2+C\left( \sum_{i=1}^3 \norm{\px^{k+1}\tilde{\Psi}_i}_{L^2}^2+\norm{\px^{k+1}\tilde{W}}_{L^2}^2\right)\notag\\
			&+ C\etab \sum_{j=1}^k(1+t)^{-j}\norm{\px^{k+1-j}V}_{L^2}^2+C\deltab(1+t)^{-\frac{3}{2}-k}.
		\end{align}
		For $I_5$, one should take care of the highest order of derivatives so that
		\begin{align}\label{J21-1}
			\begin{aligned}
				&I_5= \int_{\R}(2\mu+\lambda)\px^k\left[\px(\mr{u}_1-\tilde{u}_1)-\frac{1}{\rhot}\px^2\tilde{\Psi}_1\right]\px^{k+1}\tilde{\Phi}dx_1\\
				=&(\lambda+2\mu)\int_{\R} \px^k \bigg[ \px\left( \frac{1}{\rhot} \right) \px\Psi_1 - \px \left( \frac{\px\Phi\px\Psi_1}{\rhom\rhot} \right) -\px \left( \frac{\ut_1\px\Phi}{\rhom}  \right)+ \frac{1}{\rhot}\px \mathcal{B}_{1r} +\px \Do \left( \frac{m_1}{\rho} -\frac{\mm_1}{\rhom}\right) \bigg] \px^{k+1}\tilde{\Phi} dx_1  .
			\end{aligned}
		\end{align}
		The zero mode term in the above formula that may contain the 
		$(k+2)$-th order derivative is 
		$$ \px^k\left[- \px \left( \frac{\px\Phi\px\Psi_1}{\rhom\rhot} \right) -\px \left( \frac{\ut_1\px\Phi}{\rhom}  \right)+ \frac{1}{\rhot}\px \mathcal{B}_{1r}\right].$$
		Notice that $\px^{k+2}\Phi=\frac{1}{\thetat}\px^{k+2}\tilde{\Phi}+\sum_{i=0}^{k+1} \px^{k+2-i}(\frac{1}{\thetat}) \px^i\tilde{\Phi}$ and $\px^{k+2}\tilde{\Phi}$
		and $\px^{k+1}\tilde{\Phi}$
		can form a total differential, so the terms containing 
		$\px^{k+2}\Phi$
		can transfer one derivative out through integration by parts. Therefore, next we will only show in detail the treatment of 
		$-\frac{1}{\rhom\rhot}\px\Phi\px^{k+2}\Psi_1$ in the term 
		$- \px^{k+1} \left( \frac{\px\Phi\px\Psi_1}{\rhom\rhot} \right)$.
		By \cref{eq-pert1}$_1$, one has
		\begin{align}\label{sec-4-1}
			& -\int_{\R} \frac{\px \Phi \px^{k+2} \Psi_1} {\rhom\rhot}\px^{k+1}\tilde{\Phi}dx_1=\int_{\R} \frac{\px \Phi \p_t \px^{k+1}\Phi +\px\Phi\px^{k+1}\mathcal{Q}_1}{\rhom\rhot}\px^{k+1}\tilde{\Phi}dx_1 \notag\\
			=&  \int_{\R} \frac{\px \Phi \p_t \left[ \px^{k+1}\tilde{\Phi} + \sum_{i=0}^k \px^{k+1-i} \left( \frac{1}{\thetat}\right)\px^i\tilde{\Phi} \right] +\px\Phi\px^{k+1}\mathcal{Q}_1}{\rhom\rhot}\px^{k+1}\tilde{\Phi}dx_1 \notag\\
			=&\p_t \int_{\R} \frac{\px \Phi (\px^{k+1}\tilde{\Phi})^2}{2\rhom\rhot}dx_1 - \int_{\R} \p_t\left( \frac{\px\Phi}{2\rhom\rhot} \right) \abs{\px^{k+1}\tilde{\Phi}}^2 dx_1 + \int_{\R} \frac{\px\Phi \px^{k+1}\mathcal{Q}_1\px^{k+1}\tilde{\Phi}}{\rhom\rhot}dx_1 \notag\\
			&+\int_{\R} \frac{\px\Phi}{\rhom\rhot}\p_t\left[  \sum_{i=0}^k \px^{k+1-i}\left( \frac{1}{\thetat} \right) \px^i\tilde{\Phi} \right]\px^{k+1}\tilde{\Phi} dx_1.
		\end{align}
		Combining \cref{J21-1} and \cref{sec-4-1} and using \cref{mathcal-Z}-\cref{J}  and the {\it a priori} assumptions \cref{apa}, one has
		\begin{align}\label{J21}
			\begin{aligned}
				I_5-(\lambda+2\mu) \p_t\int_{\R} \frac{\px \Phi \abs{\px^{k+1}\tilde{\Phi}}^2}{2\rhom\rhot}dx_1
				&\leq C \etab\sum_{j=0}^k(1+t)^{-j}\norm{\px^{k+1-j}V}_{L^2}^2+C\norm{\mathcal{Z}^{(k+1)}}_{L^2}^2 \\
				&\leq C \etab\sum_{j=0}^k(1+t)^{-j}\norm{\px^{k+1-j}V}_{L^2}^2+C\varepsilon^2  e^{-ct}.
			\end{aligned}
		\end{align}
		For $I_{7}$, one should also take care of the highest order of derivatives, so that 
		\begin{align}\label{2025-11-7-1}
			I_7&=\int_{\R} \frac{\px^{k+1} \left( \p_t\thetat\Phi-\ut_1\px\tilde{\Phi}+\px\thetat\ut_1\tilde{\Phi}-\px\ut_1 \tilde{\Phi} \right)}{\thetat\rhot}\px^{k+1}\tilde{\Phi}dx_1\notag\\
			&=\int_{\R} \frac{\px^{k+1} \left( \p_t\thetat\Phi+\px\thetat\ut_1\tilde{\Phi}-\px\ut_1 \tilde{\Phi} \right)-\sum_{i=0}^{k}\px^{k+1-i}\ut_1\px^{i+1}\tilde{\Phi}}{\thetat\rhot}\px^{k+1}\tilde{\Phi}dx_1+\int_{\R}\px\left( \frac{\ut_1}{2\thetat\rhot}\right)\abs{\px^{k+1}\tilde{\Phi}}^2dx_1\notag\\
			&\leq C\etab \sum_{j=0}^k(1+t)^{-j}\norm{\px^{k+1-j}V}_{L^2}^2.
		\end{align}
		Combining \cref{sec-4-2}-\cref{2025-11-7-1}, one has
		\begin{align}\label{sec-4-9}
			& \frac{d}{dt} \int_{\R}\left(\px^{k+1}\tilde{\Phi}\px^k \tilde{\Psi}_{1}+ \frac{\lambda+2\mu}{2 \rhot\Tt} \abs{\px^{k+1}\tilde{\Phi}}^{2}-(\lambda+2\mu)\frac{\px\Phi\abs{\px^{k+1}\tilde{\Phi}}^2}{2\rhom\rhot}\right)dx_1+\frac{\gamma-1}{2}\norm{\px^{k+1}\tilde{\Phi}}^{2}_{L^2} \notag\\
			\leq & C \etab \sum_{j=1}^{k}(1+t)^{-j}\norm{\px^{k+1-j}V}_{L^2}^2+ C \deltab(1+t)^{-\frac{3}{2}-k}+C\left( \sum_{i=1}^3 \norm{\px^{k+1}\tilde{\Psi}_i}_{L^2}^2+\norm{\px^{k+1}\tilde{W}}_{L^2}^2\right). 
		\end{align}
		
		\begin{flushleft}
			\textbf{Step 4. Estimates on $\tilde{\Psi}_i$, $i=2,3$.}
		\end{flushleft}
		
		Applying $\px^k$, $k=0,1,2$ to \cref{sys-Eu0-1}$_3$ , then multiplying them by $\p^k\tilde{\Psi}_i$, $i=2,3$ and then integrating the resulting equation over $\R$ with respect to $x_1$, one has
		\begin{align*}
			\frac{d}{dt}\norm{\px^k\tilde{\Psi}_i}_{L^2}^2+\int_{\R}\frac{\mu}{\rhot}\abs{\px^{k+1}\tilde{\Psi}_i}^2dx_1\leq& C\int_{\R}\abs{\px^2\left( \frac{\mu}{\rhot}\right)}\abs{\px^k \tilde{\Psi}_i}^2dx_1+\int_{\R}\abs{\px^{k}\Big(J_{2i}+\mathcal{Q}_{2i}\Big)\px^{k}\tilde{\Psi}_i}dx_1
			:=I_{\psi}^{k1}+I_{\psi}^{k2}.
		\end{align*}
		Directly,
		\begin{align}\notag
			I_{\psi}^{k1}\leq C\etab\sum_{j=0}^{k}(1+t)^{-j}\norm{\px^{k+1-j}V}_{L^2}^2,
		\end{align}
		and the estimate of $I_{\psi}^{k2}$ is similar to \cref{I61,I62,I63,I64,I65}. Finally, we have
		\begin{align}\label{psi2}
			& \sum_{i=2}^3\left[\frac{d}{dt}\norm{\px^k\tilde{\Psi}_i}_{L^2}^2+\int_{\R}\frac{\mu}{\rhot}\abs{\px^{k+1}\tilde{\Psi}_i}^2dx_1\right]\\
			\notag
			\leq& C \etab\left[\norm{\px^{k+1}\tilde{\Phi}}_{L^2}^2+\norm{\px^{k+1}\tilde{\Psi}_1}_{L^2}^2+\norm{\px^{k+1}\tilde{W}}_{L^2}^2 +\sum_{j=1}^{k}(1+t)^{-j}\norm{\px^{k+1-j}V}_{L^2}^2\right]+C \deltab(1+t)^{-\frac{1}{2}-k}.
		\end{align}
		Then combining \cref{sec-4-9,psi2} and \cref{1,2,3}, we have proved \cref{11-1,22-1,33-1}. 
	\end{proof}
	\begin{flushleft}
		\textbf{Decay rate for the zero modes.}
	\end{flushleft}
	\begin{Lem}\label{lem-zeromode-decay}
		Under the same assumptions as \cref{Thm-ape}, one has the following estimates for the zero modes
		\begin{align*}
			&\norm{\px^k(\phim,\varphim,\zetam)}_{L^2}^2\leq C(\varepsilon^2+\deltab)(1+t)^{-\frac{2k+1}{2}},\quad k=0,1.\\
			&\norm{(\phim,\varphim,\zetam)}_{L^{\infty}}\leq C(\varepsilon+\deltab^{\frac{1}{2}})(1+t)^{-\frac{1}{2}}.
		\end{align*}
	\end{Lem}
	
	\begin{proof}
		In the sequel, we use the following fact which comes from \cref{sec-n-5}:
		\begin{align}\notag
			{E}_1\leq CK_0+C\deltab(1+t)^{-1}E_0,\qquad {E}_2\le CK_1+C\deltab(1+t)^{-1}K_0+C\deltab(1+t)^{-2}E_0.
		\end{align}
		By Gronwall's inequality and \cref{11-1}, one has
		\begin{align}\label{E0}
			\sum_{i=0}^2E_i\leq C(\varepsilon^2+\deltab)(1+t)^{\frac{1}{2}},\qquad\int_0^t\sum_{i=0}^2(K_i+G_{i})d\tau\leq C(\varepsilon^2+\deltab)(1+t)^{\frac{1}{2}}.
		\end{align}
		Multiplying \cref{22-1} by $(1+t)$ and then integrating on $[0,t]$, one has
		\begin{align}\notag
			(1+t)\sum_{i=1}^2E_i+\int_{0}^t(1+\tau)\sum_{i=0}^2(K_i+G_{i})d\tau\leq C(\varepsilon^2+\deltab)(1+t)^{\frac{1}{2}}.
		\end{align}
		Then it follows that
		\begin{align}\label{E1}
			\sum_{i=1}^2E_i\leq C(\varepsilon^2+\deltab)(1+t)^{-\frac{1}{2}},\qquad\int_{0}^t(1+\tau)\sum_{i=1}^2(K_i+G_{i})d\tau\leq C(\varepsilon^2+\deltab)(1+t)^{\frac{1}{2}}.
		\end{align}
		Finally, multiplying \cref{33-1} by $(1+t)^2$, one has
		\begin{align}\notag
			(1+t)^2E_2+\int_{0}^t(1+\tau)^2(K_2+G_{2})d\tau\leq C(\varepsilon^2+\deltab)(1+t)^{\frac{1}{2}},
		\end{align}
		ans it holds that
		\begin{align}\label{E2}
			E_2\leq C(\varepsilon^2+\deltab)(1+t)^{-\frac{3}{2}}.
		\end{align}
		Then, by \cref{sec-n-5,E0,E1,E2}, the desired decay rate can be obtained as   
		\begin{align}\notag
			\norm{(\phim,\psim,\mathring{w})}_{L^\infty}\leq C(\varepsilon+\deltab^{\frac{1}{2}})(1+t)^{-\frac{1}{2}}.
		\end{align}
		Using the same argument as \cref{123456}, one can prove \cref{lem-zeromode-decay}.
	\end{proof}

		\subsection{Estimates on non-zero mode of perturbations}
		This section describes the energy estimate for the non-zero mode. Using the  fact
		$\int_{\mathbb T^2}(\phi_{\neq},\psi_{\neq},\zeta_{\neq})dx_2dx_3=0$, it can be known that the Poincar\'e's inequality holds for the non-zero mode
		\begin{align}\label{2025.6.08-1}
			\|(\phi_{\neq},\varphi_{\neq},\zeta_{\neq})\|_{L^2}
			\leq C \|\nabla^k (\phi_{\neq},\varphi_{\neq},\zeta_{\neq})\|_{L^2},
		\end{align}
		where $k\ge 1$.
		Taking $\Dn$ for the perturbation system \cref{sys-pertur}, one has
		\begin{align}\label{eqs201}
			\begin{cases}
				\p_t\phi_{\neq}+\mathring{\uf} \cdot \nabla \phi_{\neq}+\rhom \operatorname{div} \varphi_{\neq}=\mathcal {F}_{0\neq}, \\
				\p_t\varphi_{\neq}+ \mathring{\uf} \cdot \nabla \varphi_{\neq}+\frac{R}{\rhot}\left(\thetam \nabla \phi_{\neq} + \rhom \nabla \zeta_{\neq}  \right)=
				\frac{\mu}{\rhot}\Delta\varphi_{\neq}+\frac{\lambda+\mu}{\rhot}\nabla\dv\varphi_{\neq}
				+\mathcal{F}_{\neq},\\
				\p_t \zeta_{\neq}+ \mathring{\uf} \cdot \nabla \zeta_{\neq}+R \thetam \operatorname{div} \varphi_{\neq}=\frac{\kappa}{\rhot}\Delta\zeta_{\neq}
				+\mathcal {F}_{4\neq},
			\end{cases}
		\end{align}
		where 
		\begin{align*}
			&\mathcal{F}_{0\neq}= \mathring{\uf} \cdot\nabla\phi_{\neq} - \Dn\left(\uf \cdot \nabla\phi\right) + \rhom \dv \varphi_{\neq}- \Dn \left(\rho\dv \varphi \right) -\varphi_{\neq} \cdot \nabla \rhot - \phi_{\neq} \dv \ut,\\
			&\mathcal{F}_{\neq}= \mathring{\uf} \cdot \nabla \varphi_{\neq}-\Dn \left(\uf \cdot \nabla \varphi \right) + \frac{R}{\rhot}\left(\thetam \nabla \phi_{\neq} + \rhom \nabla \zeta_{\neq}  \right) - \frac{\nabla p_{\neq}}{\rhot} - \varphi_{\neq} \cdot \nabla \tilde{\uf} + \frac{1}{\rhot}\Dn\left(\frac{\phi \nabla p}{\rho} \right) \\
			&\qquad\;+\Dn\left[  \left( \frac{1}{\rho}-\frac{1}{\rhot} \right) \left( \mu \Delta \uf +(\lambda+\mu ) \nabla\dv\uf \right) \right],\\
			&\mathcal{F}_{4\neq}=  \mathring{\uf} \cdot \nabla \zeta_{\neq}-\Dn \left(\uf \cdot \nabla \zeta \right) + R \thetam \dv \varphi_{\neq} - R \Dn \left( \theta \dv \varphi \right) - \varphi_{\neq} \cdot \nabla\thetat - R\zeta_{\neq}\dv \ut+\kappa\Dn \left[ \left( \frac{1}{\rho}-\frac{1}{\rhot}\right)\Delta\theta \right] \\
			&\qquad\quad+\Dn\left[ \frac{1}{\rho}\left( \frac{\mu}{2}\abs{\nabla \uf + (\nabla \uf)^{t}}^2+\lambda(\dv\uf)^2 \right) - \frac{1}{\rhot} \left( \frac{\mu}{2}\abs{\nabla \tilde{\uf} + (\nabla \tilde{\uf})^{t}}^2+\lambda(\dv \tilde{\uf})^2 \right)   \right].
		\end{align*}
		The source terms  $\mathcal{F}_{0\neq}$, $\mathcal{F}_{\neq}$ and $\mathcal{F}_{4\neq}$ satisfy the following estimates.
		\begin{Lem}\label{lem9}
			Under the same assumptions of \cref{Thm-ape}, it holds
			\begin{align*}
				\|\nabla\mathcal{ F}_{0\neq}\|_{L^2}^2 + \sum_{i=1}^4\|\mathcal {F}_{i\neq}\|_{L^2}^2\le C\bar{\eta}\|(\nabla\phi_{\neq},\nabla^2\varphi_{\neq},\nabla^2\zeta_{\neq})\|_{L^2}^2.
			\end{align*}
		\end{Lem}
		\begin{proof}
			We first make an estimation for $\mathcal{F}_{\neq}$. With respect to $\mathcal{F}$, we first calculate  $\frac{R}{\rhot}\left(\thetam \nabla \phi_{\neq} + \rhom \nabla \zeta_{\neq}  \right) - \frac{\nabla p_{\neq}}{\rhot}$  as
			\begin{align}\label{fluid-part-1}
				&\norm{\frac{R}{\rhot}\left(\thetam \nabla \phi_{\neq} + \rhom \nabla \zeta_{\neq}  \right) - \frac{\nabla p_{\neq}}{\rhot}}_{L^2} \notag \\
				\leq&\norm{ \frac{R}{\rhot}  \left(  \px \phim \zeta_{\neq} +\px \thetam \phi_{\neq}  \right) }_{L^2}+\sum_{i=1}^3\norm{\frac{R}{\rhot} \left(\p_i\phi_{\neq} \zeta_{\neq} +\p_i\zeta_{\neq}\phi_{\neq} \right)}_{L^2} \notag\\
				\leq& C \etab  \norm{(\nabla\phi_{\neq},\nabla^2 \varphi_{\neq},\nabla^2 \zeta_{\neq})
				}_{L^2} .
			\end{align}
			Then we carry out the calculation for $\Dn \left[ \left( \frac{1}{\rho}-\frac{1}{\rhot}\right) \left(\mu\Delta \uf + (\lambda+\mu) \nabla \dv\uf \right)\right]$ as
			\begin{align}\label{fluid-part-2}
				&\norm{\Dn \left[ \left( \frac{1}{\rho}-\frac{1}{\rhot}\right) \left(\mu\Delta \uf + (\lambda+\mu) \nabla \dv\uf \right)\right] }_{L^2} \notag\\
				\leq&\norm{ \Dn\left( \frac{1}{\rho}-\frac{1}{\rhot}\right) \left(\mu\Delta\mathring{\uf} +(\lambda+\mu) \nabla \dv \mathring{\uf} \right)}_{L^2} + \norm{ \Do \left( \frac{\phi}{\rho\rhot}\right) \left(\mu\Delta \varphi_{\neq} + (\lambda+\mu) \nabla \dv \varphi_{\neq} \right)}_{L^2} \notag\\
				&+\norm{\Dn \left[\left( \frac{1}{\rho}-\frac{1}{\rhot}\right)_{\neq} \left(\mu\Delta \varphi_{\neq} + (\lambda+\mu) \nabla \dv \varphi_{\neq} \right)\right]}_{L^2} \notag\\
				\leq& C\etab  \norm{(\nabla\phi_{\neq},\nabla^2 \varphi_{\neq},\nabla^2 \zeta_{\neq})}_{L^2}.
			\end{align}
			The calculation of the remaining terms in $\mathcal{F}_{i\neq}$ is similar to the calculation of the above two terms. Combining \cref{fluid-part-1} and \cref{fluid-part-2}, one has
			\begin{align}\label{F-m-111}
				&\norm{\mathcal{F}_{\neq}}_{L^2}^2\leq C \etab\| (\nabla\phi_{\neq},\nabla^2\varphi_{\neq},\nabla^2\zeta_{\neq}
				)\|_{L^2}^2.
			\end{align}
			The estimates of $\mathcal{F}_{0\neq}$ and $\mathcal{F}_{4\neq}$ are the same as those of $\mathcal{F}_{\neq}$. Thus, using the same calculation method as that for $\mathcal{F}_{\neq}$ \cref{F-m-111}, we obtain
			\begin{align}\label{fluid-and-non-fluid-2}
				\begin{aligned}
					\|\nabla\mathcal{ F}_{0\neq}\|_{L^2}^2 + \sum_{i=1}^4\|\mathcal {F}_{i\neq}\|_{L^2}^2\le C \etab\|( \nabla\phi_{\neq},\nabla^2\varphi_{\neq},\nabla^2\zeta_{\neq})
					\|_{L^2}^2.
				\end{aligned}
			\end{align}
			Then we have completed the proof of \cref{lem9}.
		\end{proof}
		
		Based on \cref{lem9}, we provide the estimate of equation \cref{eqs201}. 
		
		\begin{Lem}\label{lem99} 
			Under the same assumptions of \cref{Thm-ape}, it holds
			\begin{align}\label{eqs230}
				&\frac{d}{dt}\|(\phi_{\neq},\varphi_{\neq},\zeta_{\neq})\|_{H^1}^2+c\|(\nabla\phi_{\neq},\nabla^2\varphi_{\neq},\nabla^2\zeta_{\neq})\|_{L^2}^2 \leq  0.
			\end{align}
		\end{Lem}
		\begin{proof}
			
			{\bf Step 1.} Multiplying \cref{eqs201}$_1$ by $R\frac{\thetam}{\rhom \rhot}\phi_{\neq}$, \cref{eqs201}$_2$ by $\varphi_{\neq}$, \cref{eqs201}$_3$ by $\frac{\rhom}{\thetam \rhot}\zeta_{\neq}$, one has
			\begin{align}\label{eqs210}
				\p_t\left(\frac{R\thetam}{2\rhom \rhot}\phi_{\neq}^2+\frac{|\varphi_{\neq}|^2}{2}+\frac{\rhom}{2\thetam \rhot}\zeta_{\neq}^2\right)+\frac{\mu}{\rhot}|\nabla\varphi_{\neq}|^2+\frac{\lambda+\mu}{\rhot}|\dv\varphi_{\neq}|^2+\frac{\rhom \kappa}{\thetam \rhot^2}|\nabla\zeta_{\neq}|^2=\mathcal J_0+\dv(\cdots),
			\end{align}
			where
			\begin{align*}
				\mathcal J_0:=&\left[\left(\frac{\thetam}{\rhom \rhot}\right)_t+\nabla\cdot\left(\frac{\thetam\mathring{\uf}}{\rhom \rhot}\right)\right]\frac{R\phi_{\neq}^2}{2}+\left[\left(\frac{\rhom}{\thetam \rhot}\right)_t+\nabla\cdot\left(\frac{\rhom\mathring{\uf}}{\thetam \rhot}\right)\right]\frac{\zeta_{\neq}^2}{2}+\frac{\px \um_1}{2} \abs{\varphi_{\neq}}^2 +\frac{R\thetam}{\rhom \rhot}\phi_{\neq} \mathcal{F}_{0\neq}+ \frac{\rhom}{\thetam \rhot} \zeta_{\neq} \mathcal{F}_{4\neq} \notag\\
				&-\px\left(\frac{1}{\rhot} \right)\left(\mu\sum_{j=1}^3\px\varphi_{j\neq}\varphi_{j\neq}+(\lambda+\mu)\dv\varphi_{\neq}\varphi_{1\neq}\right)-\px \left(\frac{\kappa\rhom}{\rhot^2\thetam}\right)\px\zeta_{\neq}\zeta_{\neq} + \varphi_{\neq} \cdot \mathcal{F}_{\neq}.
			\end{align*}
			According to  {\it a priori} assumptions \cref{apa} and applying H{\"o}lder's inequality, one has
			\begin{align}\label{eqs212}
				\int_{\Omega} \mathcal{J}_0 dx \leq C\etab \norm{(\nabla \phi_{\neq},\nabla \varphi_{\neq} ,\nabla \zeta_{\neq})}_{L^2}^2 + C \sum_{j=0}^4 \norm{\mathcal{F}_{j\neq}}_{L^2}^2.
			\end{align}
			Integrating \cref{eqs210} on $\Omega$ and combining \cref{eqs212}, it yields
			\begin{align}\label{eqs213}
				\frac{d}{dt} \norm{(\phi_{\neq},\varphi_{\neq},\zeta_{\neq})}_{L^2}^2+\norm{(\nabla \varphi _{\neq},\nabla \zeta _{\neq})}_{L^2}^2 \leq  C\etab\norm{\nabla\phi_{\neq}}_{L^2}^2+C\sum_{i=0}^4\norm{\mathcal{F}_{i\neq}}_{L^2}^2 .
			\end{align}
			
			{\bf Step 2.} We still need to estimate $\|\nabla\phi_{\neq}\|^2$. Taking \cref{eqs201}$_2\times\rhot\nabla\phi_{\neq}+\nabla{\cref{eqs201}_1}\times\frac{\lambda+2\mu}{\rhom}\nabla\phi_{\neq},$ one has
			\begin{align}\label{eqs220}
				&\p_t\left(\frac{\lambda+2\mu}{2\rhom}|\nabla\phi_{\neq}|^2+\rhot\varphi_{\neq}\cdot\nabla\phi_{\neq}\right)+R\thetam|\nabla\phi_{\neq}|^2=\mathcal J_1+\dv(\cdots),
			\end{align}
			where 
			\begin{align*}
				\mathcal J_1=&-\rhot \nabla \phi_{\neq}\cdot\mathring{\uf}\cdot\nabla\varphi_{\neq}-R\rhom\nabla\phi_{\neq}\cdot\nabla\zeta_{\neq}+\p_t\rhot \varphi_{\neq}\cdot\nabla\phi_{\neq}+\dv(\rhot\varphi_{\neq})\mathring{\uf}\cdot\nabla\phi_{\neq}+\dv(\rhot \varphi_{\neq})\rhom\dv \varphi_{\neq} \notag\\
				&+\p_t \left(  \frac{\lambda+2\mu}{2\rhom} \right)\abs{\nabla \phi_{\neq}}^2-(\lambda+2\mu)\frac{\px \mathring{\uf}}{\rhom}\cdot\nabla\phi_{\neq} \px \phi_{\neq} + (\lambda+2\mu)\px\left( \frac{\um_1}{\rhom} \right)\abs{\nabla \phi_{\neq}}^2+\rhot\varphi_{\neq}\cdot\nabla\mathcal{F}_{0\neq} \notag\\
				&+\frac{\lambda+2\mu}{\rhom}\nabla\phi_{\neq}\cdot \nabla \mathcal{F}_{0\neq}-(\lambda+2\mu)\frac{\px \rhom \dv \varphi_{\neq}}{\rhom}\px\phi_{\neq}+\rhot \mathcal{F}_{\neq}\cdot\nabla\phi_{\neq},
			\end{align*}
			where we have used
			\begin{align*}
				&\rhot \p_t \varphi_{\neq}\cdot\nabla\phi_{\neq}=\p_t (\rhot \varphi_{\neq}\cdot\nabla\phi_{\neq})-\p_t \rhot \varphi_{\neq}\cdot\nabla\phi_{\neq}-\rhot\varphi_{\neq}\cdot\nabla\p_t\phi_{\neq}\\
				=&\p_t(\rhot\varphi_{\neq}\cdot\nabla \phi_{\neq})-\p_t \rhot \varphi_{\neq}\cdot \nabla\phi_{\neq}-\dv(\rhot \varphi_{\neq}) \mathring{\uf}\cdot\nabla\phi_{\neq}-\dv(\rhot \varphi_{\neq})\rhom\dv \varphi_{\neq}-\rhot\varphi_{\neq}\cdot \nabla \mathcal{F}_{0\neq}.
			\end{align*}
			According to  {\it a priori} assumptions \cref{apa} and applying H{\"o}lder's inequality, one has
			\begin{align}\label{eqs222}
				\int_{\Omega}\mathcal J_1dx\le C\etab\|\nabla\phi_{\neq}\|_{L^2}^2 + C \left(\norm{\nabla \varphi_{\neq}}_{L^2}^2+\norm{ \nabla\zeta_{\neq}}_{L^2}^2 + \norm{\nabla \mathcal{F}_{0\neq}}_{L^2}^2 + \norm{\mathcal{F}_{\neq}}_{L^2}^2\right).
			\end{align}  
			Integrating \cref{eqs220} on $\Omega$ and combining \cref{eqs222}, it yields
			\begin{align}\label{eqs223}
				\frac{d}{dt}\int_{\Omega}&\left(\frac{\lambda+2\mu}{2\rhom}|\nabla\phi_{\neq}|^2+\rhot\varphi_{\neq}\cdot\nabla\phi_{\neq}\right)dx+\int_{\Omega}\frac{R\thetam}{2}|\nabla\phi_{\neq}|^2dx \notag\\
				&\le   C \left(\norm{\nabla \varphi_{\neq}}_{L^2}^2+\norm{ \nabla\zeta_{\neq}}_{L^2}^2 + \norm{\nabla \mathcal{F}_{0\neq}}_{L^2}^2 + \norm{\mathcal{F}_{\neq}}_{L^2}^2\right).
			\end{align}
			
			{\bf Step 3.} We estimate $\nabla\varphi_{\neq}$ and $\nabla\zeta_{\neq}$. Taking $\nabla\cref{eqs201}_1\times\frac{R\thetam}{\rhom \rhot}\nabla\phi_{\neq}+\nabla\cref{eqs201}_2\times\nabla\varphi_{\neq}+\nabla\cref{eqs201}_3\times\frac{\rhom}{\thetam\rhot}\nabla\zeta_{\neq},$ one has
			\begin{align}\label{eqs230-1}
				\p_t\Big(\frac{R\thetam}{2\rhom \rhot}\abs{\nabla\phi_{\neq}}&^2+\frac{|\nabla\varphi_{\neq}|^2}{2}+\frac{\rhom}{2\thetam \rhot}\abs{\nabla\zeta_{\neq}}^2\Big)+\frac{\mu}{\rhot}|\Delta\varphi_{\neq}|^2 \notag\\
				&+\frac{\lambda+\mu}{\rhot}|\nabla\dv\varphi_{\neq}|^2
				+\frac{\rhom \kappa}{\thetam \rhot^2}|\Delta\zeta_{\neq}|^2=\mathcal J_2+\dv(\cdots),
			\end{align}
			where 
			\begin{align*}
				\mathcal{J}_{2}=&\left[\left(\frac{\thetam}{\rhom \rhot}\right)_t+\nabla\cdot\left(\frac{\thetam\mathring{\uf}}{\rhom \rhot}\right)\right]\frac{R\abs{\nabla\phi_{\neq}}^2}{2}+\left[\left(\frac{\rhom}{\thetam \rhot}\right)_t+\nabla\cdot\left(\frac{\rhom\mathring{\uf}}{\thetam \rhot}\right)\right]\frac{\abs{\nabla\zeta_{\neq}}^2}{2}+\frac{\px \um_1}{2} \abs{\nabla\varphi_{\neq}}^2 \notag\\
				&-\frac{R\thetam\px\rhom}{\rhom \rhot} \dv \varphi_{\neq}\px \phi_{\neq}-\frac{R\rhom\px \thetam}{\thetam \rhot} \dv \varphi_{\neq}\px \zeta_{\neq} - \frac{R\thetam}{\rhom \rhot}\px \mathring{\uf} \cdot \nabla \phi_{\neq}\px \phi_{\neq} -\px \mathring{\uf} \cdot \nabla \varphi_{\neq} \cdot\px \varphi_{\neq}  \notag\\
				&-\frac{\rhom}{\thetam \rhot}\px \mathring{\uf} \cdot \nabla \zeta_{\neq}\px \zeta_{\neq} - R\px\left( \frac{\thetam}{\rhot} \right)\px\varphi_{\neq} \cdot \nabla \phi_{\neq} + R\px\left( \frac{\thetam}{\rhot} \right) \nabla \varphi_{1\neq} \cdot \nabla \phi_{\neq}+\frac{R\thetam}{\rhom \rhot}\nabla\phi_{\neq}\cdot \nabla \mathcal{F}_{0\neq} \notag\\
				&-R\px\left( \frac{\rhom}{\rhot} \right)\px \varphi_{\neq} \cdot \nabla \zeta_{\neq} + R\px\left( \frac{\rhom}{\rhot} \right) \nabla \varphi_{1\neq} \cdot \nabla \zeta_{\neq} - \px\left(\frac{\lambda+\mu}{\rhot} \right) \dv \varphi_{\neq} \Delta \varphi_{1\neq} + \nabla\varphi_{\neq} \cdot \nabla\mathcal{F}_{\neq} \notag\\
				&- \px\left(\frac{\lambda+\mu}{\rhot} \right) \dv \varphi_{\neq}\px \dv \varphi_{\neq} - \frac{\kappa}{\rhot}\px \left(\frac{\rhom}{\thetam\rhot} \right)\px \zeta_{\neq} \Delta \zeta_{\neq} + \frac{\rhom}{\thetam \rhot} \nabla\zeta_{\neq} \cdot \nabla\mathcal{F}_{4\neq}. 
			\end{align*}
			According to  {\it a priori} assumptions \cref{apa} and applying H{\"o}lder's inequality, one has
			\begin{align}\label{eqs232}
				\int_{\Omega}\mathcal J_2dx\le C\etab \left(\|\nabla\phi_{\neq}\|_{L^2}^2 + \norm{\nabla^2 \varphi_{\neq}}_{L^2}^2+\norm{ \nabla^2\zeta_{\neq}}_{L^2}^2 \right)+ C \left(\norm{\nabla \mathcal{F}_{0\neq}}_{L^2}^2 +\sum_{i=1}^4 \norm{\mathcal{F}_{i\neq}}_{L^2}^2\right).
			\end{align}  
			Integrating \cref{eqs230-1} on $\Omega$ and combining \cref{eqs232}, it yields
			\begin{align}\label{eqs233}
				\frac{d}{dt}& \norm{(\nabla \phi_{\neq},\nabla \varphi_{\neq},\nabla \zeta_{\neq})}_{L^2}^2 + \norm{(\nabla^2 \varphi_{\neq},\nabla^2 \zeta_{\neq})}_{L^2}^2
				\leq C\etab \|\nabla\phi_{\neq}\|_{L^2}^2 + C \left(\norm{\nabla \mathcal{F}_{0\neq}}_{L^2}^2 +\sum_{i=1}^4 \norm{\mathcal{F}_{i\neq}}_{L^2}^2\right).
			\end{align}
			Combining \cref{lem9}, \cref{eqs213}, \cref{eqs223},  and \cref{eqs233}, we have completed the proof of \cref{lem99}.
		\end{proof}

		\begin{flushleft}
			\textbf{Decay rate for the non-zero modes.}
		\end{flushleft}
		\begin{Lem}\label{lem-non-zeromode-decay}
			Under the same assumptions as \cref{Thm-ape}, one has the following estimates for the non-zero modes
			\begin{align}\notag
				\norm{(\phi_{\neq},\varphi_{\neq},\zeta_{\neq})}_{H^1}\leq \varepsilon e^{-ct}.
			\end{align}
		\end{Lem}
		\begin{proof}
			Note that the Poincar\'e inequality holds for $(\phi_{\neq},\varphi_{\neq},\zeta_{\neq})$ in \cref{2025.6.08-1}. Then, applying Gr{\"o}nwall's inequality to \cref{eqs230}, we have completed the proof of \cref{lem-non-zeromode-decay}.
		\end{proof}
		
		\subsection{Higher order estimates}
		The main objective of this section is to make energy estimates in the original perturbation equations and to verify the {\it a priori} assumptions.
		To close the {\it a priori} estimates, we need to return back to the original system \cref{sys-pertur}, that is
		\begin{equation}\label{sys-pertur1}
			\left\{
			\begin{aligned}
				&\pt\phi+\uf \cdot\nabla\phi+\rho div\varphi=\hat{R}_1,\\
				&\rho\pt\varphi+\rho\uf \cdot\nabla\varphi+R(\T\nabla\phi+\rho\nabla\zeta)=\mu \triangle\varphi+(\mu+\lambda)  \nabla div\varphi+\hat{\mathbf{R}}_2,\\
				&\rho\pt\zeta+\rho\uf \cdot\nabla\zeta+ R\rho\T div\varphi=\kappa \triangle\zeta+\hat{R}_3.
			\end{aligned}
			\right.
		\end{equation}
		Note that the energy estimates for the zero- and first-order derivatives are already done. In order to verify the {\it a priori} assumptions, we only need to perform further energy estimates for the second- and third-order derivatives. 
		
		We emphasise the strategy of proof. Since we already have estimates of lower order, it is most important to note the sign of the energy of the highest order derivative when estimating.
		Applying $\partial_{j} \partial_{i}(i, j=1,2,3)$ to the mass equation \cref{sys-pertur1}$_{1}$ and $\partial_{j}(j=$ $1,2,3$ ) to the $i$-th $(i=1,2,3)$ component of the momentum equation \cref{sys-pertur1}$_{2}$ and $\partial_{j}(j=$ $1,2,3$ ) to the energy equation \cref{sys-pertur1}$_{3}$, we have
		\begin{align}\label{eq-2or}
			\left\{\begin{aligned}
				&\pt\p_{ij}\phi+\uf \cdot\nabla\p_{ij}\phi+\rho div\p_{ij}\varphi=\p_{ij}\hat{R}_1+\hat{Q}_1,\\
				&\rho\pt\p_{j}\varphi+\rho\uf \cdot\nabla\p_{j}\varphi+R(\T\nabla\p_{j}\phi+\rho\nabla\p_{j}\zeta)=\mu \triangle\p_{j}\varphi+(\mu+\lambda)  \nabla div\p_{j}\varphi+\p_{j}\hat{\mathbf{R}}_2+\hat{Q}_2,\\
				&\rho\pt\p_{j}\zeta+\rho\uf \cdot\nabla\p_{j}\zeta+ R\rho\T div\p_{j}\varphi=\kappa \triangle\p_{j}\zeta+\p_{j}\hat{R}_3+\hat{Q}_3,
			\end{aligned}
			\right.
		\end{align}
		where 
		\begin{align*}
			\begin{aligned}
				\hat{Q}_1:=&-\p_{ij}\left(\uf\cdot\nabla\phi+\rho div\varphi\right)+\left(\uf \cdot\nabla\p_{ij}\phi+\rho div\p_{ij}\varphi\right),\\
				\hat{Q}_2:=&-\p_{j}\rho\pt\varphi-\p_{j}\left(\rho \uf\right) \cdot\nabla\varphi+R(\p_j\T\nabla\phi+\p_j\rho\nabla\zeta),\\
				\hat{Q}_3:=&-\p_{j}\rho\pt\zeta-\p_{j}(\rho \uf) \cdot\nabla\zeta- R\p_{j}(\rho\T) div\varphi.
			\end{aligned}
		\end{align*}
		By the {\it a priori} assumptions \cref{apa}, we have the result
		\begin{align}\label{sec-6-1}
			\|\nabla^k\hat{Q}_i\|_{L^2}\leq C(\bar{\delta}+\chi)\|\nabla(\phi,\varphi,\zeta)\|_{H^{k+1}},\qquad i=1,2,3.
		\end{align}
		It remains to estimate $\hat{R}_{i}$, i=1,2,3 \cref{2025.6.02-8}. Note that the highest order of derivatives for $(\phi,\varphi,\zeta)$ is zero order in $\hat{R}_1$, $\hat {\mathbf{R}}_2$, first order in $\hat{R}_3$. Thus, by {\it a priori} assumption \cref{apa}, one has
		\begin{align}\label{sec-6-2}
			&\|\nabla^k(\hat{R}_1,\hat{\mathbf{R}}_2)\|_{L^2}\leq C\bar{\delta} (1+t)^{-\frac{5}{2} -k}+C\bar{\eta} (1+t)^{-\frac{k-l+1}{2}}\sum_{l=0}^k\big\|\nabla^l(\phi,\varphi,\zeta)\big\|_{L^2},\\ 
			&\|\nabla^k\hat{R}_3\|_{L^2}\leq C\|\nabla^k(\hat{R}_1,\hat{\mathbf{R}}_2)\|_{L^2}+C\chi\big\|\nabla\varphi\big\|_{H^k(\Omega)}+C\bar{\eta} (1+t)^{-\frac{k-l+1}{2}}\sum_{l=0}^k\big\|\nabla^{l+1}\varphi\big\|_{L^2}.\label{sec-6-3}
		\end{align}
		Then we have the following result.
		
		\begin{Lem}\label{lem-original}Under the same assumptions of \cref{Thm-ape}, one has
			\begin{align}\notag
				\begin{aligned}
					&\|\nabla^2(\phi,\varphi,\zeta)\|_{H^1}^2\Big|_0^t+ \int_0^t 
					\|(\nabla^2\phi,\nabla^3 \varphi,\nabla^3\zeta)\|_{H^1}^2d\tau
					\leq C\bar{\delta}.
				\end{aligned}
			\end{align}
		\end{Lem}
		
		\begin{proof} We have divided the proof into three parts.
			\begin{flushleft}
				\textbf{Step 1. Estimates on $\|\nabla^2\phi\|_{L^2}$.}
			\end{flushleft}
			Next, multiplying the equation $\cref{eq-2or}_{2}$ by $ \nabla\partial_{j} \phi$ and integrating the result on $\Omega$  leads to
			\begin{align}\label{eq-2rho1}
				\begin{aligned}
					&\frac{\mathrm{d}}{\mathrm{d} t} \int_{\Omega} \rho \partial_{j}  \varphi_{i} \partial_{j} \partial_{i} \phi \mathrm{d} x+\int_{\Omega}  R \T\left|\partial_{j} \partial_{i} \phi\right|^{2} \mathrm{d} x \\
					=&\int_{\Omega} \pt\rho \partial_{j}  \varphi_{i} \partial_{j} \partial_{i} \phi \mathrm{d} x +\int_{\Omega} \partial_{i}\big(\rho \partial_{j}  \varphi_{i} \big) \partial_{j} \big(\uf \cdot\nabla\phi+\rho div\varphi-\hat{R}_1\big) \mathrm{d}x \\
					&-\int_{\Omega}\left[ \rho\uf \cdot\nabla\p_{j}\varphi+R\rho\nabla\p_{j}\zeta-\p_{j}\hat{R}_2-\hat{Q}_2\right] \cdot \nabla\partial_{j} \phi \mathrm{d} x +(2 \mu +\lambda ) \int_{\Omega} \partial_{j} \partial_{i} \phi \partial_{j} \partial_{i} \operatorname{div}  \varphi \mathrm{d} x,
				\end{aligned}
			\end{align}
			where we have used the following two facts:
			$$
			\begin{aligned}
				&\int_{\Omega} \rho \pt \partial_{j} \varphi_{i} \partial_{ij} \phi \mathrm{d} x=\frac{\mathrm{d}}{\mathrm{d} t} \int_{\Omega} \rho \partial_{j}  \varphi_{i} \partial_{ij}  \phi \mathrm{d} x-\int_{\Omega} \pt\rho \partial_{j}  \varphi_{i} \partial_{j} \partial_{i} \phi \mathrm{d} x -\int_{\Omega} \rho \partial_{j}  \varphi_{i} \partial_{j} \partial_{i} \pt\phi \mathrm{d}x\\
				=&\frac{\mathrm{d}}{\mathrm{d} t} \int_{\Omega} \rho \partial_{j}  \varphi_{i} \partial_{j} \partial_{i} \phi \mathrm{d} x-\int_{\Omega} \pt\rho \partial_{j}  \varphi_{i} \partial_{j} \partial_{i} \phi \mathrm{d} x +\int_{\Omega} \partial_{j}\big(\rho \partial_{j}  \varphi_{i} \big) \partial_{i} \pt\phi \mathrm{d}x \\
				=&\frac{\mathrm{d}}{\mathrm{d} t} \int_{\Omega} \rho \partial_{j}  \varphi_{i} \partial_{j} \partial_{i} \phi \mathrm{d} x-\int_{\Omega} \pt\rho \partial_{j}  \varphi_{i} \partial_{j} \partial_{i} \phi \mathrm{d} x -\int_{\Omega} \partial_{j}\big(\rho \partial_{j}  \varphi_{i} \big) \partial_{i} \big(\uf \cdot\nabla\phi+\rho div\varphi-\hat{R}_1\big) \mathrm{d}x,
			\end{aligned}
			$$
			and
			$$
			\int_{\Omega} \left(\mu  \Delta \partial_{j}  \varphi_{i}+(\mu+\lambda)  \partial_{j} \partial_{i} \operatorname{div}  \varphi\right)  \partial_{j} \partial_{i} \phi \mathrm{d} x=(2 \mu +\lambda ) \int_{\Omega} \partial_{j} \partial_{i} \phi \partial_{j} \partial_{i} \operatorname{div}  \varphi \mathrm{d} x.
			$$
			Next, multiplying the equation \cref{eq-2or}$_{1}$ by $\frac{1}{\rho}  \partial_{j} \partial_{i} \phi$ and integrating with respect to $x$ yields
			\begin{align}\label{eq-2rho2}
				\begin{aligned}
					&\frac{\mathrm{d}}{\mathrm{d} t} \int_{\Omega} \frac{\left|\partial_{ij} \phi\right|^{2}}{2 \rho} \mathrm{d} x=- \int_{\Omega} \partial_{ij} \phi \partial_{ijk}\varphi_{k} \mathrm{d} x+\int_{\Omega} \left(\pt\big(\frac{1}{2\rho}\big)+\dv\big(\frac{u}{2\rho}\big)\right)\left|\partial_{ij} \phi\right|^{2} \mathrm{d} x +\int_{\Omega}(\p_{ij}\hat{R}_1+\hat{Q}_1)\frac{\p_{ij}\phi}{\rho} \mathrm{d} x.
				\end{aligned}
			\end{align}
			Finally, adding \cref{eq-2rho1} and \cref{eq-2rho2}$\times(2\mu+\lambda)$ up, summing $i, j$ from 1 to 3, integrating the resulting equation over $(0, t)$, and applying \cref{sec-6-1,sec-6-2,sec-6-3}, gives
			\begin{align}\label{step1}
				\begin{aligned}
					&\left(\int_\Omega \frac{2\mu+\lambda}{2\rho}|\nabla^2\phi|^2+\rho\nabla\varphi\cdot\nabla^2\phi dx\right)+\int_0^t\int_{\Omega}R\theta|\nabla^2\phi|^2dx d\tau\\
					\leq&C\int_0^t\int_{\Omega}|\nabla^2\phi|\left(|\nabla^2(\varphi,\zeta)|+|\nabla^2\hat{R}_1|+|\nabla\hat{R}_2|+|\hat{Q}_1|+|\hat{Q}_2|+O(\delta+\chi)\left(|\nabla^2\phi|+|\nabla\varphi|\right)\right)dx d\tau\\
					&\qquad+C\int_0^t\int_{\Omega}|(\nabla\rho||\nabla\varphi|+|\nabla^2\varphi|)(|\nabla \uf||\nabla\phi|+|\nabla\varphi|+\abs{\nabla^2\varphi}+\abs{\nabla\hat{R}_1})dx d\tau\\
					\leq&\frac{R}{120}\int_0^t\int_{\Omega}\theta\abs{\nabla^2\phi}^2dxd\tau+C\bar{\delta},
				\end{aligned}
			\end{align} 
			where estimates on the lower order terms have been used.

			\begin{flushleft}
				\textbf{Step 2. Estimates on $\|(\nabla^2\varphi,\nabla^2\zeta)\|_{L^2}$.}
			\end{flushleft}
			Next, We estimate $\nabla^2\varphi$ and $\nabla^2\zeta$. 
			Multiplying \cref{eq-2or}$_2$ by $\left(-\Delta \partial_{i} \varphi\right)$, \cref{eq-2or}$_3$ by $\left(-\Delta \partial_{i} \zeta\right)$ and using \cref{sec-6-1}-\cref{sec-6-3}, we have
			\begin{align}\label{step2}
				\begin{aligned}
					&\left.\int_{\Omega} \frac{\rho}{2}\left|\nabla^{2} (\varphi,\zeta)\right|^{2} \mathrm{d} x\right|_{\tau=0} ^{\tau=t}+\int_{0}^{t} \int_\Omega\left(\mu\left|\nabla^{3} \varphi\right|^{2}+(\mu +\lambda)\left|\nabla^{2} \operatorname{div} \varphi\right|^{2}+\kappa|\nabla^3\zeta|^2\right) \mathrm{d} x \mathrm{d} \tau \\
					\leq&\int_{0}^{t} \int_{\Omega}\left|\rho\uf \cdot\nabla\p_{j}\varphi+R(\T\nabla\p_{j}\phi+\rho\nabla\p_{j}\zeta)+\p_{j}\hat{R}_2+\hat{Q}_2\right|| \Delta \partial_{i}  \varphi|+\frac{\abs{\pt\rho}}{2}|\nabla^2\varphi|^2 +\abs{\nabla\rho}\abs{\p_t \nabla\varphi}\abs{\nabla^2 \varphi} \mathrm{d} x \mathrm{d} \tau\\
					&+\int_{0}^{t} \int_{\Omega}\left|\rho\uf \cdot\nabla\p_{j}\zeta+R\rho\T div\p_{j}\varphi+\p_{j}\hat{R}_3+\hat{Q}_3\right|| \Delta \partial_{i}  \zeta|+\frac{\abs{\pt\rho}}{2}|\nabla^2\zeta|^2 +\abs{\nabla\rho}\abs{\p_t \nabla\zeta}\abs{\nabla^2 \zeta} \mathrm{d} x \mathrm{d} \tau\\
					\leq&\frac{\min\{\mu,\kappa,\lambda+\mu\}}{120}\int_0^t\|\nabla^3(\varphi,\zeta)\|_{L^2}^2d\tau+C\bar{\delta}.
				\end{aligned}
			\end{align}
			\begin{flushleft}
				\textbf{{Step 3. Third Order Estimate.}}
			\end{flushleft}
			
			To close the {\it a priori} estimate, we need to estimate the third order derivatives. Applying $\p_k$, $k=1,2,3$ to \cref{eq-2or}$_1$ and $\p_i$, $i=1,2,3$ to \cref{eq-2or}$_2$ and \cref{eq-2or}$_3$, we have
			\begin{align}\label{eq-3or}
				\left\{\begin{aligned}
					&\pt\p_{ijk}\phi+\uf \cdot\nabla\p_{ijk}\phi+\rho div\p_{ijk}\varphi=\p_{ijk}\hat{R}_1+\p_k\hat{Q}_1+\hat{Q}^{*}_1,\\
					&\rho\pt\p_{ij}\varphi+\rho\uf \cdot\nabla\p_{ij}\varphi+R(\T\nabla\p_{ij}\phi+\rho\nabla\p_{ij}\zeta)=\mu \triangle\p_{ij}\varphi+(\mu+\lambda)  \nabla div\p_{ij}\varphi+\p_{ij}\hat{R}_2+\p_i\hat{Q}_2+\hat{Q}^{*}_2,\\
					&\rho\pt\p_{ij}\zeta+\rho\uf \cdot\nabla\p_{ij}\zeta+ R\rho\T div\p_{ij}\varphi=\kappa \triangle\p_{ij}\zeta+\p_{ij}\hat{R}_3+\p_i\hat{Q}_3+\hat{Q}^{*}_3,
				\end{aligned}
				\right.
			\end{align}
			where  
			\begin{align*}
				\begin{aligned}
					\hat{Q}^{*}_1:=&-\p_{k}\uf\cdot\nabla\p_{ij}\phi-\p_k\rho div\p_{ij}\varphi,\\
					\hat{Q}^{*}_2:=&-\p_{i}\rho\pt\p_j\varphi-\p_{i}\left(\rho \uf\right) \cdot\nabla\p_j\varphi+R(\p_i\T\nabla\p_j\phi+\p_i\rho\nabla\p_j\zeta),\\
					\hat{Q}^{*}_3:=&-\p_{i}\rho\pt\p_j\zeta-\p_{i}(\rho \uf) \cdot\nabla\p_j\zeta- R\p_{i}(\rho\T) div\p_j\varphi.
				\end{aligned}
			\end{align*}
			By the {\it a priori} assumptions \cref{apa}, we have
			\begin{align}\label{sec-6-4}
				\|\hat{Q}^{*}_i\|_{L^2}\leq C(\bar{\delta}+\chi)\|\nabla(\phi,\varphi,\zeta)\|_{H^{2}},\qquad i=1,2,3.
			\end{align}
			Firstly, we estimate $\|\nabla^3\phi\|_{L^2}^2$. Similarly, multiplying \cref{eq-3or}$_2$ by $\nabla\p_{ij}\phi$, one has
			\begin{align}\label{sec-6-5}
				&\frac{\mathrm{d}}{\mathrm{d} t} \int_{\Omega} \rho \partial_{ij}  \varphi_{k} \p_{ijk} \phi \mathrm{d} x+\int_{\Omega}  R \T\left|\partial_{ijk} \phi\right|^{2} \mathrm{d} x \notag\\
				=&\int_{\Omega} \pt\rho \partial_{ij}  \varphi_{k} \partial_{ijk} \phi \mathrm{d} x +\int_{\Omega} \partial_{k}\big(\rho \partial_{ij}  \varphi_{k} \big)  \big(\uf \cdot\nabla\p_{ij}\phi+\rho div\p_{ij}\varphi-\p_{ij}\hat{R}_1-\hat{Q}_1\big) \mathrm{d}x \notag \\
				&-\int_{\Omega}\left[ \rho\uf \cdot\nabla\p_{ij}\varphi+R\rho\nabla\p_{ij}\zeta-\p_{ij}\hat{R}_2-\p_{i}\hat{Q}_2-\hat{Q}^{*}_2\right] \cdot \nabla\partial_{ij} \phi \mathrm{d} x +(2 \mu +\lambda ) \int_{\Omega}  \p_{ijk}\phi \partial_{ijk} \operatorname{div}  \varphi \mathrm{d} x.
			\end{align}
			Next, multiplying \cref{eq-3or}$_1$ by $\frac{\p_{ijk}\phi}{\rho}$, it follows
			\begin{align}\label{sec-6-6}
				\frac{\mathrm{d}}{\mathrm{d} t} \int_{\Omega} \frac{\left|\partial_{ijk} \phi\right|^{2}}{2 \rho} \mathrm{d} x=&- \int_{\Omega} \partial_{ijk} \phi \partial_{ijk}\dv\varphi \mathrm{d} x\notag\\
				&+\int_{\Omega} \left(\pt\big(\frac{1}{2\rho}\big)+\dv\big(\frac{\uf}{2\rho}\big)\right)\left|\partial_{ijk} \phi\right|^{2} \mathrm{d} x +\int_{\Omega}(\p_{ijk}\hat{R}_1+\p_k\hat{Q}_1+\hat{Q}^{*}_1)\frac{\p_{ijk}\phi}{\rho} \mathrm{d} x.
			\end{align}
			Adding \cref{sec-6-5} and \cref{sec-6-6}$\times(2\mu+\lambda)$ up, summing $i, j$ from 1 to 3, integrating the resulting equation over $(0, t)$, and applying \cref{sec-6-1}-\cref{sec-6-3} and \cref{sec-6-4}, gives
			\begin{align}\label{step31}
				\begin{aligned}
					&\left(\int_\Omega \frac{2\mu+\lambda}{2\rho}|\nabla^3\phi|^2+\rho\nabla^2\varphi\cdot\nabla^3\phi dx\right)+\int_0^t\int_{\Omega}R\theta|\nabla^3\phi|^2dxd\tau\\
					\leq&\frac{R}{120}\int_0^t\int_{\Omega}\theta|\nabla^3\phi|^2dxd\tau+C \int_0^t\|\nabla^3(\varphi,\zeta)\|_{L^2}^2d\tau+C\bar{\delta}.
				\end{aligned}
			\end{align} 
			Finally, we estimate $\nabla^3\varphi$ and $\nabla^3 \zeta$. Using the same argument as \cref{step2},
			multiplying \cref{eq-3or}$_2$ by $-\Delta\p_{ij} \varphi$, \cref{eq-3or}$_3$ by $-\Delta\p_{ij}\zeta$, and using \cref{sec-6-1}-\cref{sec-6-3} and \cref{sec-6-4}, we have
			\begin{align}\label{step32}
				\begin{aligned}
					&\left.\int_{\Omega} \frac{\rho}{2}\left|\nabla^{3} (\varphi,\zeta)\right|^{2} \mathrm{d} x\right|_{\tau=0} ^{\tau=t}+\int_{0}^{t} \int_{\Omega}\left(\mu\left|\nabla^{4} \varphi\right|^{2}+(\mu +\lambda)\left|\nabla^{3} \operatorname{div} \varphi\right|^{2}+\kappa|\nabla^4\zeta|^2\right) \mathrm{d} x \mathrm{d} \tau \\
					\leq&\frac{\min\{\mu,\kappa,\lambda+\mu\}}{120}\|\nabla^4(\varphi,\zeta)\|_{L^2}^2+C\bar{\delta}.
				\end{aligned}
			\end{align} 
			Combining \cref{step1,step2,step31,step32}, we have proved \cref{lem-original}.
		\end{proof}
		Combining \cref{lem-zeromode-decay}, \cref{lem-non-zeromode-decay} and \cref{lem-original}, we have proved \cref{Thm-ape}.
		\section{Proof of \cref{mt0}: the zero initial mass}
		
		The differences between  \cref{mt} and \cref{mt0} 
		lie in the estimates of zero modes. We show these estimates and omit similar ones.
		
		We start from \cref{eq-ecd} and $\Do$\cref{NS-ln} since there are no diffusion waves if the initial mass is zero. Using the same notations as \cref{pert-anti}, one has
		\begin{align*}
			\left\{\begin{array}{l}
				\pt\Phi+\px\Psi_{1}=0, \\ 
				\pt\Psi_{1}+(\gamma-1)\px W=\left(2\mu+\lambda\right)\px \left(\mathring{u}_{1}- \ub_{1}\right)+\mathcal{Q}_{21}^{\#}-Q_1, \\ 
				\pt\Psi_{i}=\mu \px\left(\mathring{u}_{i}- \ub_{i}\right)+\mathcal{Q}_{2i}^{\#},\quad i=2,3, \\ 
				\pt W+\gamma\thetab\px\Psi_1= \kappa \px\left(\mathring{\theta}-{\thetab}\right) + \mathcal{Q}_{3}^{\#}-Q_2,\end{array}\right.
		\end{align*}
		where
		\begin{align*}
			\begin{aligned}
				\mathcal{Q}_{21}^{\#}:=&\bigg(-\mr{\mathcal{P}}(U)+\mr{\mathcal{P}}(\bar{U})\bigg),\qquad
				\mathcal{Q}_{2i}^{\#}:=\bigg(-\D_0(\frac{{m}_{1} {m}_{i}}{{\rho}})+\D_0(\frac{{\mb}_{1} {\mb}_{i}}{{\rhob}})\bigg),\quad i=2,3,\\
				\mathcal{Q}_{3}^{\#}:=&\bigg(\gamma\tilde{\theta}\px\Psi_1-\mr{\mathcal{N}}_{31}(U)+\mr{\mathcal{N}}_{31}(\bar{U})\bigg).
			\end{aligned}
		\end{align*}
		Then, similar to \cref{transformation}, we have the following perturbation system
		\begin{align}\label{sys-Eu0}
			\left\{\begin{array}{l}
				\pt\bar{\Phi}+\thetab \px\bar{\Psi}_{1}=J^{\#}_1, \\
				\pt\bar{\Psi}_{1}+ (\gamma-1)\px\left(\bar{W}+\bar{\Phi}\right)
				=\frac{(2\mu+\lambda)}{\rhob}  \px^2\bar{\Psi}_{1}+J^{\#}_{21}+\mathcal{Q}^{\#}_{21}-Q_1, \\
				\pt\bar{\Psi}_{i}=\frac{\mu}{\rhob}  \px^2\bar{\Psi}_{i} 
				+J^{\#}_{2i}+\mathcal{Q}^{\#}_{2i}, i=2,3, \\
				\pt\bar{W}+ (\gamma-1)  {\thetab} \px\bar{\Psi}_{1}=\frac{\kappa}{\rhob}  \px^2\bar{W} 
				+J^{\#}_{3}+\mathcal{Q}^{\#}_{3}-Q_2,
			\end{array}\right.
		\end{align}
		where
		\begin{align*}
			J_1^{\#}&:=\pt{\thetab}\Phi-\thetab\mathcal{B}_{1r},\qquad J_{2i}^{\#}:=-\mathcal{B}_{id}+\left[\mu\px(\um_i-\bar{u}_i)-\frac{\mu}{\rhob}\px^2\bar{\Psi}_i\right]:=J_{2i}^{{\#}(1)}+J_{2i}^{{\#}(2)},\qquad i=2,3,\\
			J_{21}^{\#}&:=\left[-\mathcal{B}^{\#}_{1d}-(\gamma-1)\mathcal{B}_{wr}^{\#}\right]+\left[(2\mu+\lambda)\px(\mr{u}_1-\tilde{u}_1)-\frac{(2\mu+\lambda)}{\rhob}\px^2\bar{\Psi}_1\right]:=J_{21}^{{\#}(1)}+J_{21}^{{\#}(2)},\\
			J^{\#}_{3}&:=\left[-\mathcal{B}_{wd}^{\#}-\gamma\thetab\mathcal{B}_{1r}^{\#}\right]+\left[\kappa\px\left(\thetam-\thetab\right)-\frac{\kappa}{\tilde{\rho}}\px^2\bar{W}\right]:=J_{3}^{{\#}(1)}+J_{3}^{{\#}(2)},\\
			\mathcal{B}^{\#}_{ir}&:=\ub_{i} \px\Phi+ \px\ub_{i} \Phi,\quad \mathcal{B}_{wr}^{\#}:=\px\left(\mathbf{{\ubb}} \cdot {\bar\Psi}+\frac{|\mathbf{{\ubb}}|^{2}}{2} \Phi\right),\quad \mathcal{B}^{\#}_{id}:=- \ub_{i}\px\Psi_{1} + \pt\ub_{i} \Phi,\quad i=2,3,\\
			\mathcal{B}_{wd}^{\#}&:=-\frac{| \mathbf{\ub}|^{2}}{2}\px\Psi_{1}+ \bar{\uf} \cdot \pt\bar{\Psi}+\pt\bar{\uf}\cdot\bar{\Psi}+\pt\left( {\thetab}+\frac{| \mathbf{\ub}|^{2}}{2}\right) \Phi.
		\end{align*}
		We set the following {\it a priori} assumption
		\begin{align}\label{apa0}
			\left\{\begin{aligned}
				&\sup_{0\leq t\leq T}\Big\{\left\|(\Phi,\Psi, W)\right\|^2_{L^\infty}+\ln^{-1}(2+t)\left[(1+t)\norm{(\phi,\varphi,\zeta)}_{L^2}^2
				+(1+t)^{2}\norm{(\px\phi,\px\varphi,\px\zeta)}_{L^2}^2\right]+\left\|(\phi,\varphi,\zeta)\right\|_{H^3}^2\Big\}\leq \chi^2,\\
				&\sup_{0\leq t\leq T}e^{ct}\norm{(\phi_{\neq},\varphi_{\neq},\zeta_{\neq})}^2_{H^1}\le \varepsilon^2+\delta.
			\end{aligned}\right.
		\end{align} 
		Then we need to close the {\it a priori} assumptions. We prove the following {\it a priori} estimate.
		
		\begin{Prop}[{\it a priori} estimate]\label{Thm-ape-0}
			Assume that $(\phi,\varphi,\zeta)$ is the unique solution given in \cref{Thm-local} and satisfies the {\it a priori} assumptions \cref{apa0}. Then the following estimates hold
			\begin{align*}
				&\ln^{-1}(2+t)\left[\left\|({\Phi},{\Psi},{W})\right\|_{L^2}^2+(1+t)\left\|({\phi},{\varphi},{\zeta})\right\|_{L^2}^2+(1+t)^{2}\left\|\px({\phi},{\varphi},{\zeta})\right\|_{L^2}^2\right]+\sum_{i=2}^{3}\left\|\px^{i}({\phi},{\varphi},{\zeta})\right\|_{L^2}^2\leq C(\delta+\varepsilon^2),\\
				&\qquad\left\|(\phi_{\neq},\varphi_{\neq},\zeta_{\neq})\right\|_{H^1}^2\leq \varepsilon^2e^{-c(1+t)},
			\end{align*}
			where $C,c>0$ are the universal constants independent of any small parameters in this paper.
		\end{Prop}
		
		By the same argument as \cref{mathcal-Q} and \cref{J}, one has
		\begin{Lem}
			Under the same assumptions as \cref{Thm-ape-0}, one has
			\begin{align*}
				&\abs{\px^k\mathcal{Q}^{\#}_{2i}}\leq C\left(\mathcal{T}^{{\#}(k)}+\mathcal{Z}^{(k)}\right),\quad k=0,1,2,\ \ i=1,2,3,\\
				&\abs{\px^k\mathcal{Q}_3^{{\#}}}\leq C\big(\mathcal{T}^{{\#}(k)}+\mathcal{T}^{{\#}(k+1)}+\mathcal{Z}^{(k)}+\mathcal{Z}^{(k+1)}\big),\quad k=0,1,\\
				&\abs{\px^{k}J^{\#}_1,\px^kJ^{{\#}(1)}_{2i}}\leq C\mathcal{D}^{{\#}(k)},\ \ k=0,1,2,\ \ i=1,2,3,\\
				&\abs{\px^{k}J_{2i}^{{\#}(2)},\px^{k}J_{3}^{{\#}(2)}}	\leq C\big(\mathcal{D}^{{\#}(k+1)}+\mathcal{T}^{{\#}(k+1)}+\mathcal{Z}^{(k+1)}\big),\ \ k=0,1,\ \ i=1,2,3,\\
				&\abs{\px^{k}J_{3}^{{\#}(1)}}	\leq C\big(\mathcal{D}^{{\#}(k)}+\mathcal{D}^{{\#}(k+1)}+\mathcal{T}^{{\#}(k+1)}+\mathcal{Z}^{(k+1)}\big),\qquad k=0,1.
			\end{align*} 
			Compared with \cref{useful-notation}, the definition of $V$ in $\mathcal{D}^{\#},\mathcal{T}^{\#}$ is $(\bar{\Phi},\bar{\Psi},\bar{W})^{t}$, and $D_{-\alpha}$ is changed to $\omega_{-\alpha}$.
		\end{Lem}
		
		By the same definitions as those in \cref{EZM}, we arrive at
		\begin{align}\label{eq-diaBz}
			\pt B+\Lambda \px B=\bar{L}A_2\bar{R}\px ^2B+2\bar{L}A_2\px \bar{R}\px B+\left[\left(\pt\bar{L}+\Lambda \px \bar{L}\right)R+\bar{L}A_2 \px ^2\bar{R}\right]B+\bar{L} A_3,
		\end{align}
		where
		\begin{align*}
			&\Lambda:=\text{diag}\{\bar{\lambda}_1,0,\bar{\lambda}_3\},\qquad\bar{\lambda}_1=-\sqrt{\gamma(\gamma-1)\thetab},\quad \lambda_2=0,\quad\bar{\lambda}_3=\sqrt{\gamma(\gamma-1)\thetab},\\
			& A_3:=(A_{31},A_{32},A_{33})^{t}
			=(J_1^{\#},J_{21}^{\#}+\mathcal{Q}_{21}^{\#}-Q_1,J_{3}^{\#}+\mathcal{Q}_3^{\#}-Q_2)^t,
		\end{align*}
		and the corresponding eigenvectors are
		\begin{align*}
			&\bar{L}:=\left(\begin{array}{ccc}
				\frac{1}{\sqrt{2\gamma}} & \frac{1}{\sqrt{2\gamma}}\frac{\bar{\lambda}_1}{\gamma-1} &\frac{1}{\sqrt{2\gamma}} \\
				\sqrt{\frac{\gamma-1}{\gamma}} & 0 & -\frac{1}{\sqrt{\gamma(\gamma-1)}}\\
				\frac{1}{\sqrt{2\gamma}} & -\frac{1}{\sqrt{2\gamma}}\frac{\bar{\lambda}_1}{\gamma-1} &\frac{1}{\sqrt{2\gamma}}
			\end{array}\right),\quad
			&\bar{R}:=\left(\begin{array}{ccc}
				\frac{1}{\sqrt{2\gamma}} & \sqrt{\frac{\gamma-1}{\gamma}} &\frac{1}{\sqrt{2\gamma}} \\
				\frac{1}{\sqrt{2\gamma}}\frac{\bar{\lambda}_1}{\thetab} & 0 & -\frac{1}{\sqrt{2\gamma}}\frac{\bar{\lambda}_1}{\thetab}\\
				\frac{\gamma-1}{\sqrt{2\gamma}} & -\sqrt{\frac{\gamma-1}{\gamma}} &\frac{\gamma-1}{\sqrt{2\gamma}}
			\end{array}\right).
		\end{align*}
		Direct calculations yield
		\begin{align}
			&\Big(\bar{L} A_3\Big)_2=\sqrt{\frac{\gamma-1}{\gamma}}J_1^{\#}-\frac{1}{\sqrt{\gamma(\gamma-1)}}(J_3^{\#}+\mathcal{Q}_3^{\#}-Q_2),\label{key}\\
			&\Big(\bar{L} A_3\Big)_i=\sqrt{\frac{1}{2\gamma}}J_1^{\#}+\sqrt{\frac{1}{2\gamma}}\frac{\bar{\lambda}_i}{\gamma-1}(J_{21}^{\#}+\mathcal{Q}_{21}^{\#}-Q_1)+\frac{1}{\sqrt{2\gamma}}(J_3^{\#}+\mathcal{Q}_3^{\#}-Q_2), \quad \text{for}\quad i=1,3.\label{key-111}
		\end{align}
		Then, the  error term in $(\bar{L}A_3)_2$ becomes $\frac{Q_2}{\sqrt{\gamma(\gamma-1)}}\lesssim  \abs{\px \rhob}^3+\abs{\px \rhob}\abs{\px^2 \rhob}$, which is the key observation for improving the decay rate. Using the same notation as \cref{11-1}-\cref{33-1}, we have
		\begin{Lem}
			Under the same assumptions as \cref{Thm-ape-0}, one has
			\begin{align}
				&\frac{d}{dt}\big(\sum_{i=0}^2E_i\big)+\sum_{i=0}^2(K_i+G_{i})\leq C\bar{\eta} \int_{\R} \omega_{-1} (\bar{\Phi}^2+\bar{\Psi}^2+\bar{W}^2)dx_1+C\delta(1+t)^{-1},\label{111}\\
				&\frac{d}{dt}\big(\sum_{i=1}^2{E}_i\big)+\sum_{i=1}^2(K_i+G_{i})\leq C\bar{\eta}(1+t)^{-1}\big(\sum_{i=1}^2{E}_i+G_0\big)+C\delta(1+t)^{-2}+C\bar{\eta}(1+t)^{-2}{E}_0,\label{222}\\
				&\frac{d}{dt}{E}_2+{K}_2+G_{2}\leq C\bar{\eta}\sum_{i=0}^2(1+t)^{-(3-i)}{E}_{i}+C\delta(1+t)^{-3}+C\bar{\eta}\big[(1+t)^{-2}G_0+(1+t)^{-1}G_1\big],\label{333}
			\end{align}
			where the definitions of $E_i$, $K_i$ and $G_i$ are similar to \cref{sec-n-4-t}-\cref{sec-n-4} and \cref{sec-n-3}.
		\end{Lem}
		\begin{proof}
			Applying $\px^k$, $k=0,1,2$ to \cref{eq-diaBz}, one has
			\begin{align}\notag
				\begin{aligned}
					\px^kB_t+\Lambda \px^{k+1}B=\bar{L} A_2 \bar{R} \px^{k+2}B+\mathcal{M}_k,
				\end{aligned}
			\end{align}
			where
			\begin{align}
				\begin{aligned}
					\mathcal{M}_k:=&\sum_{j=1}^{k}\bigg[-\px^j\Lambda\px^{k-j+1}B+\px^j(\bar{L}A_2\bar{R})\px^{k-j+2}B\bigg]+2\sum_{i=0}^{k}\px^i\big(\bar{L} A_2 \px\bar{R}\big) \px^{k-i+1}B\\
					&+\sum_{i=0}^{k}\bigg\{\px^i\big[\left(\bar{L}_t+\Lambda \px\bar{L}\right) \bar{R}+\bar{L} A_2 \bar{R}_{x_1 x_1}\big] \px^{k-i}B \bigg\}+\px^{k}\left(\bar{L} A_3\right)\notag\\
					\leq&C\sum_{j=1}^{k+1}\left(\abs{\px^j \rhob} \abs{\px^{k-j+1}b_1},0,\abs{\px^j\rhob} \abs{\px^{k-j+1}b_3}\right)^{t}
					+C\delta \sum_{j=1}^{k+2}\abs{\omega_{-\frac{j}{2}}\px^{k-j+2}B}+\px^{k}\left(\bar{L} A_3\right):=\sum_{i=1}^{3}\mathcal{M}_{k}^{(i)}.\notag
				\end{aligned}
			\end{align}
			Using the same method as \cref{sec-3-1}, one has
			\begin{align*}
				\begin{aligned}
					&\int_\R\left(\frac{v_1^n}{2} \px^kb_1^2+\frac{1}{2} \px^kb_2^2+\frac{v_1^{-n}}{2} \px^kb_3^2\right)_t+\bar{B}^{(k)}_{x_1} A_4 \px^{k+1}Bdx_1
					+\int_{\R} a_1\px^kb_1^2+ a_3\px^kb_3^2dx_1\\
					=&\int_{\R}\bar{B}^{(k)}\px A_{4}\px^{k}\px B+\bigg[\left(\frac{v_1^n}{2}\right)_t \abs{\px^kb_1}^2+\left(\frac{v_1^{-n}}{2}\right)_t \abs{\px^kb_3}^2\bigg]+\bar{B}^{(k)}\mathcal{M}_kdx_1:=I^{\#}_1+I^{\#}_2+I^{\#}_3.
				\end{aligned}
			\end{align*}
			Same as \cref{ds}-\cref{I12} and applying {\it a priori} assumptions \cref{apa0}, one has
			\begin{align*}
				&\int_{\R} \px \left(\bar{B}^{(0)}-B^t\right)A_4\px Bdx_1 \leq C(\chi+\delta^{\frac12})\int_{\R} \omega_{-1}\left(\bar{\Phi}^2+\abs{\bar{\Psi}}^2+\bar{W}^2 \right)dx_1+C(\chi+\delta^{\frac12})K_0,\\
				&\int_{\R} \px \left(\bar{B}^{(k)}-\px^{k}B^t\right)A_4\px^{k+1}Bdx_1 \leq C(\chi+\delta^{\frac12})(1+t)^{-1}E_k+C(\chi+\delta^{\frac12})K_k,\quad \text{for} \quad k\ge 1,\\ 
				&I_1^{\#}+I_2^{\#}\leq C(\chi+\delta^{\frac12})\int_{\R} \omega_{-1}\left(\bar{\Phi}^2+\abs{\bar{\Psi}}^2+\bar{W}^2 \right)dx_1+C(\chi+\delta^{\frac12})K_0, \quad \text{for} \quad k=0,\\ 
				&I_1^{\#}+I_2^{\#}\leq C(\chi+\delta^{\frac12})(1+t)^{-1}E_k+C(\chi+\delta^{\frac12})K_k,\quad \text{for} \quad k\ge 1.
			\end{align*}
			For $I_3^{\#}$, by \cref{key,key-111}, we only provide the estimation of the error term 
			$$\mathcal{M}_0^{(3e)}:=(-\sqrt{\frac{1}{2\gamma}} \frac{\bar{\lambda}_1}{\gamma-1}Q_1-\frac{1}{\sqrt{2\gamma}}Q_2,\;\frac{1}{\sqrt{\gamma(\gamma-1)}}Q_2,\;-\sqrt{\frac{1}{2\gamma}} \frac{\bar{\lambda}_3}{\gamma-1}Q_1-\frac{1}{\sqrt{2\gamma}}Q_2)$$ 
			in $\mathcal{M}_k^{(3)}$, 
			one can find that the decay rate of the error term in the second equation has been improved
			\begin{align*}
				&\sum_{i=1,3}\int_{\R} \px^k\left( \sqrt{\frac{1}{2\gamma}} \frac{\bar{\lambda}_i}{\gamma-1}Q_1+\frac{1}{\sqrt{2\gamma}}Q_2\right) \px^k b_i dx_1\leq C\int_{\R}\px \rhob\left[(\px^kb_1)^2+(\px^kb_3)^2\right]dx_1+C\delta(1+t)^{-k-1},\\
				&  \frac{1}{\sqrt{\gamma(\gamma-1)}}\int_{\R} \px^kQ_2\px^kb_2 dx_1\leq C\delta\int_{\R}\omega_{-1}(\px ^kb_2)^2dx_1+C\delta(1+t)^{-k-\frac{3}{2}}.
			\end{align*}
			And the calculation of the remaining terms is the same as that of the estimation of \cref{IL12}-\cref{2025-10-5}.  Then  we have
			\begin{align}\notag
				&\frac{d}{dt}\big(\sum_{i=0}^2E_i\big)+\sum_{i=0}^2(K_i+G_{i})\leq C\delta(1+t)^{-1}+C\etab\int_{\R}\omega_{-1}(\bar{\Phi}^2+\bar{\Psi}^2+\bar{W}^2)dx_1.
			\end{align}
			Thus, we prove \cref{111}. The proof of \cref{222,333} is similar to that of \cref{22-1,33-1}, so we omit it here.
		\end{proof}
		
		To obtain the optimal decay rate, one cannot simply bound $\int_{\R} \omega_{-1}(\bar{\Phi}^2+\abs{\bar{\Psi}}^2+\bar{W}^2)dx_1$ by $(1+t)^{-1}\int_{\R} (\bar{\Phi}^2+\abs{\bar{\Psi}}^2+\bar{W}^2)dx_1$ which is out of control. Instead,  we need to use the following elementary inequality that involves the heat kernel as a weight function.
		
		\begin{Lem}\cite{HLM}\label{heat0}
			For $0<T\le+\infty,$ assume that $h(x_1,t)$ satisfies 
			\begin{align}\notag
				h_{x_1}\in L^2(0,T;L^2(\mathbb R)), \ \ h_t\in L^2(0,T;H^{-1}(\mathbb R)).
			\end{align}
			Then the following estimate holds:
			\begin{align}\notag
				\begin{aligned}
					\int_0^T\int_{\mathbb R}h^2\omega_{-1}d{x_1}dt\le4\pi\|h(0)\|^2+4\pi c^{-1}\int_0^T\|h_{x_1}(x_1,t)\|^2dt+8c\int_0^T\langle{h_t,hg^2}\rangle_{H^{-1}\times H^1}dt,
				\end{aligned}
			\end{align}
			where $g=\int_{-\infty}^{x_1} \omega_{-\frac12}(y_1,t)dy_1$.
		\end{Lem}
		
		Then we give the following Poincar\'e type inequality, which plays an important role in our energy estimates. 
		\begin{Lem}\label{lem5}
			For $\omega_{-n}$ defined in \cref{errors}, and assume \cref{apa0} holds. Then there exists some positive constant $C$ depending on $\alpha$ such that 
			\begin{align}\notag
				\int_0^t\int_{\mathbb R}\left(\bar{\Phi}^2+\abs{\bar{\Psi}}^2+\bar{W}^2\right)\omega_{-1}dx_1d\tau\le C+C\int_0^t\|\px(\bar{\Phi},\bar{\Psi},\bar{W})\|_{H^1}^2d\tau.
			\end{align}
		\end{Lem}
		\begin{proof}
			Defining
			$$f(x_1,t)=\int_{-\infty}^{x_1}\omega_{-1}(y_1,t)dy_1,$$ one has
			\begin{align}\notag
				&\|f(\cdot,t)\|_{L^\infty}\le 2c^{-\frac{1}{2}}(1+t)^{-\frac{1}{2}},\ \ \ \|f_t(\cdot, t)\|_{L^\infty}\le4 c^{-\frac{1}{2}}(1+t)^{-\frac{3}{2}}.
			\end{align}
			Multiplying \cref{sys-Eu0}$_2\times(\bar{W}+\bar{\Phi})f$ and integrating the resulting equation over $\mathbb R$ leads to 
			\begin{align}\label{poin5}
				\begin{aligned}
					\frac{\gamma-1}{2}\int_{\mathbb R}(\bar{W}+\bar{\Phi})^2\omega_{-1}
					dx_1
					&=\int_{\mathbb R}\bar{\Psi}_{1t}(\bar{W}+\bar{\Phi})fdx_1+\int_{\mathbb R}(2\mu+\lambda)\p_1\bar{\Psi}_1\p_1\left\{(\bar{W}+\bar{\Phi})\frac{f}{\rhob}\right\}dx_1\\
					&\quad-\int_{\mathbb R}\left(J_{21}^{\#}+\mathcal{Q}_{21}^{\#}-Q_1\right)(\bar{W}+\bar{\Phi})fdx_1:=\sum_{i=1}^{3}I_{\omega}^{i}.
				\end{aligned}
			\end{align}
			We estimate $I_{\omega}^{i}$ as follows:
			\begin{align}\notag
				\begin{aligned}
					I_{\omega}^{1}&=\left(\int_{\mathbb R}\bar{\Psi}_1(\bar{W}+\bar{\Phi})fdx_1\right)_t-\int_{\mathbb R}\bar{\Psi}_1(\bar{W}+\bar{\Phi})_tfdx_1-\int_{\mathbb R}\bar{\Psi}_1(\bar{W}+\bar{\Phi})f_tdx_1:=I_{\omega}^{11}+I_{\omega}^{12}+I_{\omega}^{13}.
				\end{aligned}
			\end{align}
			Firstly, by {the {\it a priori} assumptions} \cref{apa0}, one has
			\begin{align}\notag
				I_{\omega}^{13}\leq C\int_{\R} \abs{f_t}\Big(\bar{\Phi}^2+\bar{\Psi}_1^2+\bar{W}^2\Big)dx_1\le C\etab (1+t)^{-\frac{5}{4}}.
			\end{align}
			By \cref{sys-Eu0}$_{1,4}$, we have
			\begin{align}\notag
				\Big(\bar{W}+\bar{\Phi}\Big)_t=-\gamma\bar{\theta}\px\bar{\Psi}_1+\frac{\kappa}{\rhob}\px^2\bar{W}+J_1^{\#}+J^{\#}_3+\mathcal{Q}^{\#}_3-Q_2.
			\end{align}
			Consequently
			\begin{align}\notag
				\begin{aligned}
					I_\omega^{12}&=-\int_{\mathbb R}\frac{\kappa}{\rhob}\bar{\Psi}_1\px^2\bar{W}fdx_1+\frac{\gamma}{2}\int_{\mathbb R}\px\big(\bar{\Psi}_1^2\big)\bar{\T}fdx_1-\int_{\mathbb R}\Big(J_1^{\#}+J^{\#}_3+\mathcal{Q}^{\#}_3-Q_2\Big)\bar{\Psi}_1fdx_1\\
					:=&J_\omega^{121}+J_\omega^{122}+J_\omega^{123},
				\end{aligned}
			\end{align}
			and
			\begin{align}\notag
				J_\omega^{121}=&\int_{\mathbb R}\frac{\kappa}{\rhob}\px\bar{\Psi}_1\px\bar{W}fdx_1+\int_{\mathbb R}\bar{\Psi}_1\px\bar{W}\px\Big(f\frac{\kappa}{\rhob}\Big)dx_1\leq C\norm{\px\bar{\Psi}_1,\px\bar{W}}_{L^2}^2+\eta\int_{\R}\omega_{-1}\bar{\Psi}_1^2dx_1,
			\end{align}
			where $\eta$ is a sufficiently small constant to be determined later.
			The estimate on $I_{\omega}^{123}$ is similar to \cref{I62,I63,I64,I65}
			\begin{align}\label{sec-n-1}
				I_{\omega}^{123}\leq C\norm{\px(\bar{\Phi},\bar{\Psi},\bar{W})}_{H^1}^2+\eta\int_{\R}\omega_{-1}\Big(\bar{\Phi}^2+\abs{\bar{\Psi}}^2+\bar{W}^2\Big)dx_1+C\delta(1+t)^{-\frac32}.
			\end{align}
			And $I_{\omega}^{122}=-c \int_{\R} \omega_{-1} \Psi_1^2 dx_1$,
			then we arrive at
			\begin{align}\notag
				I_{\omega}^{12}\le C\|\px(\bar{\Phi},\bar{\Psi},\bar{W})\|_{H^1}^2
				+C\delta (1+t)^{-\frac{3}{2}}-c
				\int_{\mathbb R}\bar{\Psi}_1^2
				\omega_{-1}dx_1+C\eta\int_{\R}\omega_{-1}\Big(\bar{\Phi}^2+\bar{\Psi}_2^2+\bar{\Psi}_3^2+\bar{W}^2\Big)dx_1.
			\end{align}
			Also, we can estimate $I_{\omega}^3$ by the same line as \cref{sec-n-1}. As for $I_{\omega}^2$, one has
			\begin{align}\notag
				I_{\omega}^2\leq C\norm{\px(\bar{\Phi},\bar{\Psi},\bar{W})}_{L^2}^2+\eta\int_{\R}\omega_{-1}\Big(\bar{\Phi}^2+\bar{\Psi}_1^2+\bar{W}^2\Big)dx_1.
			\end{align}
			Finally, collecting \cref{poin5} and estimates on $I_{\omega}^{1},I_{\omega}^{2},I_{\omega}^{3}$, one has
			\begin{align}\label{sec-n-3-1}
				&\frac{\gamma-1}{2}\int_{\mathbb R}(\bar{W}+\bar{\Phi})^2\omega_{-1}
				dx_1+c
				\int_{\mathbb R}\bar{\Psi}_1^2
				\omega_{-1}dx_1-I_{\omega}^{11}  \notag\\
				\leq &C\|\px(\bar{\Phi},\bar{\Psi},\bar{W})\|_{H^1}^2+C\eta\int_{\R}\omega_{-1}\Big(\bar{\Phi}^2+\bar{\Psi}_2^2+\bar{\Psi}_3^2+\bar{W}^2\Big)dx_1
				+C(1+t)^{-\frac{5}{4}}.
			\end{align}
			
			Note that we still need to estimate $\int_{\R}(\bar{W}-\bar{\Phi})^2\omega_{-1} dx_1$, for which we will apply \cref{heat0}. Firstly, we use \cref{sys-Eu0} to obtain the cancellation as 
			\begin{align}\notag
				\Big(\bar{W}-(\gamma-1)\bar{\Phi}\Big)_t=\frac{\kappa}{\rhob}\px^2\bar{W}+J^{\#}_3+\mathcal{Q}^{\#}_3-Q_2-(\gamma-1)J_1^{\#}.
			\end{align}
			Next, to apply \cref{heat0}, we define $h=\bar{W}-(\gamma-1)\bar{\Phi}$ to have
			\begin{align}\label{2025-11-7-2}
				\langle{h_t,hg^2}\rangle_{H^{-1}\times H^1}=\int_{\R}\frac{\kappa}{\rhob}\px^2\bar{W}hg^2 dx_1+\int_{\R}\big(J_3^{\#}+\mathcal{Q}_3^{\#}-Q_2-(\gamma-1)J_1^{\#}\big)hg^2dx_1:=I_{h}^{1}+I_{h}^2.
			\end{align}
			We only show the estimations of $I_{h}^1$, since $I_{h}^2$ is similar to \cref{sec-n-1}.

			Using an integration by part about $x_1$, we have
			\begin{align}\label{poin10}
				\begin{aligned}
					\int_{\mathbb R}\frac{\kappa}{\rhob}\px^2\bar{W}hg^2dx_1&=-\int_{\mathbb R}\px\left(\frac{\kappa}{\rhob}\right)\px\bar{W}hg^2dx_1-\int_{\mathbb R}\frac{\kappa}{\rhob}\px\bar{W}\px hg^2dx_1-2\int_{\mathbb R}\frac{\kappa}{\rhob}\px\bar{W}hg\omega_{-\frac12} dx_1\\
					&\le C(\delta+\eta)\int_{\mathbb R}(\bar{\Phi}^2+\bar{W}^2)\omega_{-1}dx_1+C\|\px\bar{W}\|^2+C\|\px\bar{\Phi}\|^2.
				\end{aligned}
			\end{align}
			By choosing a suitably small $\eta>0$, one has
			\begin{align}\label{poin13}
				\begin{aligned}
					\langle{h_t,hg^2}\rangle_{H^{-1}\times H^1}\le &C(\delta+\eta)\int_{\mathbb R}(\bar{\Phi}^2+\abs{\bar{\Psi}}^2+\bar{W}^2)\omega_{-1}dx_1+C\|\px(\bar{\Phi},\bar{\Psi},\bar{W})\|_{H^1}^2+C \delta(1+t)^{-\frac{3}{2}}.
				\end{aligned}
			\end{align}
			It follows from \cref{heat0} and \cref{poin13} that 
			\begin{align}\label{poin14}
				\begin{aligned}
					\int_0^t&\int_{\mathbb R}(\bar{W}-(\gamma-1)\bar{\Phi})^2\omega_{-1}dx_1d\tau\le C+C(\delta+\eta)\int_0^t\int_{\mathbb R}(\bar{\Phi}^2+\abs{\bar{\Psi}}^2+\bar{W}^2)\omega_{-1}dx_1+C\int_0^t\|\px(\bar{\Phi},\bar{\Psi},\bar{W})\|_{H^1}^2d\tau.
				\end{aligned}
			\end{align}
			By \cref{sys-Eu0}$_{3,4}$ and \cref{lem5}, using the same method as \cref{2025-11-7-2,poin10,poin13,poin14}, one has
			\begin{align}\label{poin15}
				\begin{aligned}
					\int_0^t&\int_{\mathbb R}(\bar{\Psi}_2+\bar{\Psi}_3)^2\omega_{-1}dx_1d\tau\le C+C(\delta+\eta)\int_0^t\int_{\mathbb R}(\bar{\Phi}^2+\bar{W}^2+\bar{\Psi}_1^2)\omega_{-1}dx_1+C\int_0^t\|\px(\bar{\Phi},\bar{\Psi},\bar{W})\|_{H^1}^2d\tau.
				\end{aligned}
			\end{align}
			Combining \cref{sec-n-3-1}, \cref{poin14} and \cref{poin15} and taking $\eta$ small enough, we have completed the proof of \cref{lem5}.
		\end{proof}
		
		{\bf The decay rate: proof of \cref{mt0}.}
		
		\begin{proof}
			Integrating \cref{111} with $[0,t]$ and employing \cref{lem5}, we have
			\begin{align}\notag
				\sum_{i=0}^2E_i +\int_0^t\sum_{i=0}^2(K_i+G_{i})d\tau\le C(\e^2+\delta)\ln(2+t).
			\end{align}
			Multiplying \cref{222} by $(1+t)$ and integrating the resulting inequality on $[0,t]$, one has
			\begin{align}\label{b111}
				\sum_{i=1}^2{E}_i\le C(\e^2+\bar\delta)(1+t)^{-1}
				\ln(2+t),\ \ \int_0^t(1+\tau)\sum_{i=1}^2(K_i+G_{i})d\tau\le C(\e^2+\delta)
				\ln(2+t).
			\end{align}
			Multiplying \cref{333} by $(1+t)^2$ and integrating the resulting inequality into $[0,t]$, one has
			\begin{align}\notag
				(1+t)^2{E}_2+\int_0^t(1+\tau)^2({K}_2+G_{2})d\tau
				\le C(\e^2+\delta)
				\ln(2+t).
			\end{align}
			It then follows
			\begin{align}\label{b110}
				{E}_2\le C(\e^2+\delta)(1+t)^{-2}\ln(2+t).
			\end{align}
			Then by \cref{b111} and \cref{b110}, it holds that
			\begin{align}\notag
				\|(\phim,\psim,\mathring{w})\|_{L^\infty}\le C(\e+\delta^{\frac{1}{2}})(1+t)^{-\frac{3}{4}}\ln^{\frac{1}{2}}(2+t).
			\end{align}
			The proof of the decay rate of the non-zero mode is the same as that of \cref{lem99} and \cref{lem-non-zeromode-decay}, so we omit it.
			Then the proof of \cref{Thm-ape-0} is completed.
		\end{proof}

		\medskip
		\noindent {\bf Acknowledgment:}\,
		The work of Feimin Huang was partially supported by National
Key R$\&$D Program of China No. 2021YFA1000800, and National Natural Sciences Foundation of China (NSFC) No. 12288201. The research of Renjun Duan was partially supported by the General Research Fund (Project No.~14303523) from RGC of Hong Kong and also by the grant from the National Natural Science Foundation of China (Project No.~12425109). Lingda Xu was supported by the Research Center for Nonlinear Analysis of the Hong Kong Polytechnic University.

		\medskip
		\noindent{\bf Data availability:} The manuscript contains no associated data.

		\medskip
		\noindent{\bf Conflict of Interest:} The authors declare that they have no conflict of interest.

		
	\end{document}